\documentclass{birkjour}
\usepackage{amsmath}
\usepackage{amssymb}
\usepackage{array}
\usepackage{booktabs}
\usepackage{multirow}
\usepackage{float}

\usepackage{hyperref}

 \newtheorem{teor}{Theorem}[section]
 \newtheorem{corollary}[teor]{Corollary}
 \newtheorem{lemma}[teor]{Lemma}
 \newtheorem{prop}[teor]{Proposition}
 \theoremstyle{definition}
 \newtheorem{defi}[teor]{Definition}
 \theoremstyle{remark}
 \newtheorem{remark}[teor]{Remark}
 \newtheorem*{exa}{Example}

 \numberwithin{equation}{section}

\usepackage{graphicx}
\usepackage{calc}

\newlength{\depthofsumsign}
\newcommand{\hector}[1][1.65]{
    \mathop{%
        \raisebox
            {-#1\depthofsumsign+1\depthofsumsign}
            {\scalebox
                {#1}
                {$\chi$}%
            }
    }
}

\newcommand{\chisota}[1][3pt]{%
  \mathrel{\raisebox{#1}{$\hector$}}%
}

\newcommand{\chisotazo}[1][3pt]{%
  \mathrel{\raisebox{#1}{$\hector$}}%
}

\begin{document}

%
%
%
%
%
%
%
%
%

\title[Generalized Polarization Modules]{Generalized Polarization Modules}

\author[H\'ector J. Blandin N.]{H\'ector J. Blandin N.}

\address{%
Laboratoire de Combinatoire et d'informatique Math\'ematique (LaCIM)\\
Universit\'e du Qu\'ebec \`{a} Montr\'eal (UQ\`AM)\\
CP 8888, Succ. Centre-ville\\ Montr\'eal (Qu\'ebec) H3C 3P8\\
Canada}

\email{hectorblandin@gmail.com}

\thanks{This work was completed with the support of our
\TeX-pert.}


\date{January 1, 2004}

\begin{abstract}
This work enrols the research line of M. Haiman on the Operator Theorem (the former Operator Conjecture). Given a $\mathfrak{S}_n$-stable family $F$ of homogeneous polynomials in the variables $x_{ij}$ with $1\leq i\leq \ell$ and $1\leq j\leq n$. We define the polarization module generated by the family $F$, as the smallest vector space closed under taking partial derivatives and closed under the action of polarization operators that contains $F$. These spaces are representations of the direct product $\mathfrak{S}_n\times{GL}_{\ell}(\mathbb{C})$. We compute the graded Frobenius characteristic of these modules. We use some basic tools to study these spaces and give some in-depth calculations of low degree examples of a family or a single symmetric polynomial.
\end{abstract}

\maketitle
\tableofcontents

\noindent This paper is a full version of \cite{Blandin}. Here we present the proofs of all theorems announced in \cite{Blandin}. This work is inspired by a theorem of M. Haiman (the former Operator Conjecture, see \cite{Haiman1} Conjecture 5.1.1. and \cite{Haiman3}). This Haiman's theorem states that the smallest subspace of $\mathbb{C}[x_1,\ldots,x_n]$ closed under taking partial derivatives ${\frac{\partial }{\partial x_i}}$, closed under the action of generalized polarization operators $E_{p}=\sum_{j=1}^{n}y_{j}\frac{\partial^p }{\partial x_j^p}$, and that contains the Vandermonde determinant  
$$\Delta_{n}({\mathbf x})~:=~\prod_{1\leq i<j\leq n}(x_{i}~-~x_{j}),$$ coincides with the space $\mathcal{D}_n$ of diagonally harmonic polynomials of $\mathfrak{S}_n$ (for more details see \cite{FBergeron2,FBergeronEminem,Haiman1,Haiman3}). 

We generalise the context of the Operator Theorem to the context of polynomials in the matrix variables $X=(x_{ij})$, with $1\leq i\leq \ell$ and $1\leq j\leq n$. The diagonal action of $\mathfrak{S}_n$ on these polynomials is defined by permuting the columns of $X$. We say that a family $F$ of homogeneous polynomials (in $X=(x_{ij})$) is $\mathfrak{S}_n$-stable if $F$ is closed under the diagonal action of $\mathfrak{S}_n$. Given any such family $F$, we define the polarization module $\mathcal{M}_F$ generated by the family $F$ as the smallest vector space closed under taking partial derivatives $\frac{\partial }{\partial x_{ij}}$, closed under the action of generalized polarization operators $E_{i,k}^{(p)}=\sum_{j=1}^{n}x_{ij}\frac{\partial^{p} }{\partial x_{kj}^{p}}$, that contains $F$. The diagonal action of $\mathfrak{S}_n$ makes $\mathcal{M}_F$ into an $\mathfrak{S}_n$-module. The closure by the action of polarization operators $E_{i,k}^{(p)}$ is equivalent to the closure by the action $x_{ij}\longmapsto \sum_{k=1}^{\ell}m_{ik}x_{kj}$ where $M=(m_{ij})\in{GL}_{\ell}(\mathbb{C})$ (see \cite{ProcesiKraft,Procesi}). Then, with this action $\mathcal{M}_F$ is also a polynomial representation of ${GL}_{\ell}(\mathbb{C})$. The actions of $\mathfrak{S}_n$ and ${GL}_{\ell}(\mathbb{C})$ on $\mathcal{M}_F$ commute and this implies that $\mathcal{M}_{F}$ is a representation of the direct product $\mathfrak{S}_n\times{GL}_{\ell}(\mathbb{C})$. In particular, when the family $F$ is the orbit of a single homogeneous polynomial $f$, that is, $F=\{\sigma\cdot f \ \vert \ \sigma\in\mathfrak{S}_n\}$ we denote the polarization module generated by $F$ simply as $\mathcal{M}_f$. Also, we call $\mathcal{M}_f$ the polarization module generated by $f$. This construction correspond to certain important spaces in algebraic combinatorics (see \cite{FBergeron,FBergeronEminem}) and algebraic geometry (see \cite{Geramita}).

The goal of this paper is to study the decomposition into irreducible submodules, under the action of $\mathfrak{S}_n\times{GL}_{\ell}(\mathbb{C})$, of polarization modules $\mathcal{M}_F$. To do this we compute explicitly the graded Frobenius characteristic of $\mathcal{M}_F$ in the form
\begin{equation}\label{Formula1Introduccion}
\mathcal{M}_{F}({\bold q},{\bold w})=\sum_{\lambda\vdash n}
\,\sum_{\vert\mu\vert\leq d}b_{\lambda,\mu}s_{\mu}({\bold q})s_{\lambda}({\bold w}) 
\end{equation}
where $b_{\lambda,\mu}\in\mathbb{N}$, ${\bold q}=q_1,q_2,\ldots,q_{\ell}$, ${\bold w}=w_1,w_2,\ldots$, and $d$ is the maximal degree of polynomials in $F$. Here the Schur functions $s_{\mu}({\bold q})$ encode the irreducibles for $GL_{\ell}(\mathbb{C})$ while $s_{\lambda}({\bold w})$ encode the irreducibles for $\mathfrak{S}_n$. The coefficients $b_{\lambda,\mu}$ are the multiplicities of irreducible submodules under the action of $\mathfrak{S}_n\times{GL}_{\ell}(\mathbb{C})$ (see \cite{FBergeron} for more details). 

In this paper we prove the decomposition into irreducible submodules of the polarization modules generated by each of the polynomials $p_1^d$, $p_d$, and $e_d$ for any $d\geq 1$ as announced in \cite{Blandin}. We construct an explicit linear basis of each module and then we compute the graded Frobenius characteristic of $\mathcal{M}_{p_1^d}$, $\mathcal{M}_{p_d}$ and $\mathcal{M}_{e_d}$. 

Obviously, we have $\mathcal{M}_f\cong\mathcal{M}_{k\cdot f}$ for every scalar $k$. In order to classify, up to isomorphism, polarization modules generated by a given homogeneous symmetric polynomial of any degree $d$, we identify any non zero homogeneous symmetric polynomial $f$ of degree $d$, written in the monomial basis as $f~=~\sum_{\lambda\vdash d}c_{\lambda}m_{\lambda}$, with a point in the real projective space $\mathbb{RP}^{p(d)-1}$, where $p(d)$ is the number of integer partitions of $d$. The homogeneous coordinates of the corresponding  point are ordered according to the following order on integer partitions of $d$:
$(d)$, $(d-1,1)$, $(d-2,2)$, $(d-2,1,1)$,$\ldots$, $(1,1,\ldots,1)$. In degree 2, we show that there are two types of polarization modules 
$\mathcal{M}_{p_1^2}$ and $\mathcal{M}_{p_2}$ up to isomorphism. More precisely, if $[a:b]\in\mathbb{RP}^1$ and $f=a\cdot m_{2}+b\cdot m_{11}$ then $\mathcal{M}_{f}\cong\mathcal{M}_{p_1^2}$ when $[a:b]=[1:2]$, while $[a:b]\neq [1:2]$ implies $\mathcal{M}_{f}\cong\mathcal{M}_{p_2}$.
Notice that the last statement is independent on the number of variables $n$. The situation when the degree is 3 is more complicated, in this case, we will need to introduce the notion of $n$-exception to completely classify these polarization modules. Let $n\geq 3$, a point $[a:b:c]\in\mathbb{RP}^{2}$ is a $n$-exception if and only if\ $[a:b:c]\neq[1:3:6]$ and $6a(2b+(n-2)c)=4(n-1)b^2$. When $n=2$, $[a:b:c]$ is a 2-exception if and only if $b=0$ or $b=3a$. There are three types of polarization modules generated by a single polynomial of the form $f=a\cdot m_3+b\cdot m_{21}+c\cdot m_{111}$. If $[a:b:c]=[1:3:6]$ then $\mathcal{M}_f\cong\mathcal{M}_{p_1^3}$; if $[a:b:c]$ is a $n$-exception then $\mathcal{M}_f\cong\mathcal{M}_{p_3}$; otherwise, 
$\mathcal{M}_f\cong\mathcal{M}_{h_3}$. These results are valid for any $\ell$ (the number of sets of $n$ variables). Also, we will see that $n$-exceptions appear in any degree $d\geq 3$.  Characterizing $n$-exceptions for degrees higher than 3 is a problem for the future. A Theorem for the existence of $n$-exceptions in degree greater or equal than 4 will be discussed in (see \cite{BlandinExceptions}). 


\section{Preliminaries and some notations}

\noindent Let $X$ be a $\ell\times n$ matrix of commuting and independent variables $x_{ij}$, in symbols:
\begin{equation}
X:=\left(\begin{array}{cccc} x_{11}&x_{12} &\dots & x_{1n}\\
x_{21}&x_{22} &\dots & x_{2n}\\
\vdots& \vdots & \ddots & \vdots \\
x_{\ell 1}&x_{\ell 2} &\dots & x_{\ell n}\\\end{array}\right).
\end{equation}
For any fixed integer $i$, we call the $i^{th}$-row of $X$, denoted by ${\bold x}_{i}:=(x_{i1},\ldots,x_{in})$ the $i^{th}$ set of $n$ variables. For any $j$, $X_{j}$ denotes the $j^{th}$-column of $X$. The last convention is adopted for any $\ell\times n$ matrix of exponents $A$. Then monomials are defined as follows:
\begin{equation}
X^{A}:=x_{11}^{a_{11}}\cdots x_{ij}^{a_{ij}}\cdots x_{\ell{n}}^{a_{\ell{n}}}.
\end{equation}
these monomials form a linear basis of the $\mathbb{K}$-vector space $\mathcal{R}_{n}^{(\ell)}:=\mathbb{K}[X]$ of polynomials in $\ell$ sets of $n$ variables. The (vector) \textbf{degree} $\deg\big(X^{A}\big)$ lies in $\mathbb{N}^{\ell}$ and is given by
$$\deg\left(X^{A}\right):=\left(\sum_{j=1}^{n}a_{1j},\dots,\sum_{j=1}^{n}a_{\ell j}\right).$$
For each ${\bold d}\in\mathbb{N}^{\ell}$, we denote by $\mathcal{R}_{n,{\mathbf d}}^{(\ell)}$ the span of degree ${\mathbf d}$ monomials in $\mathcal{R}_{n}^{(\ell)}$. Then $\mathcal{R}_{n}^{(\ell)}$ is a 
$\mathbb{N}^{\ell}$-graded vector space, that is, 
$$\mathcal{R}_{n}^{(\ell)}=\bigoplus_{{\mathbf d}\in\mathbb{N}^{\ell}}\mathcal{R}_{n,{\mathbf d}}^{(\ell)}.$$
We write any polynomial $f(X)\in\mathcal{R}_{n}^{(\ell)}$ in the form \ $f(X)=\sum_{A\in\mathbb{N}^{\ell\times n}}f_{A}X^{A}$. A polynomial $f(X)\in\mathcal{R}_{n}^{(\ell)}$ is said to be \textbf{homogeneous} if it satisfies $f(QX)={\bold q}^{\bold d}f(X)$, where $Q$ is the diagonal matrix
\begin{equation*}
Q=\left(\begin{array}{ccc}
q_{1}&\cdots&0\\
\vdots&\ddots&\vdots\\
0&\cdots& q_{\ell}
\end{array}\right),\ \ {\bold q}:=(q_1,\ldots,q_{\ell}),\ \ \ {\bold q}^{\bold d}:=q_{1}^{d_1}\cdots q_{\ell}^{d_{\ell}},
\end{equation*}

\noindent In this work we consider homogeneous subspaces $\mathcal{V}$ of $\mathcal{R}_{n}^{(\ell)}$, that is, a subspace $\mathcal{V}$ that affords a basis of homogeneous polynomials. The degree ${\bold d}$ homogeneous component of $\mathcal{V}$ is denoted by $\mathcal{V}_{\bold d}$. So, $\mathcal{V}_{\bold d}:=\mathcal{V}\cap\mathcal{R}_{n,{\bold d}}^{(\ell)}$, and the \textbf{Hilbert series} of $\mathcal{V}$ is defined as
$$\mathcal{V}({\bold q}):=\sum_{{\bold d}\in\mathbb{N}^{\ell}}\dim(\mathcal{V}_{\bold d}){\bold q}^{\bold d},$$ 
On $\mathcal{V}$ we consider two linear group actions:
\begin{enumerate}
\item The (left) diagonal action of $\mathfrak{S}_n$, \ \ $\sigma\cdot X^A:=x_{1\sigma(1)}^{a_{11}}\cdots x_{i\sigma(j)}^{a_{ij}} \cdots
x_{{\ell}\sigma(n)}^{a_{{\ell}n}}$,\ $\forall\sigma\in\frak{S}_n$.
\item The (right) action of $GL_{\ell}(\mathbb{C})$, \ \ $X^A\cdot M :=(MX)^A$,\ $\forall M\in GL_{\ell}(\mathbb{C})$, that is, $x_{ij}\longmapsto \sum_{k=1}^{\ell}m_{ik}x_{kj}$, for every matrix $M=(m_{ij})\in GL_{\ell}(\mathbb{C})$.
\end{enumerate}
It's not hard to show that these two group actions on $\mathcal{V}$ commute, and then we can consider $\mathcal{V}$ as a representation of the direct product $\mathfrak{S}_n\times{GL}_{\ell}(\mathbb{C})$, with the (left) action, 
$$(\sigma,M)\cdot X^{A}:=\sigma\cdot (M^{-1}X)^{A}.$$ 
Then we have a direct sum decomposition of the form (see \cite{BruceSagan,FultonHarris,Procesi}) 
\begin{equation}\label{DecompIrred}
\mathcal{V}=\bigoplus_{\lambda\vdash n}\bigoplus_{\mu} b_{\lambda,\mu}\,\mathcal{W}_{\mu} \otimes \mathcal{S}^{\lambda},
\end{equation}
where $b_{\lambda,\mu}\in\mathbb{N}$, the\ $\mathcal{S}^{\lambda}$ are irreducible $\frak{S}_n$-modules and the $\mathcal{W}_{\mu}$ are irreducible polynomial representations of  $GL_{\ell}(\mathbb{C})$. The \textbf{graded Frobenius characteristic} of $\mathcal{V}$ is defined as follows
\begin{equation}
\mathcal{V}({\bold q},{\bold w}):=\sum_{{\bold d}\in\mathbb{N}^{\ell}}\left(
\frac{1}{n!}\sum_{\sigma\in\frak{S}_n}{\chisotazo}_{\mathcal{V}_{\bold d}}(\sigma)\,
\,p_{_{\lambda(\sigma)}}({\bold w})\right){\bold q}^{\bold d},
\end{equation}
where $\chisotazo_{\mathcal{V}_{\bold d}}$ is the $\mathfrak{S}_n$-character of $\mathcal{V}_{\bold d}$ and $p_{\lambda(\sigma)}({\bold w}):=p_{1}({\bold w})^{c_1(\sigma)}\cdots{p_{n}({\bold w})^{c_{n}(\sigma)}}$. 

One can show that the graded Frobenius characteristic of $\mathcal{V}$ has the following form:

\begin{equation}\label{FormaGeneralDelaCaracteristicaDeFrobenius}
\mathcal{V}({\bold q},{\bold w})=\sum_{\lambda\vdash n}\sum_{\mu}b_{\lambda,\mu}s_{\mu}({\bold q})s_{\lambda}({\bold w}),
\end{equation}

\noindent where $b_{\lambda,\mu}$ are the multiplicities in formula (\ref{DecompIrred}). A theorem of F. Bergeron shows that the multiplicities $b_{\lambda,\mu}$ in formula (\ref{FormaGeneralDelaCaracteristicaDeFrobenius}) do not depend on $\ell$ and that $\ell(\mu)\leq n$ (see \cite{FBergeron} for more details). The Schur functions $s_{\mu}({\bold q})$ encode the irreducible polynomial representations of $GL_{\ell}(\mathbb{C})$ in $\mathcal{V}$ of type $\mu$, and the Schur functions $s_{\lambda}({\bold w})$ encodes the irreducible $\mathfrak{S}_n$-modules in $V$ of type $\lambda$. Recall that the Hilbert Series of $\mathcal{V}$ is obtained by replacing each Schur function $s_{\lambda}({\bold w})$ by the number $f^{\lambda}$ (the number of standard Young tableaux of shape 
$\lambda$) in the formula (\ref{FormaGeneralDelaCaracteristicaDeFrobenius}). Notice that the Hilbert series of $\mathcal{V}$ coincides with the character of $\mathcal{V}$ as a representation of $GL_{\ell}(\mathbb{C})$ with the (right side) action given by $f(X)\longmapsto f(MX)$, for all $M\in{GL}_{\ell}(\mathbb{C})$ and for all $f(X)\in\mathcal{V}$. \\

 One can also understand the Frobenius series with the following approach (see, \cite{Haiman3} page 387). Any homogeneous subspace $\mathcal{V}$ of $\mathcal{R}_n^{(\ell)}$ closed by polarizations is a $\mathcal{R}_{n}^{(\ell)}$-module (regarding $\mathcal{R}_{n}^{(\ell)}$ as a $\mathbb{C}$-algebra) with the action of $x_{ij}$ of $f(X)$, given by 
\[ x_{ij}\cdot f(X):=E_{i,j}(f),\ \ E_{i,j}:=\sum_{r=1}^{n}x_{ir}\frac{\partial }{\partial x_{jr}}. \]
On $\mathcal{V}$ we have the diagonal action of $\mathfrak{S}_n$ that commutes with this action. Then $\mathcal{V}$ has a canonical decomposition 
  \[  \mathcal{V}=\bigoplus_{\lambda\vdash n} \mathcal{V}_{\lambda} \otimes_{\mathbb{C}} \mathcal{S}^{\lambda},\ \ \  \mathcal{V}_{\lambda}:=\mathrm{Hom}_{\mathfrak{S}_n}\big(\mathcal{V},\mathcal{S}^{\lambda}\big), \]
in which for each partition $\lambda$, $\mathcal{S}^{\lambda}$ is an irreducible representation of $\mathfrak{S}_n$ and $\mathcal{V}_{\lambda}$ is a $\mathbb{N}^{\ell}$-graded ${\mathbb{C}[{\bold x}_1,\ldots,{\bold x}_{\ell}]}^{\mathfrak{S}_n}$-module. Notice that
each space $\mathcal{V}_{\lambda}$ is a $GL_{\ell}(\mathbb{C})$-module with character given by its Hilbert series $\mathcal{V}_{\lambda}({\bold q})$. Then we can consider the graded Frobenius characteristic in the following manner
\[ \mathcal{V}({\bold q},{\bold w})=\sum_{\lambda\vdash n}\mathcal{V}_{\lambda}({\bold q})s_{\lambda}({\bold w}). \]
\noindent We will describe linear bases of generalized polarization modules generated by a single homogeneous symmetric polynomial. In several cases, our main tool for proving these results are diagonally symmetric polynomials (see \cite{MercedesRosas}). We adopt the notations of \cite{MercedesRosas}. Recall that polynomial $f$ invariant under the diagonal action of $\mathfrak{S}_n$  is said to be \textbf{a diagonally symmetric polynomial}. For any $i$ such that $1\leq i\leq \ell$, we set 
$p_{d}({\bold x}_{i}):=x_{i1}^{d}+\cdots+x_{in}^{d}$,\ $e_{1}({\bold x}_{i}):=x_{i1}+\cdots+x_{in}$.  
In this work we will use the following diagonally symmetric polynomials:
\begin{enumerate}
\item For each ${\bold d}=(d_1,\ldots,d_{\ell})\in\mathbb{N}^{\ell}$ we set ${\bold d}!:=d_{1}!\cdots d_{\ell}!$, $\vert{\bold d}\vert:=d_1+\cdots+d_{\ell}$ and
\[ X_{j}^{\bold d}:=x_{1j}^{d_1}\cdots x_{{\ell}j}^{d_{\ell}} \]
\[ e_{1}^{\bold d}(X):=e_{1}({\bold x}_1)^{d_1}\cdots e_{1}({\bold x}_{\ell})^{d_{\ell}}. \]
\item The \textbf{diagonal power sum} polynomials are
\begin{equation*}
p_{\bold d}(X):=\sum_{j=1}^{n}X_{j}^{\bold d}=\sum_{j=1}^{n}x_{1j}^{d_1}\cdots x_{\ell,j}^{d_{\ell}},
\end{equation*}
and the generating series of $p_{\bold d}(X)$ is the following:
\begin{equation}\label{SerieGeneratrizSumasPotenciasDiagonales}
\sum_{{\bold d}\in\mathbb{N}^{\ell}-\{\bold 0\}}p_{\bold d}\left(X\right)\frac{\vert{\bold d}\vert!\,{\bold t}^{\bold d}}{{\bold d}!\vert{\bold d}\vert}=\log\left(\,\prod_{j= 1}^{n}\frac{1}{1-\sum_{i=1}^{\ell}t_{j}x_{ij}}\right).
\end{equation}
\item The \textbf{MacMahom symmetric elementary} polynomials $e_{\bold d}(X)$ have the generating series:
\begin{equation}\label{SerieGeneratrizMultisimetricasElementales}
\sum_{{\bold d}\in\mathbb{N}^{\ell}}
e_{\bold d}\left(X\right){\bold t}^{\bold d}=\prod_{j= 1}^{n}
\left(1+\sum_{i=1}^{\ell}t_{j}x_{ij}\right).
\end{equation}
where ${\bold t}^{\bold d}:=t_{1}^{d_1}\cdots t_{\ell}^{d_{\ell}}$.
\end{enumerate}
One can show that an explicit form for $e_{\bold d}(X)$ is
\begin{equation*}
e_{\bold d}(X)=\sum_{B\subseteq[n]}\,\sum_{\{g:B\rightarrow{[\ell]}\ : \ {\vert{g^{-1}}(i)\vert=d_{i}}\}}\,\prod_{b\in B}x_{g(b),b}.
\end{equation*}
where $[n]:=\{1,2,\ldots,n\}$ and $[\ell]:=\{1,2,\ldots,\ell\}$. For each $B\subseteq [n]$ the second sum is taken over the set of all maps from $B$ to $[\ell]$.


\section{Definitions and discussions}\label{Definitions}

\noindent We denote the partial derivative operator on $\mathcal{R}_{n}^{(\ell)}$ by $\displaystyle{\partial_{ij}:=\frac{\partial \ \ }{\partial x_{ij}}}$ and for any $p\geq 1$ we set $\displaystyle{\partial_{ij}^p:=\frac{\partial^p \ \ }{\partial x_{ij}^p}}$. We use the \textbf{generalized polarization operators $E_{i,k}^{(p)}$} (see, \cite{HWeyl,Hunziker,Procesi}) given by 
$$E_{i,k}^{(p)}:=\sum_{j=1}^{n}x_{ij}\partial_{kj}^{p}.$$ 
For $p=1$ we simply write $E_{i,k}:=E_{i,k}^{(1)}$ (see \cite{Procesi}, p.40,41.). For any $r\geq 0$ we set $E_{i,k}^{0}$ to be the identity operator on $\mathcal{R}_{n}^{(\ell)}$, and $E_{i,k}^{\,r}:=E_{ik}\circ\cdots\circ E_{i,k}$ to be the composition of $E_{i,k}$ with itself $r$ times. In particular, when $p=1$ this polarization operators satisfy the identity (see \cite{Procesi})
\begin{equation}\label{201601151427}
E_{i,j}E_{h,k}-E_{h,k}E_{i,j}=\delta_{j,h}E_{i,k}-\delta_{i,k}E_{h,j}. 
\end{equation}
so, when $i\neq k$ and $j\neq h$ they commute. For instance, $E_{2,1}E_{3,1}=E_{3,1}E_{2,1}$.
\begin{lemma}[see C. Procesi \cite{Procesi}, p.43]
\label{PolarizacionesEikInyectivas}
Let $g({\bold x}_1)$ be an homogeneous polynomial of degree $d$ in the variables ${\bold x}_1=x_{11},\ldots,x_{1n}$. The polarization operators $E_{i,1}$ and $E_{1,i}$ satisfy the identity \ $E_{1,i}E_{i,1}\big(g({\bold x}_1)\big)=d\cdot g({\bold x}_1)$ if $i>1$. Furthermore, $E_{i,1}\big(g({\bold x}_1)\big)$ is an homogeneous polynomial.
\end{lemma}
\begin{proof}
We start by applying the operator $E_{i,1}$ to $g({\bold x}_1)$:
\[\left(\sum_{j=1}^{n}x_{ij}\frac{\partial }{\partial x_{1j}}\right)
g({\bold x}_1)=\sum_{j=1}^{n}x_{ij}\frac{\partial g}{\partial x_{1j}}({\bold x}_1),\]
then, we apply to both sides $E_{1,i}$, and since $i>1$ we get
\begin{align*}
&\left(\sum_{r=1}^{n}x_{1r}\frac{\partial }{\partial x_{ir}}\right)
\left(\sum_{j=1}^{n}x_{ij}\frac{\partial g}{\partial x_{1j}}({\bold x}_1)\right)=\sum_{r=1}^{n}x_{1r}\left(\sum_{j=1}^{n}\delta_{r,j}\frac{\partial g}{\partial x_{1j}}({\bold x}_1)\right)\\
&=\sum_{r=1}^{n}x_{1r}\frac{\partial g}{\partial x_{1r}}({\bold x}_1)
=d\cdot g({\bold x}_1).
\end{align*}
In the last step, we use the Euler's identity (see \cite{Humphreys}) that asserts for any homogeneous polynomial $g\in\mathbb{K}[x_{11},\ldots,x_{1n}]$ of degree $d$ we have
\[ \displaystyle{\sum_{r=1}^{n}x_{1r}\frac{\partial g}{\partial x_{1r}}({\bold x}_1)=d\cdot g({\bold x}_1)}. \]
\end{proof} 
\begin{remark}
We can check that if $f=f(x_{i1},\ldots,x_{in})$ is an homogeneous polynomial of total degree $d$ in the variables $x_{i1},\ldots,x_{in}$ then
\[ E_{k,i}^{d}(f)=d!\cdot f(x_{k1},\ldots,x_{kn}).  \]
\end{remark}

\subsection{Polarization and Restitution operators}

\noindent Let ${\bold d}\in\mathbb{N}^{\ell}$. We define the ${\bold d}$-\textbf{polarization} operator by 
$$E^{{\bold d}}:=\frac{d_{1}!}{(d_1+\cdots+d_{\ell})!}\,E_{{\ell},1}^{d_{\ell}}\circ\dots\circ E_{2,1}^{d_{2}},$$  
and the ${\bold d}$-\textbf{restitution} operator by 
$$E_{{\bold d}}:=\frac{1}{d_{2}!\cdots d_{\ell}!}\,E_{1,2}^{d_2}\circ\dots\circ E_{1,{\ell}}^{d_{\ell}}.$$ 
To study the effect of iterated polarizations operators we have to consider the following classical result (see 
\cite{Hunziker,ProcesiKraft,Procesi}). We set $\mathcal{R}_{n}:=\mathcal{R}_{n}^{(1)}=\mathbb{C}[x_{11},\ldots,x_{1n}]$. For any homogeneous polynomial 
$f\in\mathcal{R}_{n}$ of degree $d$ (in the variables ${\bold x}_1:=x_{11},\ldots,x_{1n}$) we have
\[ f\left(\sum_{i=1}^{\ell}t_{i}x_{i1},\ldots,\sum_{i=1}^{\ell}t_{i}x_{in}\right)
=\sum_{k_1+k_2+\cdots+k_{\ell}=d}
\,\frac{t_1^{k_1}t_2^{k_2}\cdots\,t_{\ell}^{k_{\ell}}}{k_2!\cdots\,k_{\ell}!}\,\cdot
E_{\ell,1}^{k_{\ell}}\cdots\,E_{2,1}^{k_{2}}
\big(f({\bold x}_1)\big). \]
Clearly we can write the last identity simply as:
\begin{equation}\label{FormulillaPolarizato}
f\left(\sum_{i=1}^{\ell}t_{i}{\bold x}_i\right)
=\sum_{\vert{\bold d}\vert=d}E^{\bold d}\big(f({\bold x}_1)\big)
\,{d}!\frac{{\bold t}^{\bold d}}{{\bold d}!}.
\end{equation}
The above formula will help us to prove Lemma \ref{201504211846pm}.
\subsection{Polarization and some diagonally symmetric polynomials}
One can check directly the first identity in formulas (\ref{201504201124a}) and (\ref{201504201124b}). The other identities follow from formula (\ref{FormulillaPolarizato}) and the generating series (\ref{SerieGeneratrizSumasPotenciasDiagonales}) and (\ref{SerieGeneratrizMultisimetricasElementales}).
\begin{lemma}\label{201504211846pm}
If ${\bold d}\in\mathbb{N}^{\ell}$ is such that  $\boldsymbol{\vert}{\bold d}\boldsymbol{\vert}=d$, then we have the identities:
\begin{align}\label{201504201124a}
&E^{\bold d}\big(e_{1}({\bold x}_1)^{d}\big)=e_{1}^{\bold d}(X), \ \ \ E^{\bold d}\big(p_{d}({\bold x}_1)\big)=p_{\bold d}(X),\ \ \
E^{\bold d}\big(e_{d}({\bold x}_1)\big)=\frac{{\bold d}!}{d!}\,e_{\bold d}(X),
\end{align}
\begin{align}\label{201504201124b}
&E_{\bold d}\big(e_{1}^{\bold d}(X)\big)=e_{1}({\bold x}_1)^{d},
\ \ \
E_{\bold d}\big(p_{\bold d}(X)\big)=p_{d}({\bold x}_1),
\ \ \
E_{\bold d}\big(e_{\bold d}(X)\big)=\frac{d!}{{\bold d}!}\,e_{d}({\bold x}_1), 
\end{align}
\end{lemma}
\begin{corollary}
For any ${\bold d}\in\mathbb{N}^{\ell}$ such that
$\big\vert{\bold d}\big\vert=d$ we have the identities:  $\mathcal{M}_{e_{1}^{d}}=\mathcal{M}_{e_{1}^{\bold d}}$,\ $\mathcal{M}_{p_{d}}=\mathcal{M}_{p_{{\bold d}}}$,\ $\mathcal{M}_{e_{d}}=\mathcal{M}_{e_{{\bold d}}}$. 
\end{corollary}
\subsection{Generalized Polarization Modules}
Let $F$ be any subset of $\mathcal{R}_{n}^{(\ell)}$ consisting only of homogeneous polynomials. The polarization module generated by $F$ is the smallest $\mathbb{K}$-vector space $\mathcal{M}_{F}$ that satisfies the following axioms:
\begin{enumerate}
\item $\sigma\cdot g\in\mathcal{M}_{F}$, for all $g\in F$, and every permutation $\sigma\in\mathfrak{S}_n$,\\
\item $\mathcal{M}_{F}$ contains the set $F$,
\item $\mathcal{M}_{F}$ is closed under partial derivatives 
$\displaystyle{\partial_{i,j}=\frac{\partial}{\partial x_{ij}}}$,
\item $\mathcal{M}_{F}$ closed under the action of generalized polarization operators $$\displaystyle{E_{i,k}^{(p)}:=\sum_{j=1}^{n}x_{ij}\frac{\partial^p }{\partial x_{kj}^p}}.$$ 
\end{enumerate} 
In order to explain the construction of polarization modules starting from a set of homogeneous polynomials we describe the derivative closure and polarization closure of an homogeneous subspace $\mathcal{V}$ of $\mathcal{R}_{n}^{(\ell)}$. We say that $\mathcal{V}$ is \textbf{closed under partial derivatives} if for every $g\in\mathcal{V}$ we have $\partial_{ij}(g)\in\mathcal{V}$, for all $(i,j)$ such that $1\leq i\leq \ell$ and $1\leq j\leq n$. We say that $\mathcal{V}$ is \textbf{closed under polarization} if $E_{i,k}^{(p)}(g)\in \mathcal{V}$ for all $g\in \mathcal{V}$ and all suitable triples $(i,k,p)$. For any collection of subespaces which are separately closed under derivatives, their intersection is also closed under derivatives, so we can, define the \textbf{derivative closure $\boldsymbol{\mathcal{D}}(\mathcal{V})$ of $\mathcal{V}$} as the smallest subspace of $\mathcal{R}_{n}^{(\ell)}$ closed under derivatives that contains $\mathcal{V}$. Similarly, we define the \textbf{polarization closure $\boldsymbol{\mathcal{E}}(\mathcal{V})$ of $\mathcal{V}$}. The following lemma shows that we can compute the polarization module generated by $F$ in two ways:
\begin{lemma}
Let $\mathcal{V}$ be a homogeneous subspace of $\mathcal{R}_n^{(\ell)}$, then one has
\[ \boldsymbol{\mathcal{E}}(\boldsymbol{\mathcal{D}}(\mathcal{V}))=\boldsymbol{\mathcal{D}}(\boldsymbol{\mathcal{E}}(\mathcal{V})). \]
\end{lemma}
\begin{proof} 
The product rule of derivatives gives us the following formula:
\begin{equation}\label{DEVigualEDV}
\partial_{\alpha,\beta}E_{i,k}^{(p)}=\delta_{\alpha,i}\partial_{k\beta}^{p}+E_{i,k}^{(p)}\partial_{\alpha,\beta}.
\end{equation}
We start by showing that $\boldsymbol{\mathcal{D}}\left(\boldsymbol{\mathcal{E}}(\mathcal{V})\right)
\subseteq\boldsymbol{\mathcal{E}}\left(\boldsymbol{\mathcal{D}}(\mathcal{V})\right)$. An arbitrary element of $\boldsymbol{\mathcal{D}}(\boldsymbol{\mathcal{E}}(V))$ is a linear combination of sequences of partial derivatives followed by $E$-operators. So, it's enough to prove the inclusion for such sequences. From \ref{DEVigualEDV} we argue that any sequence of partial derivatives followed by $E$-operators can be rewritten as a sum of terms of two following kinds: 
\begin{itemize}
\item[(1)] a sequence of $E$-operators followed by partial derivatives, 
\item[(2)] a sequence of partial derivatives.
\end{itemize}
Now we observe that the terms of type (1) obviously belong to $\boldsymbol{\mathcal{E}}\left(\boldsymbol{\mathcal{D}}(\mathcal{V})\right)$. To show that the terms of type (2) belong to $\boldsymbol{\mathcal{E}}\left(\boldsymbol{\mathcal{D}}(\mathcal{V})\right)$ we have to argue that 
$\boldsymbol{\mathcal{D}}(\mathcal{V})\subseteq\boldsymbol{\mathcal{E}}\left(\boldsymbol{\mathcal{D}}(\mathcal{V})\right)$. We see that in both cases each term belongs to $\boldsymbol{\mathcal{E}}\left(\boldsymbol{\mathcal{D}}(\mathcal{V})\right)$. This give the desired inclusion. 
Conversely, we must show that \[\boldsymbol{\mathcal{E}}\left(\boldsymbol{\mathcal{D}}(\mathcal{V})\right)
\subseteq\boldsymbol{\mathcal{D}}\left(\boldsymbol{\mathcal{E}}(\mathcal{V})\right). \] 
In a similar way, as before we start from formula \ref{DEVigualEDV} to get
\begin{equation}\label{DEVigualEDV2}
E_{i,k}^{(p)}\partial_{\alpha,\beta}(g)
=\partial_{\alpha,\beta}E_{i,k}^{(p)}(g)-\delta_{\alpha,i}\partial_{k,\beta}^{p}(g),
\end{equation}
this implies that any sequence of $E$-operators followed by a sequence of partial derivatives can be rewritten as a sum of terms of the two following kinds:
\begin{itemize}
\item[(3)] a sequence of partial derivatives followed by a sequence of $E$-operators, 
\item[(4)] a sequence of partial derivatives.
\end{itemize}
Immediately we that the terms of type (3) belogns to $\boldsymbol{\mathcal{D}}\left(\boldsymbol{\mathcal{E}}(\mathcal{V})\right)$. Also the terms of type (4) satisfy the same inclusion because by definition we have 
\[\mathcal{V}\subseteq\boldsymbol{\mathcal{E}}(\mathcal{V})
\subseteq\boldsymbol{\mathcal{D}}\left(\boldsymbol{\mathcal{E}}(\mathcal{V})\right),\]
so, $\boldsymbol{\mathcal{D}}\left(\boldsymbol{\mathcal{E}}(\mathcal{V})\right)$ is a subspace closed by derivatives that contains $\mathcal{V}$. Since $\boldsymbol{\mathcal{D}}(\mathcal{V})$ is the smallest vector space with this property we must have 
\[\boldsymbol{\mathcal{D}}(\mathcal{V})\subseteq\boldsymbol{\mathcal{D}}\left(\boldsymbol{\mathcal{E}}(\mathcal{V})\right).\]
In both cases a term of types (3) or (4) belong to $\boldsymbol{\mathcal{D}}\left(\boldsymbol{\mathcal{E}}(\mathcal{V})\right)$. Thus, we have the inclusion $\boldsymbol{\mathcal{E}}\left(\boldsymbol{\mathcal{D}}(\mathcal{V})\right)
\subseteq\boldsymbol{\mathcal{D}}\left(\boldsymbol{\mathcal{E}}(\mathcal{V})\right)$. 
\end{proof}

We set $\boldsymbol{\mathcal{P}}(\mathcal{V}):=\boldsymbol{\mathcal{E}}(\boldsymbol{\mathcal{D}}(\mathcal{V}))$.
\noindent A subset $F$ of $\mathcal{R}_{n}^{(\ell)}$ is called a \textbf{homogeneous stable family} if the following conditions hold:
\begin{enumerate}
\item $F$ consists only of homogeneous polynomials, 
\item $F$ is \textbf{stable} (or $\mathfrak{S}_n$-stable w.r.t. the diagonal action of $\mathfrak{S}_n$), that is, for any permutation $\sigma\in\mathfrak{S}_n$ we have $\sigma\cdot g\in F$, for all $g\in F$.
\end{enumerate}
\begin{defi}
For a given homogeneous stable family $F$,  we set $\mathcal{M}_{F}$ to be the smallest $\mathbb{K}$-vector space closed under derivatives and closed under polarization containing the family $F$. We call the vector space $\mathcal{M}_{F}$ the \textbf{polarization module generated by the family $F$}.
\end{defi}
From the above considerations we can readly describe the polarization module generated by $F$ as follows:
\begin{prop}
If $F$ is an homogeneous stable family then
\begin{equation*}
\mathcal{M}_{F}:=\boldsymbol{\mathcal{P}}\big(\mathbb{K}\{F\}\big),
\end{equation*}
where $\mathbb{K}\{F\}$ denotes the $\mathbb{K}$-vector space spanned by $F$.
\end{prop} 
\begin{remark}
When the family $F$ consists of only one homogenous polynomial $f\in\mathcal{R}_{n}^{(\ell)}$, we denote by $\mathcal{M}_f$ the polarization module generated by the family $\{\sigma\cdot f\ \vert\ \sigma\in\mathfrak{S}_n\}$. In symbols,
\[ \mathcal{M}_{f}:=\boldsymbol{\mathcal{P}}\big(\mathbb{K}\{\sigma\cdot f\ \vert\ \sigma\in\mathfrak{S}_n\}\big).\]
\end{remark}
\subsection{Properties of polarization modules}

The property of being closed under polarization operators  implies that any polarization module is a polynomial representation of $GL_{\ell}(\mathbb{C})$ with the (right) action $f(X)\cdot M:=f(MX)$. In fact, we have the following well known result: 
\begin{lemma}[see C. Procesi, \cite{Procesi} page 83.]
A subspace $V$ of $\mathcal{R}_{n}^{(\ell)}$ is closed under the action of polarization operators $E_{i,k}$ (when $p=1$) if and only if $V$ is a $GL_{\ell}(\mathbb{C})$-module with the action $X^{A}\cdot M=(MX)^{A}$.
\end{lemma}
\noindent For any $g\in\mathcal{R}_{n}^{(\ell)}$ we have the identities: 
\begin{align}
&\sigma\cdot E_{i,k}^{(p)}(g)=E_{i,k}^{(p)}(\sigma\cdot g)\label{201504201232c1}\\   
&\sigma\cdot \partial_{ij}(g)=\partial_{i,\sigma(j)}(\sigma\cdot g).\label{201504201232c2}
\end{align}
Identities (\ref{201504201232c1}) and (\ref{201504201232c2}) imply the following assertions for any $\mathfrak{S}_n$-stable family $F$
\begin{lemma}
$\mathcal{M}_{F}$ is a representation of $\frak{S}_n$ with the diagonal action of $\mathfrak{S}_n$.
\end{lemma}
Recall that the two actions of $\mathfrak{S}_n$ and $GL_{\ell}(\mathbb{C})$ on $\mathcal{R}_{n}^{(\ell)}$ commute, so, we can assert that
\begin{lemma}
$\mathcal{M}_{F}$ is a\ $\frak{S}_{n}\times GL_{\ell}(\mathbb{C})$-module with the action $$(\sigma,M)\cdot f(X):=f(M^{-1}X_{\sigma}),$$
where $X_{\sigma}$ is the $\ell\times n$ matrix with $x_{i\sigma(j)}$ at the entry in the $i^{th}$ row and $j^{th}$ column.
\end{lemma}

\begin{remark}
We often use the notation ${\bold x}:=(x_1,\ldots,x_n)={\bold x}_1$, that is, ${\bold x}:={\bold x}_{1}$. Also, recall that ${\bold q}=(q_1,\ldots,q_{\ell})$ and ${\bold w}=(w_1,w_2,\ldots)$.
\end{remark}

\section{Fast examples}

For short we often write the Schur polynomial $s_{\mu}({\bold q})$ simply as $s_{\mu}$. The formulas below hold for any $\ell\geq 1$:
\begin{enumerate}
\item We consider $n=3$ and $f=x_{11}x_{22}x_{33}$, then the graded Frobenius characteristic of $\mathcal{M}_f$ is given by

\begin{align*}
\mathcal{M}_{f}({\bold q},{\bold w})=&(1+s_1+s_2+s_3)s_{3}({\bold w})+(s_1+s_2+s_{1,1}+s_{2,1})s_{2,1}({\bold w})\\
&+(s_{1,1}+s_{1,1,1})s_{1,1,1}({\bold w}).
\end{align*}

\noindent from the equation above we can immediately get the Hilbert series 

\[ \mathcal{M}_{f}({\bold q})
=1+\boldsymbol{3}s_1({\bold q})+\boldsymbol{3}s_2({\bold q})+\boldsymbol{3}s_{1,1}({\bold q})+s_3({\bold q})+\boldsymbol{2}s_{2,1}({\bold q})+s_{1,1,1}({\bold q}). \]

\item
Let $g$ be the symmetrization of $f$, that is
\[g:=x_{11}x_{22}x_{33}+x_{11}x_{23}x_{32}+x_{12}x_{21}x_{33}+x_{12}x_{23}x_{31}+x_{13}x_{22}x_{31}+x_{13}x_{21}x_{32}.\]
\noindent in this case we obtain 

\[ \mathcal{M}_{f}({\bold q},{\bold w})
=(1+s_1({\bold q})+s_2({\bold q})+s_3({\bold q}))\cdot s_{3}({\bold w})
+(s_1({\bold q})+s_2({\bold q}))\cdot s_{2,1}({\bold w}),\]
\noindent and the Hilbert series is simply
\[\mathcal{M}_{f}({\bold q})=1+\boldsymbol{3}s_1({\bold q})+\boldsymbol{3}s_2({\bold q})+s_3({\bold q}).\]  
\noindent More details will be explained in a corollary of Theorem \ref{ProposicionTRES}.

\item We take the set ${F}=\big\{h_{1,1}({\bold x}),h_{2}({\bold x})\big\}$, then the graded Frobenius series of $\mathcal{M}_{F}$ is  
\begin{equation*}
\mathcal{M}_{F}({\bold q},{\bold w})
={\big(}1+s_1({\bold q})+\boldsymbol{2}s_{2}({\bold q}){\big)}\cdot 
s_{n}({\bold w})
+s_1({\bold q})\cdot s_{n-1,1}({\bold w}).
\end{equation*}
\item For instance, in degree 3 we can take  
${F}=\big\{e_{_3}({\bold x}),m_{_{21}}({\bold x}),s_{_{21}}({\bold x})\big\}$, then the graded Frobenius series is the following:
\begin{equation*}
\mathcal{M}_{{F}}({\bold q},{\bold w})
={\big(}1+s_1({\bold q})+\boldsymbol{2}s_{2}({\bold q})+\boldsymbol{2}s_{3}({\bold q}){\big)}\cdot s_{n}({\bold w})
+{\big(}s_1({\bold q})+\boldsymbol{2}s_2({\bold q}){\big)}\cdot s_{n-1,1}({\bold w}).
\end{equation*}
\item  In this case we take  ${F}=\big\{p_{_3}({\bold x}),p_{_{21}}({\bold x}),p_{_{111}}({\bold x})\big\}$, then we get
\begin{equation*}
\mathcal{M}_{{F}}({\bold q},{\bold w})
={\big(}1+s_1({\bold q})+\boldsymbol{2}s_{2}({\bold q})+\boldsymbol{3}s_{3}({\bold q}){\big)}\cdot s_{n}({\bold w})
+{\big(}s_1({\bold q})+\boldsymbol{2}s_2({\bold q}){\big)}\cdot s_{n-1,1}({\bold w}).
\end{equation*}


\item Let $F$ be the following set of homogeneous polynomials of degree 4,
\begin{equation*}
F:=\big\{m_{_{22}}({\bold x}),p_{_{31}}({\bold x})\big\},
\end{equation*}
in this case, the graded Frobenius series is the following:
\begin{equation*}
(1+s_1+\boldsymbol{2}s_2+\boldsymbol{2}s_3+s_{2,1}+\boldsymbol{2}s_4)\cdot s_{n}({\bold w})+(s_1+\boldsymbol{2}s_2+s_{1,1}+\boldsymbol{2}s_3)\cdot s_{n-1,1}({\bold w})+s_{2}\cdot s_{n-2,2}({\bold w}).
\end{equation*} 
\end{enumerate}

\section{Detailed examples}\label{EjemplosDetallados}

\subsection{The polarization module $\mathcal{M}_{e_1^d}$}

Let's start by describing elements in the polarization module $\mathcal{M}_{e_1^{d}}$. If $d\geq 1$ and we set $f=e_{1}({\bold x}_1)^{d}=(x_{11}+\cdots+x_{1n})^d$, then $\partial_{1i}(f)=d\cdot e_{1}({\bold x}_1)^{d-1}$. In general, if $d\geq r\geq 1$ then
\[ \partial_{1i}^r(f)=\frac{d!}{(d-r)!}\cdot e_{1}({\bold x}_1)^{d-r}.\]
Then $\{1,e_{1}({\bold x}_1),e_{1}({\bold x}_1)^{2},\ldots,e_{1}({\bold x}_1)^{d}\}\subseteq\mathcal{M}_{e_1^{d}}$. Taking polarizations of 
$f$ we obtain 
\begin{equation*}
E_{i,1}^{\,r}(f)=\frac{d!}{(d-r)!}\,e_{1}({\bold x}_i)^{r}e_{1}({\bold x}_1)^{d-r}.
\end{equation*} 
In particular, $E_{i,1}^{\,d}(f)=d!\cdot e_{1}({\bold x}_i)$. If ${\bold d}\in\mathbb{N}^{\ell}$ is such that $\vert{\bold d}\vert=d$, we can generate by iterated polarizations the polynomial $e_{1}^{\bold d}(X)$:
\begin{equation}\label{201504211237}
e_{1}({\bold x}_{1})^{d_1}e_{1}({\bold x}_{2})^{d_2}\cdots e_{1}({\bold x}_{{\ell}})^{d_{\ell}}=\frac{d_1!}{d!}\,E_{{\ell},1}^{d_{\ell}}\cdots E_{2,1}^{d_2}\,(f).
\end{equation}
Then for any vector ${\bold a}\in\mathbb{N}^{\ell}$ such that $0\leq \vert{\bold a}\vert\leq d$ we can generate any polynomial of the form $e_{1}^{\bold a}(X)$ by taking partial derivatives of (\ref{201504211237}). Therefore,

\[ \big\{e_{1}^{\bold a}(X)\,:\,{\bold a}\in\mathbb{N}^{\ell},\,0\leq \vert{\bold a}\vert\leq d \big\}\subseteq\mathcal{M}_{e_1^d} \]

\noindent In the following we will show that the set 
\[ \mathcal{B}:=\big\{e_{1}^{\bold a}(X)\,:\,{\bold a}\in\mathbb{N}^{\ell},\,0\leq \vert{\bold a}\vert\leq d \big\}, \]
is indeed an homogeneous basis of $\mathcal{M}_{e_{1}^{d}}$. Since no two elements of $\mathcal{B}$ have the same vector degree, the set $\mathcal{B}$ is linearly independent. 
We can show that $\mathcal{B}$ is closed under derivatives and generalized polarization operators $E_{i,k}^{(p)}$ and contains $f$. This implies that $\mathcal{B}$ is a set of generators of 
$\mathcal{M}_{e_1^d}$. 
First, we compute the partial derivative
\[ \partial_{kj} e_{1}^{\bold a}(X)=a_k\cdot e_{1}({\bold x}_1)^{a_1}\cdots e_{1}({\bold x}_k)^{a_k-1}\cdots e_{1}({\bold x}_{\ell})^{a_{\ell}},\]
this shows that
\begin{equation} 
E_{i,k}\big(e_{1}^{\bold a}(X)\big)=a_k\cdot e_{1}^{\bold c}(X),
\end{equation}

where ${\bold c}\in\mathbb{N}^{\ell}$ is defined by
$c_{j}:=\begin{cases} 
a_k-1& \text{if}\ j=k,\\
a_i+1 & \text{if}\ j=i,\\ 
a_j & \text{otherwise}.
\end{cases}$ 
Iterating this, we see
\[ E_{i,k}^{\,r}\big(e_{1}^{\bold a}(X)\big)=\frac{a_k!}{(a_k-r)!}\cdot e_{1}({\bold x}_i)^r\cdot e_{1}({\bold x}_1)^{a_1}\cdots{e_{1}({\bold x}_k)^{a_{k}-r}}\cdots e_{1}({\bold x}_{\ell})^{a_{\ell}}. \]

If $p>1$ then, assuming $a_{k}\geq p$, we get
\begin{equation}\label{201504211348pm} 
{\partial_{kj}^p } e_{1}^{\bold a}(X)
=\frac{a_k!}{(a_k-p)!}
e_{1}({\bold x}_1)^{a_1}\cdots e_{1}({\bold x}_k)^{a_k-p}\cdots e_{1}({\bold x}_{\ell})^{a_{\ell}},\ \forall j.
\end{equation}
Taking the whole sum we get
\begin{align*}
\sum_{j=1}^{n}x_{ij}\partial_{kj}^{p}\big(e_{1}^{\bold a}(X)\big)
&=\sum_{j=1}^{n}x_{ij}\cdot \frac{a_k!}{(a_k-p)!}
e_{1}({\bold x}_1)^{a_1}\cdots e_{1}({\bold x}_k)^{a_k-p}\cdots e_{1}({\bold x}_{\ell})^{a_{\ell}}\\
&=\frac{a_k!}{(a_k-p)!}
e_{1}({\bold x}_1)^{a_1}\cdots e_{1}({\bold x}_k)^{a_k-p}\cdots e_{1}({\bold x}_{\ell})^{a_{\ell}}\left(\sum_{j=1}^{n}x_{ij}\right)\\
&=\frac{a_k!}{(a_k-p)!}\cdot{e_{1}({\bold x}_i)}\cdot
e_{1}({\bold x}_1)^{a_1}\cdots e_{1}({\bold x}_k)^{a_k-p}\cdots e_{1}({\bold x}_{\ell})^{a_{\ell}}.
\end{align*}
We can simply write 
\begin{equation}\label{201504211405pm} 
E_{i,k}^{(p)}\big(e_{1}^{\bold a}(X)\big)=
\begin{cases}
\displaystyle{\frac{a_k!}{(a_k-p)!}\,e_{1}^{\bold b}(X)} &\ \text{if}\ a_k\geq p,\\
&\\
\qquad 0 & \text{otherwise},
\end{cases} 
\end{equation}
where ${\bold b}=(b_1,\ldots,b_{\ell})$ is given by
\[b_{j}=\begin{cases} a_{j} & \text{if}\ j\neq k\ \text{et}\ j\neq i,\\a_{i}+1& j=i,\\
a_{k}-p& j=k.\\ 
\end{cases}\]
Thus, $E_{i,k}^{(p)}\big(e_{1}^{\bold a}(X)\big)\in\mathbb{K}\{\mathcal{B}\}$. If $a_k<p$ then $E_{i,k}^{(p)}\big(e_{1}^{\bold a}(X)\big)=0$. Thus, formulas (\ref{201504211348pm}) and (\ref{201504211405pm}) shows that  $\mathbb{K}\{\mathcal{B}\}$ (the $\mathbb{K}$-span of $\mathcal{B}$) is closed by derivatives and polarizations. Evidently $e_{1}^{d}({\bold x}_1)$ belongs to $\mathbb{K}\{\mathcal{B}\}$. The above considerations imply that 
\begin{equation*}
\mathcal{M}_{e_1^d}=\mathbb{K}\{\mathcal{B}\}.
\end{equation*}
and we already saw that $\mathcal{B}$ is linearly independent, so $\mathcal{B}$ is a linear basis of $\mathcal{M}_{e_1^d}$. In fact, we can assert that
\[ \mathcal{M}_{e_1^d}=\mathbb{K}_{\leq d}\big[e_{1}({\bold x}_1),\ldots,e_{1}({\bold x}_{\ell})\big], \]
that is, $\mathcal{M}_{e_1^d}$ equals the set of all polynomials of degree at most $d$ in the variables $e_{1}({\bold x}_1),\ldots,e_{1}({\bold x}_{\ell})$. Then
\begin{equation*}
\dim\big(\mathcal{M}_{e_{1}^{d}}\big)
=\displaystyle{{\ell+d\choose d}},\ \text{for every}\ 
\ell\geq {\bold 1}.
\end{equation*}

\subsection{The polarization module $\mathcal{M}_{p_d}$} 

If $d\geq 1$ and $f=p_{d}({\bold x}_1)=x_{11}^d+\cdots+x_{1n}^d$, we obtain\ $\partial_{1j}f=d\cdot x_{1j}^{d-1}$. In general, if $d\geq r\geq 1$ we get
\[ \partial_{1j}^r(f)=\frac{d!}{(d-r)!}\cdot x_{1j}^{d-r}.\]
This implies that $\{x_{1j}^k\,:\,0\leq k\leq d-1\}\subseteq\mathcal{M}_{p_d}$. Taking polarizations of $x_{1j}^k$ we obtain the polynomials
\[ E_{i,1}^{\,r}\big(x_{1j}^k\big)=\frac{k!}{(k-r)!}\cdot x_{ij}^{r}x_{1j}^{k-r}\]
\[ E_{i,1}^{(p)}\big(x_{1j}^k\big)=\frac{k!}{(k-p)!}\cdot x_{ij}x_{1j}^{k-p}.\]
In general, if ${\bold k}=(k_1,\ldots,k_{\ell})$ is such that $\vert{\bold k}\vert=k$ then we have (see, \cite{Hunziker})
\[ X_{j}^{\bold k}:=x_{\ell j}^{k_{\ell}}\cdots x_{1j}^{k_1}
=\frac{k_{1}!}{k!}\cdot E_{\ell,1}^{k_{\ell}}\cdots E_{2,1}^{k_1}\big(x_{1j}^k\big).\]
This implies that
\[ \{X_{j}^{\bold a}\,:\, 1\leq j\leq n,\, {\bold a}\in\mathbb{N}^{\ell},\,0\leq\vert{\bold a}\vert\leq d-1\}\subseteq\mathcal{M}_{p_d}.\]

Recall from Lemma \ref{201504211846pm} that for any ${\bold d}\in\mathbb{N}^{\ell}$ with $\vert{\bold d}\vert=d$ we obtain $E^{\bold d}(f)=p_{\bold d}(X)$. Therefore
\[ \big\{p_{\bold b}(X)\,:\,{\bold b}\in\mathbb{N}^{\ell},\,\vert{\bold b}\vert=d\big\}\subseteq\mathcal{M}_{p_d}.\]
The above considerations show that the set 
\[ \mathcal{B}:=\big\{p_{\bold b}(X)\,:\,{\bold b}\in\mathbb{N}^{\ell},\,\vert{\bold b}\vert=d\big\}\bigcup\big\{X_{j}^{\bold a}\,:\, 1\leq j\leq n,\, {\bold a}\in\mathbb{N}^{\ell},\,0\leq\vert{\bold a}\vert\leq d-1\big\}, \]
is contained in $\mathcal{M}_{p_d}$. Clearly, we have the identity:
\begin{equation}\label{201601111403}
E_{a,b}^{(p)}\big(x_{c,d}^{m}\big)=\delta_{b,c}\frac{m!}{(m-p)!}x_{a,d}x_{c,d}^{m-p}. 
\end{equation}
 we can show the following result 
\begin{lemma}
The $\mathbb{K}$-span of $\mathcal{B}$ is closed under derivatives and closed under the action of generalized polarization operators $E_{i,k}^{(p)}$. 
\end{lemma}
\begin{proof}
Using the usual derivatives rules and formula \ref{201601111403} we get:
\begin{align*}
\partial_{i,j}^r\big(X_{s}^{\bold a}\big)
&=\delta_{sj}\frac{a_{i}!}{(a_{i}-r)!}\cdot x_{1s}^{a_1}\cdots{x_{is}^{a_{i}-r}}\cdots x_{\ell{s}}^{a_{\ell}}\in\mathbb{K}\{\mathcal{B}\},\\
E_{i,k}^{\,r}\big(X_{s}^{\bold a}\big)&=
\frac{a_k!}{(a_k-r)!}\cdot x_{is}^{r}\,x_{1s}^{a_1}\cdots {x_{ks}^{a_k-r}}\cdots x_{\ell{s}}^{a_{\ell}}\in\mathbb{K}\{\mathcal{B}\},\\
E_{i,k}^{(p)}\big(X_{s}^{\bold a}\big)&=
\frac{a_k!}{(a_k-p)!}\cdot x_{is}\,x_{1s}^{a_1}\cdots {x_{ks}^{a_k-p}}\cdots x_{\ell{s}}^{a_{\ell}}\in\mathbb{K}\{\mathcal{B}\},\\
\partial_{i,j}^r\big(p_{\bold b}\big)&=
\frac{b_{i}!}{(b_i-r)!}\cdot x_{1j}^{b_1}\cdots{x_{ij}^{b_i-r}}
\cdots x_{\ell{j}}^{b_{\ell}}\in\mathbb{K}\{\mathcal{B}\},\\
E_{i,k}^{\,r}\big(p_{\bold b}\big)&=
 \frac{a_k!}{(a_k-r)!}\cdot 
\sum_{j=1}^{n}x_{ij}^{r}x_{1j}^{b_1}\cdots{x_{kj}^{b_i-r}}
\cdots {x_{\ell{j}}^{b_{\ell}}}\in\mathbb{K}\{\mathcal{B}\}\\
E_{i,k}^{(p)}\big(p_{\bold b}\big)&=
\frac{a_k!}{(a_k-p)!}\cdot 
\sum_{j=1}^{n}x_{ij}x_{1j}^{b_1}\cdots{x_{kj}^{b_i-p}}\cdots
{x_{\ell{j}}^{b_{\ell}}}\in\mathbb{K}\{\mathcal{B}\}.  
\end{align*}
\end{proof}
Then $\mathcal{B}$ satisfies the axioms of the definition of polarization module. Since $\mathcal{M}_f$ is the smallest with these properties we must have that $\mathcal{M}_f\subseteq \mathbb{K}\{\mathcal{B}\}$. Also, no two polynomials in $\mathcal{B}$ have the same multidegree, so, $\mathcal{B}$ is linearly independent. Therefore $\mathcal{B}$ is an homogeneous basis of $\mathcal{M}_{p_d}$.
 
\section{Main theorems}

\begin{teor}\label{ProposicionUNO}
Let $d$ be a positive integer. The following formulae hold for any $\ell\geq 1$
\begin{equation}\label{ProposicionUNOPrimeraParte}
\mathcal{M}_{e_{1}^{d}}({\bold q},{\bold w})=\left(\sum_{j=0}^{d}s_{j}({\bold q})\right)s_{n}({\bold w}).
\end{equation}
\begin{align}\label{ProposicionUNOSegundaParte}
\mathcal{M}_{p_{d}}({\bold q},{\bold w})&=\left(\sum_{j=0}^{m}s_{j}({\bold q})\right){s}_{n}({\bold w})
+\left(\sum_{j=1}^{m-1}s_{j}({\bold q})\right){s}_{n-1,1}({\bold w}).
\end{align}
\begin{align}\label{ProposicionUNOTerceraParte}
\mathcal{M}_{e_{d}}({\bold q},{\bold w})
&=\sum_{i=0}^{\lfloor{d/2}\rfloor}\left(\sum_{j=i}^{d-i}s_{j}({\bold q})\right)s_{n-i,i}({\bold w}).
\end{align}
\end{teor}
\begin{corollary}\label{CorolarioUNO}
Let $m$ be a positive integer. Let ${\bold d}\in\mathbb{N}^{\ell}$ be any vector such that $\vert{\bold d}\vert=d$. For the diagonally symmetric polynomials \ $e_{1}^{\bold d}(X)$, $p_{\bold d}(X)$, and $e_{\bold d}(X)$ we have (respectively) the following universal formulas for the Frobenius characteristic of their associated $\mathfrak{S}_n$-modules:
\begin{equation}\label{CorolarioUNOPrimeraParte}
\mathcal{M}_{e_{1}^{\bold d}(X)}({\bold q},{\bold w})=\left(\sum_{j=0}^{d}h_{j}({\bold q})\right)h_{n}({\bold w}).
\end{equation}
\begin{equation}\label{ProposicionUNOSegundaParteEcuacion}
\mathcal{M}_{p_{\bold d}}({\bold q},{\bold w})=\big(1+h_{d}({\bold q})\,\big)\,h_{n}({\bold w})+\left(\sum_{j=1}^{d-1}h_{j}({\bold q})\right)h_{n-1,1}({\bold w})
\end{equation}
\begin{equation}\label{CorolarioUNOTerceraParte}
\mathcal{M}_{e_{\bold d}}({\bold q},{\bold w})
=\sum_{i=0}^{\lfloor\frac{d}{2}\rfloor}h_{n-i,i}({\bold w})\,h_{i}({\bf q})
+\sum_{i=\lfloor\frac{d}{2}\rfloor+1}^{d}h_{n-d+i,d-i}({\bold w})\,h_{i}({\bf q}).
\end{equation}
\end{corollary}

Consider the following $\frak{S}_n$-stable families of homogeneous polynomials in the variables $x_{11},\ldots,x_{1n}$,  defined as,
$\mathcal{A}:=\big\{x_{1j}^d\big\vert\ 1\leq j\leq n\big\}$ and $\mathcal{B}:=\big\{x_{1,i}^d-x_{1,j}^d\big\vert\ 1\leq i<j\leq n \big\}$.

\begin{teor}\label{FamiliasAB}
The graded Frobenius characteristic of the families $\mathcal{A}$ and $\mathcal{B}$ are given by 
\begin{equation*}
\mathcal{M}_{_\mathcal{A}}({\bold q},{\bold w})
=\left(\sum_{j=0}^{d}s_{j}({\bold q})\right)s_{n}({\bold w})
+\left(\sum_{j=1}^{d}s_{j}({\bold q})\right)s_{n-1,1}({\bold w}).
\end{equation*}
\begin{equation*}
\mathcal{M}_{_\mathcal{B}}({\bold q},{\bold w})
=\left(\sum_{j=0}^{d-1}s_{j}({\bold q})\right)s_{n}({\bold w})
+\left(\sum_{j=1}^{d}s_{j}({\bold q})\right)s_{n-1,1}({\bold w}).
\end{equation*}
\end{teor}
\begin{teor}\label{ProposicionDOS}
Let $f({\bold x}_1)$ be a symmetric polynomial of degree 2  in $n\geq 2$ variables ${\bold x}_1=x_{11},x_{12},\ldots,x_{1n}$. Suppose that $f({\bold x}_1)$ is given in the monomial basis as follows:
\begin{equation}
f({\bold x}_1)=a\cdot m_{2}({\bold x}_1)+b\cdot m_{1,1}({\bold x}_1),
\end{equation}
then the Frobenius characteristic of the space $\mathcal{M}_{f}$ is given by one of the following two cases:
\begin{equation}
\mathcal{M}_{f}({\bold{q}},{\bold{w}})=
\left\{
\begin{array}{cc}
\big(1+h_1({\bold q})+h_2({\bold q})\big)\cdot h_{n}({\bold{w}}) &  \text{if} \ b=2a \\
&\\
\big(1+s_1({\bold q})+s_2({\bold q})\big)\cdot s_{n}({\bold{w}})+s_1({\bold q})\cdot s_{n-1,1}({\bold{w}})& {\rm otherwise}.\\
=\big(1+h_{2}({\bold{q}})\,\big)\cdot h_{n}({\bold{w}})+h_{1}({\bold q})\cdot h_{n-1,1}({\bold w}).
\end{array}
\right.
\end{equation}
\end{teor}
\begin{corollary}\label{CorolarioGrado2}
If $f$ is a homogeneous symmetric polynomial of degree 2, then the associated $\frak{S}_n$-module $\mathcal{M}_f$ is isomorphic as an $\mathfrak{S}_n\times GL_{\ell}(\mathbb{K})$-module to one of the two modules $\mathcal{M}_{p_1^{2}}$,
$\mathcal{M}_{p_{2}}$.
\end{corollary}
\begin{teor}\label{ProposicionTRES}
Let $f$ be a homogeneous symmetric polynomial of degree 3 in $n\geq 2$ variables. Suppose that $f\in\mathbb{R}[x_{11},\ldots,x_{1n}]$ and $[f]=[a:b:c]$, then the Frobenius characteristic of the $\frak{S}_n$-module $\mathcal{M}_{f}$ is given by one of the following three cases:
\begin{enumerate}
\item If $\big[f\big]=\big[1:3:6\big]$ then
$$\big(\,1+h_1({\bold q})+h_2({\bold q})+h_3({\bold q})\,\big)\cdot h_{n}({\bold w}).$$
\item If $n\geq 3$ and $[a:b:c]$ satisfies $6a(2b+(n-2)c)=4(n-1)b^2$ then we get:
\begin{align*}
\mathcal{M}_{f}({\bold q},{\bold w})&=\big(\,1+h_3({\bold q})\,\big)\cdot h_{n}({\bold{w}})+\big(\,h_{1}({\bold q})+h_{2}({\bold q})\,\big)h_{n-1,1}({\bold{w}}) \\
&=(1+s_1({\bold q})+s_2({\bold q})+s_3({\bold q}))\cdot s_{n}({\bold w})
+(s_1({\bold q})+s_2({\bold q}))\cdot s_{n-1,1}({\bold w}) 
\end{align*}
\item Otherwise we have the following answer:
\begin{align*}
\mathcal{M}_{f}({\bold q},{\bold w})=&\big(\,1+h_{2}({\bold q})+h_3({\bold q})\,\big)\cdot h_{n}({\bold{w}})+\big(h_1({\bold q})+h_2({\bold q})\,\big)\cdot h_{n-1,1}({\bold{w}})\\
&=(1+s_1({\bold q})+\boldsymbol{2}s_2({\bold q})+s_3({\bold q}))\cdot s_{n}({\bold w})+(s_1({\bold q})+s_2({\bold q}))\cdot s_{n-1,1}({\bold w}).
\end{align*}
\end{enumerate}
When $n=2$ the condition giving the second case above is that either $b=0$ or $b=3a$.
\end{teor}

\section{Exceptions}

\subsection{Definiton and examples of $n$-exceptions}
Recall that we identify $f=a\cdot m_3+b\cdot m_{21}+c\cdot m_{111}$, ($a$,$b$,$c$ in $\mathbb{R}$) with its homogeneous coordinates 
$[f]:=[a:b:c]\in\mathbb{RP}^{2}$. For instance $p_{1}^{3}$ is the point $[1:3:6]$, $p_{21}$ is $[1:1:0]$ and $h_{3}$ is the point $[1:1:1]$. 
\begin{defi}
We say that an homogeneous symmetric polynomial $f$ in 
$\mathbb{R}[x_{11},\ldots,x_{1n}]$ is a \textbf{$n$-exception} if 
\[ \dim\left(\mathbb{R}\{\partial_{11}(f),\ldots,\partial_{1n}(f),E_{1,1}^{(2)}(f)\}\right)=n. \]
In other words, $f$ is a $n$-exception if the dimension of the real linear span of its first order partial derivatives $\partial_{11}(f),\ldots,\partial_{1n}(f)$ and the polynomial $E_{1,1}^{(2)}(f)$ has dimension $n$.  
\end{defi}

\subsubsection{Examples of $n$-exceptions}

\begin{enumerate}

\item $[\alpha:0:1]$ is a $2$-exception if $\alpha\neq0$. \\
\item $[1:3:s]$ is a $2$-exception, for all $s\in\mathbb{R}$. \\
\item $[0:0:1]$ ($f=k\,e_3$), is a $n$-exception, for all $n\geq 2$. \\
\item $[1:0:0]$ ($f=k\,p_3$), is a $n$-exception, for all $n\geq 3$. \\
\item $[2:-3:12]$ is a $3$-exception. \\
\item $[3:3:-2]$ is a $3$-exception. \\
\item $[1:1:0]$ ($f=p_{21}$), is a  $4$-exception. \\ 
\item $[1:-1:2]$ is a  $4$-exception. \\
\item $[16:-12:21]$ is a  $4$-exception. \\
\item $[9:21:28]$ is a  $4$-exception. \\
\item $[4:-3:4]$, is a  5-exception. \\
\item $[5:-3:3]$, is a  6-exception. \\
\item $[10:-5:4]$, is a  7-exception. \\
\item $[0:1:0]$, $(f=m_{_{21}})$ \textbf{is not} a $n$-exception when  $n\geq 3$.  \\
\item $[1:1:1]$, $(f=h_3)$ \textbf{is not} a $n$-exception when $n\geq 3$.  \\

 \end{enumerate}


In the table below we have some examples of conditions for a point $[a:b:c]$ to be an $n$-exception up to $n=7$. For more examples of $n$-exception equations up to $n=42$ see \cite{BlandinTesisPhD}.
\begin{table}[ht!]
\centering
\caption{$n$-exceptions for $n\leq 7$}
\label{table:ExcepcionesGrado3}
\begin{tabular}{|c|c|}
\hline
$n=3$ & $3a(2b+c)=4b^2$\\
\hline
$n=4$ & $a(b+c)=b^2$\\
\hline
$n=5$  & $3a(2b+3c)=8b^2$ \\
\hline
$n=6$  & $3a(b+2c)=5b^2$ \\
\hline
$n=7$  & $a(2b+5c)=4b^2$ \\
\hline
\end{tabular}
\end{table}

For instance, in degree 4, the point $[5:14:21:28:35]$ is a 11-exception. Note that for general degrees, we believe that $p_{2}p_{1}^{d-2}$ is a $d+1$-exception for any $d\geq 3$ (this assertion is already settled when $d=3,4,\ldots,11$. Also, one can easily show that for any $d\geq 3$ the polynomial $g=p_{2}p_{1}^{d-2}$ (in the variables $x_{11},\ldots,x_{1d},x_{1,d+1}$) satisfies the identity below
\[ E_{1,1}^{(2)}(g)=\sum_{j=0}^{d+1}{\partial_{1j}(g)}.\]
The following theorem seems to hold for any degree greater than 3 (see \cite{BlandinExceptions}). We will present in section \ref{Seccion201601151608} a proof for the case of degree 3 of the theorem below:
\begin{teor}\label{TeoremaExcepcionGrado3}
Let $f$ be a homogeneous symmetric polynomial of degree 3 in $\mathbb{R}[x_{11},\ldots,x_{1n}]$. If $[f]\neq[1:3:6]$ then 
\[ \dim\left(\mathbb{R}\{\partial_{11}(f),\ldots,\partial_{1n}(f),E_{1,1}^{(2)}(f)\}\right)\geq n. \]
\end{teor}
\begin{teor}\label{EcuacionesExcepciones}
Suppose that $n\geq 3$ and $f({\bold x}_1)=a\cdot m_3({\bold x}_1)+b\cdot m_{21}({\bold x}_1)+c\cdot m_{111}({\bold x}_1)$, ($a$,$b$,$c$ in $\mathbb{R}$). Then $f$ is a $n$-exception if and only if $[a:b:c]\neq[1:3:6]$ and $6a(2b+(n-2)c)=4(n-1)b^2$. When $n=2$, $[a:b:c]$ is a 2-exception if and only if $b=0$ or $b=3a$. 
\end{teor}
\begin{prop}\label{ConjeturaDeHector}
Let $a$, $b$ and $c$ be real numbers and $n\geq 3$. The point $[a:b:c]\in\mathbb{RP}^{2}$ is an $n$-exception if and only if  \ $[a:b:c]\neq[1:3:6]$ \ and  \ $n_{_1}a(n_{_2}b+n_{_3}c)=n_{_4}b^{2}$ \ where the integers $n_{_i}$ (depending on $n$) are 
\begin{align*}
& n_{_1}(n)=\frac{3}{\mathrm{g.c.d.}(n+2,3)}, \qquad \qquad \qquad 
 n_{_2}(n)=\mathrm{g.c.d.}(n+1,n-1), \quad \\
& n_{_3}(n)
=\begin{cases}
n-2  & \text{if}\ n\ \text{is odd},\\
&\\
\displaystyle{ \frac{n-2}{2} }  
& \text{if}\ n\ \text{is even}.
\end{cases}   \ \ \ \ \ \
 n_{_4}(n)
=\begin{cases} 
\displaystyle{\frac{n-1}{\mathrm{g.c.d.}(n-1,6)}} & n\ \text{is even,} \\ 
& \\ \displaystyle{\frac{2n-2}{\mathrm{g.c.d.}(n-1,3)}} & n\ \text{is odd.} 
\end{cases}.
\end{align*}
\end{prop}

\begin{corollary}\label{CorolarioGrado3}
Le $f\in\mathbb{R}[{\bold x}]$ be a homogeneous symmetric polynomial of degree 3. There are three types of polarization modules: $\mathcal{M}_{p_1^{3}}$, $\mathcal{M}_{p_{3}}$ \ or $\mathcal{M}_{h_3}$.
\end{corollary}

\section{Proof of Theorems \ref{ProposicionUNO} and \ref{FamiliasAB}}

In this section we explain the proof of main results of this paper. We know that $\mathcal{M}_{F}$ is a $\mathbb{N}^{\ell}$-graded vector space 
\[ \mathcal{M}_{F}
=\bigoplus_{{\bold d}\in\mathbb{N}^{\ell}}V_{\bold d}.\] 
Here $V_{\bold d}:=\mathcal{M}_{F}\cap\mathcal{R}_{n,{\bold d}}^{(\ell)}$ denotes the homogeneous component of $\mathcal{M}_F$ and we write $\mathcal{B}_{\bold d}$ to denote a basis of the corresponding homogeneous component. 

\begin{remark}
Suppose that $V$ is a vector space and $B=\{v_1,\ldots,v_r\}$ is a basis of $V$. If $w\in V$, then, there exist scalars $\alpha_1,\ldots,\alpha_r$ such that $w=\alpha_1\cdot v_1+\cdots+\alpha_r\cdot v_r$. We denote by $w\big\vert_{v_j}$ the coefficient of $v_{j}$ in the linear combination of $w$ with respect to the basis $B$, that is
 \[ w\bigg\vert_{v_j}:=\alpha_j.\]
\end{remark}

\subsection{The Frobenius characteristic of $\mathcal{M}_{e_1^d}$}
\begin{lemma}\label{Lema5}
For each ${\bold d}:=(d_1,\ldots,d_{\ell})\in\mathbb{N}^{\ell}$ such that \ $0\leq\boldsymbol{\vert}{\bold d}\boldsymbol{\vert}\leq d$,\ a basis for the degree ${\bold d}$ homogeneous component $V_{\bold d}$, of the space $\mathcal{M}_{e_1^m}$ is given by the one element set
\begin{equation*}
\mathcal{B}_{\bold d}=\big\{e_1({\bold x}_1)^{d_1}\cdots\, e_{1}({\bold x}_{\ell})^{d_\ell}\big\}=\big\{e_{1}^{\bold d}(X)\big\}.
\end{equation*}
If $\vert{\bold d}\vert>d$, then $\mathcal{B}_{\bold d}=\varnothing$. Hence a basis for the space $\mathcal{M}_{e_1^d}$ is the set 
\begin{equation*}
\mathcal{B}:=\big\{e_{1}^{\bold d}(X)\,:\,{\bold d}\in\mathbb{N}^{\ell},\,0\leq\boldsymbol{\vert}{\bold d}\boldsymbol{\vert}\leq d\big\}.
\end{equation*}
\end{lemma}
\begin{proof}
See Section \ref{EjemplosDetallados}. 
\end{proof}
\noindent In this case, if $0\leq\vert{\bold d}\vert\leq d$, the homogeneous component of $\mathcal{M}_{e_1^d}$ is 
\[ V_{\bold d}=\big\{k\cdot e_{1}^{\bold d}(X)\,:\,k\in\mathbb{K}\big\}. \] 
Recall that $e_{1}^{\bold d}(X)$ is diagonally symmetric, so 
$\mathfrak{S}_n$ acts trivially on each $V_{\bold d}$. This implies the following assertion:
\begin{lemma}\label{201504212257pm}
For any ${\bold d}\in\mathbb{N}^{\ell}$ such that $0\leq\vert{\bold d}\vert\leq d$, the $\mathfrak{S}_n$-character of $V_{\bold d}$ is   
\begin{equation*}
{\chisotazo}_{V_{\bold d}}(\sigma)=1,\ \forall\,\sigma\in\mathfrak{S}_n.
\end{equation*}
\end{lemma}
\noindent The above lemma tells us that the Frobenius characteristic is given simply by:
\[ \mathcal{M}_{e_1^{d}}({\bold{q}},{\bold{w}})
=\sum_{0\leq\vert{\bold d}\vert\leq d}{\bold q}^{\bold d}s_{n}({\bold w}). \]
Writting everthing in terms of Schur fucntions, we get 
\begin{corollary}\label{201504231304pm}
The Frobenius characteristic of the space $\mathcal{M}_{e_1^{d}}$ is given by the formula:
\begin{equation*}
\mathcal{M}_{e_1^{d}}({\bold{q}},{\bold{w}})=\left(\sum_{j=0}^{d}s_{j}({\bold q})\right)s_{n}({\bold{w}}),\ \ \forall\,n\geq 1.
\end{equation*}
\end{corollary}
\noindent To get the Hilbert series we replace each $s_{\lambda}(\bold w)$ by the number $f^{\lambda}$ (the number of standard Young tableaux of shape 
$\lambda$). In this case $s_{n}({\bold w})$ is replaced by $f^{(n)}=1$. Then we have the following:
\begin{corollary}
The Hilbert series of $\mathcal{M}_{e_{1}^{d}}$ is given by
\begin{equation*}
\mathcal{M}_{e_{1}^{d}}({\bold q})=1+s_1({\bold q})+s_2({\bold q})+\cdots+s_d({\bold q}).
\end{equation*}
\end{corollary}

\subsection{The Frobenius characteristic of $\mathcal{M}_{p_d}$} 

\begin{lemma}\label{Base2}
Let $f=p_{d}({\bold x}_1):=x_{11}^{d}+\cdots+x_{1n}^{d}$.
A basis for the space $\mathcal{M}_{p_{d} }$ is given by the following set:
\begin{equation*}
\mathcal{B}=\big\{\,p_{{\bold b}}(X)\,:\,{\bold b}\in\mathbb{N}^{\ell},\,\big\vert{\bold b}\big\vert=d\big\}\bigcup\big\{X_{j}^{\bold a}\,\,:\,1\leq j\leq n,\,{\bold a}\in\mathbb{N}^{\ell} ,\,0\leq\big\vert{\bold a}\big\vert<d\big\}.
\end{equation*}
where\ $b_{i}\geq 0$, for all $i$ such that $1\leq i\leq \ell$.
\end{lemma}
\begin{proof}
See Section \ref{EjemplosDetallados}.
\end{proof}
\begin{lemma}\label{201504221310pm}
A basis for the homogeneous component $V_{\bold d}$ of 
$\mathcal{M}_{p_d}$ is given by the following rule
\begin{equation*}
\mathcal{B}_{\bold d}:=
\begin{cases}
\qquad \qquad \big\{\,p_{{\bold d}}(X)\big\} & \ \text{if} \ \ \vert{\bold d}\vert=d,\\
&\\
\big\{X_{j}^{\bold d}\,\,:\,1\leq j\leq n\big\} & 
\ \text{if}\ \ 0\leq\vert{\bold d}\vert<d,
\end{cases}
\end{equation*}
\end{lemma}
\noindent From the representation theory of the symmetric group, it's well known that (see \cite{BruceSagan} page 87)
\begin{equation}
{\chisotazo}^{(n)}(\sigma)=1,
\end{equation}
\begin{equation}
{\chisotazo}^{(n-1,1)}(\sigma)=|\textrm{Fix}(\sigma)|-1,
\end{equation}
where $\mathrm{Fix}(\sigma)$ is the set of fixed points of $\sigma$.
If a group $G$ acts on a set $S$ with the corresponding permutation representation $\mathbb{C}S$ then the character $\chisota_{S}$ of this representation is
\begin{equation}\label{CaracterDeLaRepresentacionPorPermutacion}
\chisota_{S}(g)=\vert\{x\in S\,:\, g\cdot x=x\}\vert.
\end{equation}
\begin{lemma}\label{201504221304pm}
The $\mathfrak{S}_n$-character of the homogeneous component $V_{\bold d}$ of $\mathcal{M}_{p_d}$ is given by
\begin{equation*}
{\chisotazo}_{\mathcal{V}_{\bold d}}(\sigma)=\left\{
\begin{array}{cc}
1 & \textrm{if}\ \ {\bold d}={\bold 0}\ \ \textrm{or} \ \ \big\vert{\bold d}\big\vert=d,\\
&\\
\big\vert\mathrm{Fix}(\sigma)\big\vert & \textrm{if}\ \ 0<\big\vert{\bold d}\big\vert<d.
\end{array}
\right.
\end{equation*}
\end{lemma}
\begin{proof}
Let ${\bold d}\in\mathbb{N}^{\ell}$. We have two cases:
\begin{enumerate}
\item If $\big\vert{\bold d}\big\vert=d$ then a basis $\mathcal{B}_{\bold d}$ for $V_{\bold d}$ consists of only one diagonally symmetric polynomial. So, just as in the proof of Lemma \ref{201504212257pm}, we get \ ${\chisotazo}_{V_{\bold d}}(\sigma)=1$, for every $\sigma\in\frak{S}_n$.
\item If \ $0\leq \big\vert{\bold d}\big\vert<d$ \ then \ 
 $\mathcal{B}_{\bold d}=\big\{X_{j}^{\bold d}\,\,:\,1\leq j\leq n\big\}$. Therefore
\[\displaystyle{X_{\sigma(j)}^{\bold d}\Bigg\vert_{X_{j}^{\bold d}}
=\begin{cases} 1 & \text{if}\ \sigma(j)=j,\\ 
0 & \text{otherwise}.
\end{cases}}\]
\begin{align}
{\chisotazo}_{V_{\bold d}}(\sigma)
=\sum_{j=1}^{n}X_{\sigma(j)}^{\bold d}\Bigg\vert_{X_{j}^{\bold d}}
=\big\vert\mathrm{Fix}(\sigma)\big\vert={\chisotazo}^{(n)}(\sigma)+{\chisotazo}^{(n-1,1)}(\sigma).
\end{align}
\end{enumerate} 
\end{proof}
\noindent Recall the classical formula (see, \cite{BruceSagan})
\[ s_{\lambda}({\bold w})=\frac{1}{n!}\sum_{\sigma\in\mathfrak{S}_n}\chisotazo^{\lambda}(\sigma)\,p_{\lambda(\sigma)}({\bold w}).\]
where $\chisotazo^{\lambda}(\sigma)$ is the character of the irreducible representation $\mathcal{S}^{\lambda}$ of $\mathfrak{S}_n$.
\begin{corollary}
The Frobenius characteristic of the space $\mathcal{M}_{p_{d}}$ is given by the formula:
\begin{equation*}
\mathcal{M}_{p_d}({\bold{q}},{\bold{w}})=\left(\sum_{j=0}^{d}s_{j}({\bold q})\right)s_{n}({\bold{w}})
+\left(\sum_{j=1}^{d-1}s_{j}({\bold q})\right)s_{n-1,1}({\bold{w}}),\ \  n\geq 2.
\end{equation*}
\end{corollary}
\begin{proof}
We use Lemma \ref{201504221304pm} as follows
\begin{align*}
&\mathcal{M}_{f}({\bold q},{\bold w}):=
\sum_{{\bold d}\in\mathbb{N}^{\ell}}\left(\frac{1}{n!}\sum_{\sigma\in\frak{S}_{n}}{\chisotazo}_{\mathcal{V}_{{\bold d}}}(\sigma)
p_{\lambda(\sigma)}({\bold w})\right)
{\bold q}^{{\bold d}}\\
&=\left(1+\sum_{\big\vert{\bold d}\big\vert=m}{\bold q}^{\bold d}\right)\left(\frac{1}{n!}\sum_{\sigma\in\frak{S}_{n}}
p_{\lambda(\sigma)}({\bold w})\right)+\sum_{0<\big\vert{\bold d}\big\vert<m}\frac{1}{n!}\sum_{\sigma\in\frak{S}_{n}}
\big\vert\mathrm{Fix}(\sigma)\big\vert
p_{\lambda(\sigma)}({\bold w})\\
&=\left(\sum_{j=0}^{m}h_{j}({\bold q})\right)\cdot s_{n}({\bold w})+\left(\sum_{j=1}^{m-1}h_{j}({\bold q})\right)
\cdot s_{n-1,1}({\bold w}).
\end{align*}
\end{proof}  

\begin{corollary}
The Hilbert series of $\mathcal{M}_{p_d}$ is given by
\[\mathcal{M}_{p_d}({\bold q})=h_{d}({\bold q})+n\cdot h_{d-1}({\bold q})+\cdots+n\cdot h_{1}({\bold q})+1.\]
\end{corollary}
\begin{proof}
By Lemma \ref{201504221310pm} we know that
\begin{equation*}
\mathcal{B}_{\bold d}:=
\begin{cases}
\qquad \qquad \big\{\,p_{{\bold d}}(X)\big\} & \ \text{if}\ \ \ \vert{\bold d}\vert=d,\\
\big\{X_{j}^{\bold d}\,\,:\,1\leq j\leq n\big\} & 
\ \text{if}\ \ \ 0\leq\big\vert{\bold d}\big\vert<d.
\end{cases}
\end{equation*} 
then we compute the Hilbert series as follows:
\begin{align*}
\mathcal{M}_{f}({\bold q})
&=\sum_{\big\vert{\bold d}\big\vert=d}\,{\bold q}^{\bold d}
+\sum_{0<\big\vert{\bold d}\big\vert<d}n\cdot{\bold q}^{\bold d}+1
=s_{d}({\bold q})+n\cdot \sum_{j=1}^{d-1}s_{j}({\bold q})+1.
\end{align*}
\end{proof}

\subsection{The Frobenius characteristic of $\mathcal{M}_{e_d}$}
For each $Y\subseteq\{1,2,\ldots,n\}$ we set
\begin{equation}\label{20151201800pm}
e_{k}(Y):=\sum_{\{B\subseteq Y,\ |B|=k\}}\left(\ \prod_{b\in B}x_{b}\right).
\end{equation}
Notice that when $Y=\{1,2,\ldots,n\}$ we get the usual elementary symmetric function $e_{k}({1,2,\ldots,n})=e_{k}({\bold x})$. For instance, if $\vert Y\vert=k$ then $\displaystyle{e_{k}(Y)=\prod_{a\in Y}x_{a}}$.
\begin{lemma}\label{lema201504081118}
For every permutation $\sigma\in\frak{S}_n$ the polynomial $e_{i}(Y)$ satisfies the identity:
\begin{equation*}
\sigma\cdot e_{i}(Y)=e_{i}(\sigma(Y)).
\end{equation*}
\end{lemma}
\noindent For each subset $Y$ of $\{1,2,\ldots,n\}$ we define the polynomial $e_{\bold d}(Y)\in\mathcal{R}_{n}^{(\ell)}$ by the equation
\begin{equation}\label{PolinomiosEDY}
e_{{\bold d}}(Y):=\sum_{\big\{B\subseteq Y\big\vert\ \vert B\vert=\vert{\bold d}\vert\big\}}\sum_{\big\{f:B\rightarrow[\ell]\big\vert\ \boldsymbol{\vert} f^{-1}(i)\boldsymbol{\vert}=d_i,\,\forall i\in[\ell]\big\}}\prod_{b\in B}x_{f(b),b}.
\end{equation}
\noindent Using Lemma \ref{201504211846pm} and formula \ref{201504201232c1} we get the following:
\begin{lemma}\label{PermutationETruncada}
For every $\sigma\in\frak{S}_n$ the polynomial $e_{{\bold d}}(Y)$ satisfies the identity:
\begin{equation}
\sigma\cdot e_{{\bold d}}(Y)=e_{{\bold d}}\big(\sigma(Y)\big).
\end{equation}
\end{lemma}

\begin{lemma}
For ${\bold d}\in\mathbb{N}^{\ell}$ and $S\subseteq
\{1,2,\ldots,n\}$ such that $|S|+|{\bold d}|=m$, 
we have the identity:
\begin{equation*}
e_{{\bold d}}\left(S^{c}\right)
=\frac{m!}{{\bold d}!}\,
E^{\bold d}\partial_{S}\big(e_{m}({\bold x}_1)\big),
\end{equation*}
where $S^{c}:=\{1,2,\ldots,n\}\backslash S$ is the complement of $S$,\ and $\partial_{S}$ is the sequence of partial derivatives $\partial_{i}$ with $i\in S$ applied to $e_{m}({\bold x}_1)$.
\end{lemma}
\begin{remark} 
In the following we use the notation $A^{c}:=\{1,2,\ldots,n\}\backslash A$,\ for any $A\subseteq \{1,2,\ldots,n\}$. 
\end{remark}

\begin{lemma}\label{Lema110035}
Let $d$ and $n$ be integers such that $n\geq d$. For each integer $i$ such that $0\leq i\leq d$, we define, a subspace $V_{i}$ of $\mathcal{R}_{n}$ as follows:
\begin{equation*}
V_i:=\mathbb{K}\big\{ e_i\left(T^{c}\right)\big\vert\ T\subseteq \{1,2,\ldots,n\},\ |T|+i=d\big\},
\end{equation*}
that is, $V_i$ is the $\mathbb{K}$-vector space generated by the polynomials $e_i\left(T^{c}\right)$ such that $T\subseteq \{1,2,\ldots,n\}$ and $|T|+i=d$. Then, for each $i$, a linear basis  $\mathcal{B}_{i}$ of $V_i$ is given by the following rule:
\begin{enumerate}
\item\ If \ $\displaystyle{0\leq i\leq\bigg\lfloor\frac{d}{2}\bigg\rfloor}$,\ then:
 \begin{equation*}
\mathcal{B}_{i}=\big\{e_{i}(S^{c})\,
\vert\ S\subseteq \{1,2,\ldots,n\},\ |S|+i=n\big\}.
\end{equation*}
\item\ If \ $\displaystyle{\bigg\lfloor\frac{d}{2}\bigg\rfloor<i\leq d}$,\ then:
\begin{equation*}
\mathcal{B}_{i}=\big\{e_{i}(T^{c})\,
\vert\ T\subseteq \{1,2,\ldots,n\},\ |T|+i=d\big\}.
\end{equation*}
\end{enumerate}
In particular,
\begin{equation*}
\dim(V_{i})=\begin{cases} \displaystyle{{n\choose i}} & \textrm{if}\ \ \ 0\leq i\ \leq\lfloor\frac{d}{2}\rfloor,\\
&\\
 \displaystyle{{n\choose d-i}} & \textrm{if}\ \ \ \lfloor\frac{d}{2}\rfloor< i \leq d.\end{cases}
\end{equation*}
also we have,\ $\dim(V_i)=\dim(V_{d-i})$,\ for all $i$ such that $0\leq i\leq d$.
\end{lemma}
\begin{proof}\textbf{First Case:}\ First we see that the following set 
\begin{equation*}
\big\{e_{i}(S^{c})\, \vert\ S\subseteq \{1,2,\ldots,n\},\ |S|+i=n\big\},
\end{equation*} 
is {linearly independent,} because it consists only of distinct monomials. So we only have to show two properties:
\begin{enumerate}
\item[(P1)] \ If \ $0\leq i\leq \lfloor\frac{d}{2}\rfloor$ \ then \ $\big\{e_{i}(S^{c})\, \vert\ S\subseteq\{1,2,\ldots,n\},\ |S|+i=n\big\} \subseteq V_{i},$
\item[(P2)] \ $\big\{e_{i}(S^{c})\, \vert\ S\subseteq \{1,2,\ldots,n\},\ |S|+i=n\big\}$ is a set of generators for $V_{i}$.
\end{enumerate}
Then, for each $i$ such that \ $0\leq 2i\leq d$ \ and all subsets $T$ such that $\vert T\vert=d-i$, we have:
\begin{align*}
&e_{i}(T^{c}):=\sum_{\big\{B:\ B\subseteq T^{c}, \ \vert B\vert=i\big\}}e_{i}(B)
=\sum_{\big\{S:\ T\subseteq S, \ \vert S\vert=n-i\big\}}e_{i}(S^{c}),
\end{align*}
thus, we have shown condition (P2) because:
\begin{equation*}
e_{i}(T^{c})=\sum_{\big\{S:\ T\subseteq S, \ \vert S\vert=n-i\big\}}e_{i}(S^{c}).
\end{equation*}
To show condition (P1) we go as follows: 
Let $T$ be a subset such that $\big\vert{T}\big\vert=d-i$, then
\begin{align*}
&e_{i}(T^{c})=\sum_{\{S\subseteq \{1,2,\ldots,n\}\,:\, \vert S\vert=n-i\}}\beta_{_{T^{c},S^{c}}}\cdot e_{i}(S^{c})
\end{align*}
where \ $\beta$ is the $\displaystyle{{n\choose d-i}\times{n\choose i}}$ matrix defined as:
\begin{equation*}
\beta_{_{T^{c},S^{c}}}=\left\{
\begin{array}{cc}
1 & \text{if}\ T\subseteq S,\\
0 & \text{otherwise.}
\end{array}
\right.=\left\{
\begin{array}{cc}
1 & \text{if}\ S^{c}\subseteq T^{c},\\
0 & \text{otherwise.}
\end{array}
\right.=\alpha_{_{S^{c},T^{c}}}
\end{equation*}
and indexed by subsets of $\{1,2,\ldots,n\}$ of cardinality $n-d+i$ for the rows and by subsets of $\{1,2,\ldots,n\}$ of cardinality $n-i$ for the columns. Here \ $\vert T\vert=d-i$,\ $\vert S\vert=n-i$, \ $\vert T\vert\leq\vert S\vert$. Since we have the conditions:
\begin{align*}
&\vert T^{c}\vert+\vert S^{c}\vert
=n-(d-i)+(n-(n-i))\\&=n-d+i+i=n-d+2i\leq n,\ \text{because}\ 0\leq 2i\leq d,
\end{align*}
\begin{equation*}
\vert S^{c}\vert\leq\vert T^{c}\vert\ \ \text{because}\ \
i\leq n-d+i\ \ \text{because}\ d\leq n,
\end{equation*}
the matrix $\beta$ is of full rank (see \cite{BruceSagan}, page 222). More precisely, 
\begin{itemize}
\item[(i)] The matrix ${\alpha}_{_{P,Q}}$
\begin{equation*}
\alpha_{_{P,Q}}=\left\{
\begin{array}{cc}
1 & \ P\subseteq Q,\\
0 & \ \text{otherwise},
\end{array}
\right.
\end{equation*}
indexed by subsets $P$ of $\{1,2,\ldots,n\}$ of cardinality $i$ for the rows and by subsets $Q$ of $\{1,2,\ldots,n\}$ of cardinality $n-d+i$ for the columns is a full rank matrix.
  \item[(ii)] The matrix $\beta_{_{Q,P}}=\alpha_{_{P,Q}}^{t}$, 
	that is:
\begin{equation*}
\beta_{_{P,Q}}=\left\{
\begin{array}{cc}
1 & \ P\supseteq Q,\\
0 & \ \text{otherwise},
\end{array}
\right.
\end{equation*}
is also a full rank matrix. 
\end{itemize}
Then, since $\beta_{_{Q,P}}$ has full rank and by the unimodality of binomial coefficients we have
\begin{equation*}
\displaystyle{\mathrm{rank}(\beta)={n\choose i}\leq {n \choose d-i}}.
\end{equation*}
So, there exists a \ $\displaystyle{{n\choose i}\times{n\choose d-i}}$ \ matrix such that:
\begin{equation*}
\displaystyle{\gamma\,\beta
=I_{_{{{n\choose i},{n \choose i}}}};}
\end{equation*}
that is, for each subset $S$ of cardinality $n-i$ we have:
\begin{equation*}
e_{i}(S^{c})=\sum_{\big\{T\subseteq \{1,2,\ldots,n\}:\ \vert T\vert=d-i\big\}}
\gamma_{_{S^{c},T^{c}}}\cdot e_{i}(T^{c}).
\end{equation*}
Hence, we have shown that condition (P1) holds.\\

\textbf{Second Case:} Suppose that $\lfloor\frac{d}{2}\rfloor\leq i\leq d$. We only have to show that the following set:
\begin{equation*}
\big\{e_{i}(T^{c})\,\,
\vert \ T\subseteq \{1,2,\ldots,n\},\ |T|+i=d\big\} 
\end{equation*}
is {linearly independent} if \ $d\leq 2i$, this is because, by definition, this set generates $V_i$. So, suppose that \ $d\leq 2i$, then, we have to verify following the condition:
\begin{equation*}
\sum_{\big\{T:\ \vert T\vert=m-i\big\}}\,\lambda_{T}\cdot e_{i}(T^{c})={\bold 0}({\bold x})\ \ \text{implies}\ \ \lambda_{T}=0,\ \forall\, T\ \text{s.t.}\ \vert T\vert+i=d,
\end{equation*}
By definition,\ \ $\displaystyle{e_{i}(T^{c})=\sum_{\big\{B:\ B\subseteq T^{c},\ \vert B\vert=i\big\}}\prod_{b\in B}b,}$ \ so
\begin{align*}
{\bold 0}({\bold x})&=\sum_{\big\{T:\ \vert T\vert=d-i\big\}}\,\lambda_{T}\cdot
\left(\sum_{\big\{B:\ B\subseteq T^{c},\ \vert B\vert=i\big\}}\prod_{b\in B}b\right)\\
&=\sum_{\big\{T:\ \vert T\vert=d-i\big\}}\sum_{\big\{B:\ B\subseteq T^{c},\ \vert B\vert=i\big\}}\,\lambda_{T}\prod_{b\in B}b\\
&=\sum_{\big\{B:\ \vert B\vert=i\big\}}
\sum_{\big\{T: T\subseteq B^{c},\ \vert T\vert=d-i\big\}}\lambda_{T}\prod_{b\in B}b\\
&=\sum_{\big\{B:\ \vert B\vert=i\big\}}
\left(\sum_{\big\{T: T\subseteq B^{c},\ \vert T\vert=d-i\big\}}\lambda_{T}\right)\prod_{b\in B}b.
\end{align*}
We see that for all $B\subseteq \{1,2,\ldots,n\}$, the product $\displaystyle{\prod_{b\in B}b}$ is a monomial in $\mathcal{R}_n$. Since any set of distinct monomials is linearly independent, we must have that, for any subset $B$ of $\{1,2,\ldots,n\}$ the identity below holds:
\begin{equation*}
\sum_{\big\{T: T\subseteq B^{c},\ \vert T\vert=d-i\big\}}\lambda_{T}=0,\ \ \forall B\subseteq \{1,2,\ldots,n\}.
\end{equation*}
Let $\varepsilon_{_{P,Q}}$ be  the $\displaystyle{{n\choose i}\times{n\choose d-i}}$ matrix such that:
\begin{equation*}
\varepsilon_{_{P,Q}}=\left\{
\begin{array}{cc}
1 & \ \ \text{if} \ \ P\supseteq Q,\\
0 & \ \text{otherwise}.
\end{array}
\right.
\end{equation*}
indexed by subsets $P$ of $\{1,2,\ldots,n\}$ of cardinality $n-i$ for rows and by subsets $Q$ of $\{1,2,\ldots,n\}$ of cardinality $d-i$ for columns. Then, we have:
\begin{equation*}
\sum_{\big\{T:\ \vert T\vert=d-i\big\}}\varepsilon_{_{B^{c},\, T}}\cdot\lambda_{T}=0,\ \ \forall B\subseteq \{1,2,\ldots,n\}.
\end{equation*}
The matrix $\varepsilon_{_{P,Q}}$ has full rank, because it is the transpose of a full rank matrix $\delta_{_{Q,P}}$, indexed by subsets $P$ of $\{1,2,\ldots,n\}$ of cardinality $n-i$ for rows and by subsets $Q$ of $\{1,2,\ldots,n\}$ of cardinality $m-i$ for columns and such that:
\begin{equation*}
\delta_{Q,P}=\left\{
\begin{array}{cc}
1 & \text{if}\ \ Q\subseteq P,\\
0 & \text{otherwise}.
\end{array}
\right.
\end{equation*}
The matrix \ $\delta_{_{Q,P}}$ has full rank (see \cite{BruceSagan} p.222.) because:\\
\begin{itemize}
\item[(i)] $\vert Q\vert\leq\vert P\vert$ \ because \ $d-i\leq n-i$ \ because \ $d\leq n$,
\item[(ii)] $\vert Q\vert+\vert P\vert\leq n$ \ because \ $(n-i)+(d-i)\leq n$ \ because \ $d\leq 2i.$ 
\end{itemize}
Since the matrix $\varepsilon_{_{P,Q}}$ of size $\displaystyle{{n\choose i}\times{n\choose d-i}}$ has full rank and $d\leq 2i$ we must have:
\begin{equation*}
\mathrm{rank}(\varepsilon)=\displaystyle{{n\choose d-i}}
\leq{n\choose i},
\end{equation*}
where the last inequality comes from unimodality. So we see that $\lambda_{T}=0$, for all subset $T$ of $\{1,2,\ldots,n\}$ with $\vert T\vert=d-i$. This implies that the set of polynomials $e_{i}(T^{c})$, over all subsets $T$ of $\{1,2,\ldots,n\}$ whose cardinality is $\vert T\vert=d-i$, is linearly independent. 
\end{proof}

\noindent For each vector ${\bold d}\in\mathbb{N}^{\ell}$ with $0\leq \big\vert{\bold d}\big\vert\leq d$ we denote by $V_{\bold d}$ the homogeneous component of $\mathcal{M}_{e_d}$ given by
\begin{equation}
V_{{\bold d}}:=\mathbb{K}\big\{\,e_{{\bold d}}\left(T^{c}\right)\,\big\vert\
T\subseteq\{1,2,\ldots,n\},\ |T|+|{\bold d}|=d\,\big\}. \\
\end{equation}
We can use Lemma \ref{201504211846pm} to verify the following result:
\begin{lemma}\label{2016010918h40min} 
Let $d\geq 1$ be an integer, ${\bold a}\in\mathbb{N}^{\ell}$, ${\bold b}\in\mathbb{N}^{\ell}$ and ${\bold d}\in\mathbb{N}^{\ell }$ with $\vert{\bold d}\vert=d$. The subspace 
$V_{{\bold d}}$ has the following properties: \\
\begin{enumerate}
\item $V_{{\bold d}}=E^{{\bold d}}\left(V_{(d,0,\dots,0)}\right)$, \\
\item $V_{(d,0,\dots,0)}=E_{{\bold d}}\left(V_{{\bold d}}\right)$, \\
\item $\dim\left(V_{{\bold d}}\right)=\dim\left(V_{(|{\bold d}|,0,\dots,0)}\right)$.\\
\item If $|{\bold a}|=|{\bold b}|$ then $\dim\left(V_{{\bold a}}\right)=\dim\left(V_{{\bold b}}\right)$. \\
\item If $|{\bold a}|+|{\bold b}|=d$ then $\dim\left(V_{{\bold a}}\right)=\dim\left(V_{{\bold b}}\right)$.\\
\end{enumerate}
Here, $V_{(i,0,0,\ldots)}=V_i:=\mathbb{K}\big\{ e_i\left(T^{c}\right)\big\vert\ T\subseteq \{1,2,\ldots,n\},\ |T|+i=m\big\}$. \\
\end{lemma}
\noindent It is easy to check (using formula \ref{201601111403}) the following rule:
\begin{lemma} 
For any $Y\subseteq\{1,2,\ldots,n\}$ the effect of general polarization operators $E_{i,k}^{(p)}$ on elementary McMahom symmetric functions  $e_{\bold d}(Y)$ (with ${\bold d}=(d_1,\ldots,d_{\ell})$) is given by the rule below:
\[ E_{i,k}^{(p)}\big(e_{\bold d}(Y)\big)=e_{\bold r}(Y),\ \text{whenever}\ p\leq\vert{\bold d}\vert, \]
where ${\bold r}=(r_1,\ldots,r_{\ell})$ is given by
\[r_{t}=\begin{cases}  d_t & \text{if} \quad t\neq i\ \text{and}\ t\neq k, \\ d_t+1 & \text{if} \quad t=i,\ \\ d_t-p & \text{if} \quad t=k.\end{cases}\]
\end{lemma}
Thus $E_{i,k}^{(p)}$ maps the homogeneous component $V_{\bold d}$ to $V_{\bold r}$. And this is the reason why the space described in Lemma \ref{Lema10023} below is closed by polarization operators.

\begin{lemma}\label{Lema10023}
Let $d>0$ be an integer. The polarization module generated by the polynomial $e_{d}({\bold x})$ decomposes as follows:
\begin{equation*}
\mathcal{M}_{e_d}
=\bigoplus_{0\leq |{\bold d}|\leq d}V_{{\bold d}}.
\end{equation*}
Furthermore, for each vector ${\bold d}\in\mathbb{N}^{\ell}$ with\ $0\leq |{\bold d}|\leq m$ a linear basis\ $\mathcal{B}_{{\bold d}}$ of the homogeneous component \ $V_{{\bold d}}$ \ is given by the rule below:
\begin{enumerate}
\item\ If \ $0\leq |{\bold d}|\leq\lfloor\frac{d}{2}\rfloor$ then,
\begin{equation*}
\mathcal{B}_{{\bold d}}=\big\{e_{{\bold d}}(T^{c})
\,\,\vert\ T\subseteq\{1,2,\ldots,n\},\ |T|+|{\bold d}|=n\big\},
\end{equation*}
\item\ If \ $\lfloor\frac{d}{2}\rfloor<|{\bold d}|\leq d$ then,
\begin{equation*}
\mathcal{B}_{{\bold d}}=\big\{e_{{\bold d}}(T^{c})
\,\,\vert\ T\subseteq\{1,2,\ldots,n\},\ |T|+|{\bold d}|=d\big\}.
\end{equation*}
\end{enumerate}
This implies that
\begin{equation*}
\dim(V_{{\bold d}})=\begin{cases} \displaystyle{{n\choose|{\bold d}|}} & \textrm{if}\ \ 0\leq |{\bold d}|\leq\lfloor\frac{d}{2}\rfloor,\\
&\\
 \displaystyle{{n\choose d-|{\bold d}|}} & \textrm{if}\ \ \lfloor\frac{d}{2}\rfloor< |{\bold d}|\leq d.\end{cases}
\end{equation*}
In particular, we see that \ $\dim(V_{\bold a})=\dim(V_{\bold b})$  whenever $\vert {\bold a} \vert+\vert {\bold b}\vert=d$.
\end{lemma}
\begin{proof}
The proof is very similar to the proof of Lemma \ref{Lema110035} because the basis $\mathcal{B}_{\bold d}$ of $\mathcal{V}_{\bold d}$ depends only on the value of $\vert{\bold d}\vert=d_{1}+\cdots+d_{\ell}$. Also, we can use Lemma \ref{2016010918h40min} to show that Lemma \ref{Lema110035} implies the desired result.
\end{proof}
\begin{remark}
Recall that the permutation representation $\mathbb{C}\big\{A\subseteq [n]\,:\,\vert{A}\vert=s\big\}$ of $\frak{S}_n$ is equal to $\mathrm{Ind}^{\mathfrak{S}_n}_{\mathfrak{S}_s\times\mathfrak{S}_{n-s}}$ trivially; by the Pieri rule, this is isomorphic to $\bigoplus_{j=0}^{s}\mathcal{S}^{(n-j,j)}$. It's easy to see directly that the $\mathfrak{S}_n$-module $V_{\bold d}$ is isomorphic to the above permutation representation whose character is $\sum_{j=0}^{s}{\chisotazo}^{(n-j,j)}$.
\end{remark}
\noindent It's easy to see directly that the $\mathfrak{S}_n$-module $V_{\bold d}$ is isomorphic to the above permutation representation. Then we have the following:
\begin{lemma}\label{Teorema38}
For each vector ${\bold d}\in\mathbb{N}^{\ell}$ such that\ $0\leq |{\bold d}|\leq d$ the value of the character ${\chisotazo}_{V_{{\bold d}}}$ of the homogeneous component\ $V_{{\bold d}}$ \ is given by the following rule:\\
\begin{enumerate}
\item\ If\ $0\leq |{\bold d}|\leq\lfloor\frac{d}{2}\rfloor$ then,
\begin{equation*}
{\chisotazo}_{V_{{\bold d}}}(\sigma)=\sum_{j=0}^{\big\vert{\bold d}\big\vert}{\chisotazo}^{(n-j,j)}(\sigma),
\end{equation*}
\item\ If \ $\lfloor\frac{d}{2}\rfloor<|{\bold d}|\leq d$ then, 
\begin{equation*}
{\chisotazo}_{V_{{\bold d}}}(\sigma)=\sum_{j=0}^{d-\big\vert{\bold d}\big\vert}{\chisotazo}^{(n-j,j)}(\sigma).
\end{equation*}
\end{enumerate}
\end{lemma}
\noindent So, the graded Frobenius characteristic of $V_{\bold d}$ is given by the rule
\begin{align*}
V_{\bold d}({\bold w})=
\begin{cases} 
\displaystyle{\sum_{j=0}^{\big\vert{\bold d}\big\vert}s_{n-j,j}({\bold w})} & 
\text{if} \ 0\leq |{\bold d}|\leq\lfloor\frac{d}{2}\rfloor,\\
\displaystyle{\sum_{j=0}^{d-\big\vert{\bold d}\big\vert}s_{n-j,j}({\bold w})} & \text{if} \ \lfloor\frac{d}{2}\rfloor<|{\bold d}|\leq d.
\end{cases}
\end{align*}
Using Lemma \ref{Teorema38} we show the third part of Theorem \ref{ProposicionUNO}, that is, if ${\bold d}\in\mathbb{N}^{\ell}$ is such that $\vert{\bold d}\vert=d$ then the graded Frobenius characteristic of $\mathcal{M}_{e_{ d}}$ is given by :\\
\begin{align*}
\mathcal{M}_{e_{{d}}}({\bold q},{\bold w})
&=\sum_{i=0}^{\lfloor\frac{d}{2}\rfloor}
\left(\sum_{j=0}^{i}{s}_{n-j,j}({\bold w})\right)s_{i}({\bold q})
+\sum_{i=\lfloor\frac{d}{2}\rfloor+1}^{d}\left(\sum_{j=0}^{d-i}
{s}_{n-j,j}({\bold w})\right)s_{i}({\bold q})\\
&=\left(\sum_{j=0}^{d}s_{j}({\bold q})\right)s_{n}({\bold w})
+\sum_{i=1}^{\lfloor{d/2}\rfloor}\left(\sum_{j=i}^{d-i}s_{j}({\bold q})\right)s_{n-i,i}({\bold w}), \ \forall\,d> 1
\end{align*}
\begin{corollary}
Let $d\geq 1$ be an integer. The graded Frobenius characteristic of 
$\mathcal{M}_{e_{\bold d}}$ is given by the formula
\begin{equation*}
\mathcal{M}_{e_{\bold d}}({\bold q},{\bold w})
=\sum_{i=0}^{\lfloor{d/2}\rfloor}\left(\sum_{j=i}^{d-i}s_{j}({\bf q})\right)s_{n-i,i}({\bold w}),
\end{equation*}
where $s_{n,0}({\bold w}):=s_{n}({\bold w})$.
\end{corollary}

\subsection{Constructing a basis for the modules $\mathcal{M}_{\mathcal A}$ and $\mathcal{M}_{\mathcal B}$}
Recall that we have the notation $X_{j}^{\bold d}:=x_{1j}^{d_1}\cdots x_{\ell j}^{d_{\ell}}$. The proof of the propositions below is similar to the proof of second part of Theorem \ref{ProposicionUNO}. We describe a basis for each component of the modules $\mathcal{M}_{\mathcal A}$ and $\mathcal{M}_{\mathcal B}$.
\begin{prop}\label{BaseFamiliaA}
An homogeneous basis for the polarization module $\mathcal{M}_{\mathcal{A}}$ is the set below
\[ B_{\mathcal{A}}:=\big\{X_{j}^{\bold d} \, : \,0\leq\vert{\bold d}\vert\leq d,\,1\leq{j}\leq{n} \big\}. \]
For each vector ${\bold d}\in\mathbb{N}^{\ell}$ such that $\vert{\bold d}\vert\leq d$, a basis for the homogeneous component of multidegree ${\bold d}$ is given by 
$B_{\mathcal{A},{\bold d}}:=\big\{X_{j}^{\bold d} \, : \,1\leq{j}\leq{n} \big\}$.
\end{prop}
\begin{prop}\label{BaseFamiliaB}
An homogeneous basis for the polarization module $\mathcal{M}_{\mathcal{B}}$ is the set
\[ B_{\mathcal{B}}:=\big\{X_{1}^{\bold d}-X_{j}^{\bold d}\,:\, \vert{\bold d}\vert=d,\, 2\leq j\leq n\big\}\bigcup\big\{X_{j}^{\bold b}\,:\, 0\leq\vert{\bold b}\vert<d,\, 1\leq j\leq n \big\}. \]
In this case, for each ${\bold d}\in\mathbb{N}^{\ell}$ such that $\vert{\bold d}\vert\leq d$, a basis for the component ${\bold V}_{\bold d}$ is given by the following rule:
\[ B_{\mathcal{B},{\bold d}}:=
\begin{cases}
X_{1}^{\bold d}-X_{j}^{\bold d} & \text{if}\ \ \ \vert{\bold d}\vert=d,\\
X_{j}^{\bold d} & \text{if}\ \ \ 0\leq\vert{\bold d}\vert<d.
\end{cases}
 \]
\end{prop}
The proof of Theorem \ref{FamiliasAB} follows easily from Proposition \ref{BaseFamiliaA} and \ref{BaseFamiliaB}.

\section{Auxiliary matrices and polynomials}

\noindent We start by introducing auxiliary matrices $H_n$ and $T_n$ in order to find explicit conditions under which the matrices $F_n$, $E_n$, $D_n$ and $G_n$ are of full rank. These considerations will give us equations to classify polarization modules generated by any given homogeneous symmetric polynomial of degree 2 or degree 3.
\subsection{The matrices $T_{n}$} 
Let $x,y,z,w,t$ be commuting variables. For each integer $n\geq2$ we define the matrices $T_n(x,y,z,w,t)$ of size $n\times n$ whose $(i,j)$ entry is given by: 
\[T_{n}(i,j):=
\begin{cases}
x & \text{if}\ \ i=j \ \ \text{and}\ \ i<n\ \text{and}\ j<n, \\
y & \text{if}\ \ i\neq j \ \ \text{and}\ \ i<n\ \text{and}\ j<n, \\
z & \text{if}\ \ i=n \ \ \text{and}\ \ j<n, \\
w & \text{if}\ \ i<n \ \ \text{and}\ \ j=n, \\
t & \text{if}\ \ i=n \ \ \text{and}\ \ j=n\  
\end{cases}
\]
\begin{exa}
When $n=4$ we have
\[
T_{4}(x,y,z,w,t)=
\left[ \begin {array}{cccc} x&y&y&w\\ \noalign{\medskip}y&x&y&w
\\ \noalign{\medskip}y&y&x&w\\ \noalign{\medskip}z&z&z&t\end {array}
 \right].
\]
\end{exa}
\begin{lemma}\label{SuperDeterminante}
The determinant of $T_{n}(x,y,z,w,t)$ is given by the formula:
\begin{equation} 
\mathrm{det}\big(T_{n}(x,y,z,w,t)\big)=(x-y)^{n-2}\big(t\cdot (x+(n-2)y)-(n-1)wz\big). 
\end{equation}
\end{lemma}
\subsection{The matrices $H_n$}
Also, we consider the matrices $H_{n}(x,y,z)$ of size $n\times(n-1)$ where the $(i,j)$ entry 
is given by
\begin{equation*}
H_{n}(i,j):=
\begin{cases}
x &\ \text{if}\ \ i=j \ \ \text{and} \ \ i<n \ \ \text{and} \ \ j<n,\\
y &\ \text{if}\ \ i\neq j \ \ \text{and} \ \ i<n \ \ \text{and} \ \ j<n,\\
z &\ \text{if}\ \ i=n \ \ \text{and} \ \ j<n.
\end{cases}
\end{equation*}
Notice that
\begin{align*} 
&H_{n}^{t}(x,y,z)H_{n}(x,y,z)=T_{n-1}\big(\alpha_n,\beta_n,\beta_n,\beta_n,\alpha_n\big). 
\end{align*} 
where \ $\alpha_{n}(x,y,z):=x^2+(n-2)y^2+z^2$, \ $\beta_{n}(x,y,z):=2xy+(n-3)y^2+z^2$. Then we have the following result
\begin{lemma}
The determinant of the matrix $H_{n}^{t}H_{n}$ is given by
\begin{align*}
&\mathrm{det}\big(H_{n}^{t}H_{n}\big)
=\mathrm{det}\big(T_{n-1}\big(\alpha_n,\beta_n,\beta_n,\beta_n,\alpha_n\big)\big)\\
&=(x-y)^{2(n-2)}\big((x+(n-2)y)^2+(n-1)z^2\big). 
\end{align*}
\end{lemma}

\subsection{The polynomials $P_n$, $Q_n$ and $R_n$}

\label{PolinomiosPQRA}
We introduce the following auxiliary polynomials that depend on the number of variables $n$ of the vector ${\bold x}_1:=(x_{11},\ldots,x_{1n})$. 
 The polynomials  $P_n$, $Q_n$ and $R_n$ in the variables $a,b,c$ are defined as: 
\[P_n(a,b,c):=12\cdot ab+6(n-2)\cdot ac-4(n-1)\cdot b^2,\]
 \[Q_n(a,b,c):=9\cdot a^2-6\cdot ab+(4n-7)\cdot b^2-4(n-2)\cdot bc+(n-2)\cdot c^2,\]
\begin{align*}
R_n(a,b,c)&:=9\cdot a^2+6(n-1)\cdot ab+(n-1)(n+7)\cdot b^2+4(n-1)(n-2)\cdot bc\\
&+(n-2){n-1\choose 2}\cdot c^2.
\end{align*}
\begin{align*}
&A_n(a,b,c):=81\cdot {a}^{4}-54\cdot{a}^{3}b
+\left( 18\,{n}^{2}+18\,n-63 \right)\cdot {a}^{2}{b}^{2}\\
&+18\left( n-2 \right)  \left( {n}^{2}-2\,n-1 \right)\cdot {a}^{2}bc\\
&+\frac{9}{2}\cdot n\left(n-2\right)\left( {n}^{2}-4\,n+5\right)\cdot {a}^{2}{c}^{2}
-12 \left( n-1 \right)  \left( {n}^{2}-2\,n+2 \right)\cdot a{b}^{3}\\
&-12 \left( n-1 \right) ^{2}{n-1\choose 2}\cdot a{b}^{2}c
+2 \left(n-1\right)  \left( {n}^{3}-3\,{n}^{2}+7\,n-8 \right)\cdot {b}^{4}\\
&-8 \left( n-2 \right)  \left( n-1 \right)\cdot {b}^{3}c
+2 \left( n-2 \right)\left( n-1 \right)\cdot {b}^{2}{c}^{2}.
\end{align*}
\noindent We explain some properties of these polynomials in order to prove Theorem \ref{TeoremaExcepcionGrado3}
\begin{lemma}\label{Lema136}
Let $n\geq 3$. The only point $[a:b:c]\in\mathbb{RP}^{2}$ that satisfies $Q_{n}(a,b,c)=0$ is $[a:b:c]=[1:3:6]$. When $n=2$ the roots of 
$Q_{2}(a,b,c)$ are given by the condition $b=3a$.
\end{lemma}
\begin{proof}
Completing the square we obtain
\begin{align*}
&9\,{a}^{2}-6\,ab+ \left( 4\,n-7 \right) {b}^{2}+ \left( -4\,n+8
 \right) bc+ \left( n-2 \right) {c}^{2}\\
&=\left( n-2 \right)  \left( -2\,b+c \right) ^{2}+9\, \left( a-\frac{b}{3}
 \right) ^{2}, 
 \end{align*}
Since $n>2$ we get $\displaystyle{Q_{n}(a,b,c)=\left( n-2 \right)  \left( -2\,b+c \right) ^{2}+9\, \left( a-\frac{b}{3}
 \right) ^{2}}=0$ \ if and only if \ $c=2b$ and $b=3a$ (because $n\geq 3$ and we work on $\mathbb{R}$). So, $b=3a$ and $c=6a$, that is, 
 $[a:b:c]=[a:3a:6a]=[1:3:6]$. When $n=2$, $Q_{2}(a,b,c)=9a^2-6ab+b^2=(b-3a)^2$. So, $Q_{2}(a,b,c)=0$ iff $b=3a$.
\end{proof}
\begin{remark}
We observe that for every $n\geq 2$ the point $[a:b:c]=[1:3:6]$ is also a root of the polynomial $P_n(a,b,c)$.
\end{remark}
\begin{lemma}\label{RaicesDeRn}
The only root of  $R_n(a,b,c)$ is given by the point $[a:b:c]\in\mathbb{RP}^{2}$ such that
\[a=k\cdot (n-1){n-1 \choose 2},\ \ b=k\cdot(-3){n-1\choose 2},\ \ c=k\cdot 6(n-1),\]
for any $k\in\mathbb{R}-\{0\}$. 
\end{lemma}
\begin{proof} 
Again, we complete squares
\begin{align*}
&R_{n}(a,b,c)\\
&=9\cdot a^2+6(n-1)\cdot ab+(n-1)(n+7)\cdot b^2+4(n-1)(n-2)\cdot bc\\
&+(n-2){n-1\choose 2}\cdot c^2 \\
&=\left( n-2 \right) {n-1\choose 2} \left( c+2\,{\frac { \left( n-1
 \right) b}{{n-1\choose 2}}} \right)^{2}
+9\,\left(a+\frac{(n-1)b}{3}\right)^{2}.
\end{align*}
So\ $R_{n}(a,b,c)=0$ implies that $\displaystyle{a=\frac{-(n-1)b}{3}}$\ and\ $\displaystyle{c=-\frac{2(n-1)b}{{n-1\choose 2}}}$,\ for every $b\neq 0$. We can write $b$ in the form 
$\displaystyle{b=-3{n-1\choose 2}\cdot{k}}$ and we get the desired result.
\end{proof}
\begin{lemma}\label{LemaRnAn}
Let $n\geq 2$. The roots of $R_{n}(a,b,c)$ cannot be roots of the polynomial $A_{n}(a,b,c)$.
\end{lemma}
\begin{proof}
It's enough to compute the value of $A_{n}$ at the root given in Lemma \ref{RaicesDeRn}. Indeed, the 
expression below is always positive, for every $k\neq 0$ and $n\geq 2$:
\begin{align*} 
&A\left(k(n-1){n-1 \choose 2},-3k{n-1\choose 2},6k(n-1)\right)\\
&={\frac {81\,{n}^{2}{k}^{4} \left( n+2 \right)  \left( n+1 \right)
\left( n-2 \right) ^{3} \left( n-1 \right) ^{5}}{16}}.
\end{align*}
\end{proof}
\noindent In the next section we introduce auxiliary matrices $F_n$,$E_n$,$D_n$ and $G_n$ (depending on the number of columns $n$ of the matrix $X$) in order to find conditions that will lead us to classify polarization modules in degrees 2 and 3.
\subsection{The matrices $F_n$}
For $n\ge 3$, let $F_n$ be the ${n\choose 2}\times n$ matrix whose rows are all distinct permutations of the vector 
$(2b,2b,c,\ldots,c)\in\mathbb{R}^n$. For instance,
\[F_3:=\left[ \begin {array}{ccc} 2\,b&2\,b&c\\ \noalign{\medskip}2\,b&c&2\,
b\\ \noalign{\medskip}c&2\,b&2\,b\end {array} \right],\ \
F_4:=\left[ \begin {array}{cccc} 2\,b&2\,b&c&c\\ \noalign{\medskip}2\,b&c&
2\,b&c\\ \noalign{\medskip}2\,b&c&c&2\,b\\ \noalign{\medskip}c&2\,b&2
\,b&c\\ \noalign{\medskip}c&2\,b&c&2\,b\\ \noalign{\medskip}c&c&2\,b&2
\,b\end {array} \right].\ \ 
\]

\subsection{The matrices $E_n$}

For $n\geq 2$, let $E_n$ be the ${n+1\choose 2}\times(n+1)$ matrix where the entries are defined by the function:
\[E_{n}(i,j)=
\begin{cases}
3a &\ \ \text{if} \ \ i=j\ \text{and}\ 1\leq i\leq n, \\
b &\ \ \text{if} \ \ i\neq j\ \text{and}\ 1\leq i\leq n\ \text{and}\ 1\leq j<n+1,\\
6a &\ \ \text{if} \ \  1\leq i\leq n\ \text{and}\ j=n+1, \\
4b &\ \ \text{if} \ \  n<i \ \text{and} \ j=n+1,\\
F_n(i-n,j) &\ \ \text{if} \ \  n<i\ \text{and}\ 1\leq j<n. 
\end{cases}.\]
$$
E_3:=\left[ \begin {array}{cccc} 3\,a&b&b&6\,a\\ \noalign{\medskip}b&3\,a&
b&6\,a\\ \noalign{\medskip}b&b&3\,a&6\,a\\ \noalign{\medskip}2\,b&2\,b
&c&4\,b\\ \noalign{\medskip}2\,b&c&2\,b&4\,b\\ \noalign{\medskip}c&2\,
b&2\,b&4\,b\end {array} \right].$$

\noindent In the following, we will show that, if $[a:b:c]\in\mathbb{RP}^{2}$ is such that $$[a:b:c]\neq[1:3:6],$$ then the matrix $E_n$ has at least rank $n$. In other words we will prove Theorem \ref{TeoremaExcepcionGrado3}. The columns of $E_{n}(a,b,c)$ correspond to applying $\partial_{1,i}$ to $a\,m_{3}+b\,m_{2,1}+c\,m_{1,1,1}$, and the rows to $x_{1,i}^{2}$ and $x_{1,i}x_{1,j}$.


\subsection{The matrices $D_n$}

We define the matrix $D_n$ as the result of dropping the last column of $E_n$. For instance,

$$D_3:= \left[ \begin {array}{ccc} 3\,a&b&b\\ \noalign{\medskip}b&3\,a&b
\\ \noalign{\medskip}b&b&3\,a\\ \noalign{\medskip}2\,b&2\,b&c
\\ \noalign{\medskip}2\,b&c&2\,b\\ \noalign{\medskip}c&2\,b&2\,b
\end {array} \right], \ \
D_4:= \left[ \begin {array}{cccc} 3\,a&b&b&b\\ \noalign{\medskip}b&3\,a&b&b
\\ \noalign{\medskip}b&b&3\,a&b\\ \noalign{\medskip}b&b&b&3\,a
\\ \noalign{\medskip}2\,b&2\,b&c&c\\ \noalign{\medskip}2\,b&c&2\,b&c
\\ \noalign{\medskip}2\,b&c&c&2\,b\\ \noalign{\medskip}c&2\,b&2\,b&c
\\ \noalign{\medskip}c&2\,b&c&2\,b\\ \noalign{\medskip}c&c&2\,b&2\,b
\end {array} \right].$$ 

Both Lemmas below are consequences of Lemma 
\ref{SuperDeterminante} because the matrices
$E_n^{t}E_n$ and $D_n^{t}D_n$ are both matrices of the form $T_n$.

\begin{lemma}\label{LemaEn}
Let $a,b,c$ be real numbers not all zero. For each integer $n\geq 3$, the determinant of the matrix $E_n^{t}E_n$ is given by:
\[{\mathrm{det}(E_n^{t}E_n)={n\choose 2}P_{n}(a,b,c)^2Q_n(a,b,c)^{n-1}.}\]
\end{lemma}
\begin{lemma}\label{LemaDn}
Let $a,b,c$ be real numbers not all zero. For each integer $n\geq 3$, 
the determinant of the matrix $D_n^{t}D_n$ is given by:
\[{\mathrm{det}(D_n^{t}D_n)=R_n(a,b,c)Q_n(a,b,c)^{n-1}.}\]
\end{lemma}

\subsection{The matrices $G_n$}

The matrices $G_n$ are defined as the ${n+1\choose 2}\times n$ matrices obtained by replacing the last column of $D_n$ by the last column of $E_n$. The following lemma help us to find a linear basis of the real span of all first order partial derivatives of $f$ and the polynomial $E_{1,1}^{(2)}(f)$ ($\deg(f)=3$), when $f$ is a $n$-exception. To see this we must observe that the matrix $G_n^{t}G_n$ is of the form $T_n$ and then we have the result below
\begin{lemma}\label{LemaGn}
Let $a,b,c$ be real numbers not all zero. For each integer $n\geq 3$, the determinant of the matrix $G_n^{t}G_n$ is given by 
\[\mathrm{det}\left(G_n^{t}G_n\right)=A_n(a,b,c)\,Q_n(a,b,c)^{n-1}.\]
\end{lemma}

\section{Proof of Theorems \ref{TeoremaExcepcionGrado3} and \ref{EcuacionesExcepciones}}\label{Seccion201601151608}

\subsection{Proof of Theorem \ref{TeoremaExcepcionGrado3}}

Let $f$ be an homogeneous symmetric polynomial of degree 3, in the variables $x_{11},\ldots,x_{1n}$, such that $[f]\neq[1:3:6]$. Then, we have to show that the dimension of the real span of the set $\big\{\partial_{11}f,\ldots,\partial_{1n}(f),E_{1,1}^{(2)}(f)\big\}$ is at least $n$. 

Suppose that $[f]=[a:b:c]$. We have two cases:

\begin{enumerate}
\item If the matrix $D_{n}(a,b,c)$ has full rank then the set 
$\big\{\partial_{11}f,\ldots,\partial_{1n}(f)\big\}$ is a linearly independent subset of 
$\{\partial_{11}f,\ldots,\partial_{1n}(f),E_{1,1}^{(2)}(f)\}$, therefore, the dimension of  
$\mathbb{R}\left\{\partial_{11}f,\ldots,\partial_{1n}(f),
E_{1,1}^{(2)}(f)\right\}$ is at least $n$.

\item Otherwise, we use Lemma \ref{LemaDn}. The matrix
$D_{n}(a,b,c)$ does \textbf{not} have full rank if and only if   
\[R_{n}(a,b,c)\,Q_{n}(a,b,c)^{n-1}=0.\]
Lemma \ref{Lema136} implies that $Q_{n}(a,b,c)\neq 0$ because $[f]\neq[1:3:6]$. So, $R_{n}(a,b,c)=0$. In this case, we are going to show that
a linear basis of the real span of the set 
$\big\{\partial_{11}(f),\ldots,\partial_{1n}(f),E_{1,1}^{(2)}(f)\big\}$ is the following set of cardinality $n$:
\[\big\{\partial_{1,1}(f),\ldots,\partial_{1,{n-1}}(f),E_{1,1}^{(2)}(f)\big\}.\]
Indeed, the linear system $\displaystyle{\sum_{j=1}^{n-1}\gamma_{j}\,\partial_{1,j}(f)+\beta\,E_{1,1}^{(2)}(f)={\bold 0}}$ written in terms of matrices is equivalent to \[ G_{n}(a,b,c)\left(\begin{array}{c} \gamma_{1}\\ \vdots \\ \gamma_{n-1}\\ \beta \end{array}\right)={\bold 0}. \]
Now Lemma \ref{LemaRnAn} asserts that
\[R_{n}(a,b,c)=0 \ \Longrightarrow \ A_{n}(a,b,c)\neq 0.\]
So, the determinant of $G_{n}^{t}G_n$ is not zero, because Lemma \ref{LemaGn} states that
\[\mathrm{det}(G_{n}^{t}G_n)=A_{n}(a,b,c)\,Q_{n}(a,b,c)^{n-1}\neq 0.\]
Thus, the dimension of the real span of the set below 
$$\big\{\partial_{1,1}(f),\ldots,\partial_{1,{n-1}}(f),E_{1,1}^{(2)}(f)\big\}$$ is exactly $n$. This implies
\[ \dim\left(\mathbb{R}\left\{\partial_{1,1}(f),\ldots,\partial_{1,{n-1}}(f),\partial_{1,n}(f),E_{1,1}^{(2)}(f)\right\}\right)\geq n.\]
\qquad \qquad \qquad \qquad\qquad \qquad \qquad \qquad \qquad \qquad \qquad \qquad \qquad \qquad \qquad \qedsymbol
\end{enumerate}

\subsection{Proof of Theorem \ref{EcuacionesExcepciones}}

Let $n\geq 3$. First, we observe that the linear system
\[\lambda_{1}\,\partial_{11}(f)+\cdots+\lambda_{n}\,\partial_{1n}(f)
+\mu\,E_{1,1}^{(2)}(f)={\bold 0},\]
written in terms of matrices, is equivalent to the matrix equation:
\[ E_{n}(a,b,c)\left(\begin{array}{c} \lambda_1 \\ \vdots \\ \lambda_n \\ \mu \end{array}\right)={\bold 0},\]
where the unknowns are $\lambda_{1},\ldots,\lambda_n,\mu$. Suppose that $[a:b:c]$ is a $n$-exception, then the dimension of the real span of the following set of polynomials: 
$$\{\partial_{1,1}(f),\ldots,\partial_{1,{n}}(f),E_{1,1}^{(2)}(f)\}$$ is exactly $n$ (this implies that $[a:b:c]\neq[1:3:6]$). Therefore, the matrix $E_n(a,b,c)$ does not have full rank (it has $n+1$ columns) and so 
\[\mathrm{det}\big(E_{n}^{t}(a,b,c)E_n(a,b,c)\big)=P_{n}(a,b,c)Q_{n}(a,b,c)^{n-1}=0.\] 
Recall that $[a:b:c]\neq[1:3:6]$ implies $Q_n(a,b,c)\neq 0$. Thus, $P_{n}(a,b,c)=0$.

Conversely, if $P_{n}(a,b,c)=0$ then $\mathrm{det}\big(E_{n}^{t}(a,b,c)E_n(a,b,c)\big)=0$. This implies that $E_{n}$ does not have full rank. So Theorem \ref{TeoremaExcepcionGrado3} tells us that the dimension (over $\mathbb{R}$) of the span of $\partial_{1,1}(f),\ldots,\partial_{1,{n}}(f),E_{1,1}^{(2)}(f)$ is at least $n$. Hence the maximal dimension is $n+1$, but the matrix $E_n$ does not have full rank then the dimension cannot be $n+1$. So we have 
\[\dim\left(\mathbb{R}\left\{\partial_{1,1}(f),\ldots,\partial_{1,{n}}(f),E_{1,1}^{(2)}(f)\right\}\right)=n.\]
\qquad \qquad \qquad \qquad \qquad \qquad \qquad \qquad \qquad \qquad \qquad \qquad \qquad  \qquad \qquad \qquad
\qedsymbol
\section{Proof of Theorem \ref{ProposicionDOS}}
In this section we compute an explicit linear basis of the polarization module $\mathcal{M}_f$ generated by any homogeneous symmetric polynomial $f$ of degree 2. Let $f({\bold x}_1)$ be a such polynomial then; we write $f$ in the monomial basis as follows:
\begin{equation}\label{EcuacionPolinomioSimetricoGrado2BaseMonomial201601111157}
f({\bold x}_1):=a\cdot m_{2}({\bold x}_1)+b\cdot m_{1,1}({\bold x}_1),\ \text{with}\ a\in\mathbb{R},\ b\in\mathbb{R},
\end{equation}
that is 
\begin{equation*}
f(x_{11},\dots,x_{1n})=a\cdot\sum_{j=1}^{n}x_{1j}^2+b\cdot\sum_{1\leq p<q\leq n}x_{1p}x_{1q}.
\end{equation*}
Then, the polynomials in $\mathcal{M}_f$ have the form
\begin{equation*}
\partial_{1j}(f):=\frac{\partial f}{\partial x_{1j}}=2a\, x_{1j}+b\, \sum_{r\neq j}x_{1r}=(2a-b)\, x_{1j}+b\, \sum_{r=1}^{n}x_{1r}
\end{equation*}
\begin{align*}
E_{i,1}(f)&=2a\,\sum_{s=1}^{n}x_{1s}x_{is}+b\,\sum_{p\neq q}x_{1p}x_{iq}\\
&=(2a-b)\,\sum_{s=1}^{n}x_{1s}x_{is}+b\,\sum_{p,q}x_{1p}x_{iq},
\end{align*}
\begin{equation*}
E_{i,1}\partial_{1,j}(f)=(2a-b)\,x_{ij}+b\,\sum_{r=1}^{n}x_{ir},
\end{equation*}
\begin{align*}
E_{s,1}E_{t,1}(f)&=2a\,\sum_{j=1}^{n}x_{s j}x_{t j}
+b\cdot\sum_{p\neq q}x_{sp}x_{tq}
=(2a-b)\,\sum_{j=1}^{n}x_{sj}x_{tj}+b\,\sum_{p,q}x_{sp}x_{tq}.
\end{align*}
In particular, we get 
\begin{equation*}
E_{1,1}(f)=2\,f, \quad E_{\alpha,1}E_{\alpha,1}(f)=2\,f(x_{\alpha 1},\dots,x_{\alpha n}),\quad
\end{equation*}
\begin{equation*}
E_{\alpha,1}^{(2)}(f)=2\,a\,\sum_{j=1}^{n}x_{\alpha j}=E_{\alpha,1}E_{1,1}^{(2)}.
\end{equation*}
\subsubsection{Construction of degree 2 homogeneous components}\label{seccion201601131318}
We start by setting some convenient notations. We denote by\ ${\bold e}_{i}\in\mathbb{N}^{\ell}$ the vector having 1 in position $i$ and 0 elsewhere. For each pair of positive integers $(s,t)$ such that $1\leq s\leq t\leq \ell$, we write ${\bold e}_{s,t}$ to denote the vector\ ${\bold e}_{s,t}={\bf e}_{s}+{\bf e}_{t}$. Also, we write\ ${\bold e}_{p,q,r}:={\bold e}_{s}+{\bold e}_{q}+{\bold e}_{r}$.

\noindent The homogeneous components of $\mathcal{M}_f$ are $\mathcal{V}_{\bold 0}=\mathbb{R}$, 
\begin{equation*}
\mathcal{V}_{{\bold e}_{i}}=\mathbb{R}\left\{\left\{E_{i,1}\partial_{1,j}f\,:\,1\leq j\leq n\right\}\bigcup\left\{E_{i,1}^{(2)}(f)\right\}\right\}
\end{equation*} 
and $\mathcal{V}_{{\bold e}_{s,t}}:=\mathbb{R}\big\{E_{s,1}E_{t,1}f\big\}$. More precisely, we have for every $i$ such that $1\leq i\leq \ell$ that
\begin{align*}
\mathcal{V}_{{\bold e}_{i}}&=\mathbb{R}\left\{\left\{(2a-b)x_{ij}+b\sum_{r=1}^{n}x_{ir}\ \middle\vert\ 1\leq j\leq n\right\}\bigcup\left\{2a\,\sum_{j=1}^{n}x_{ij}\right\}\right\}, 
\end{align*}
For total degree 2, the homogenenous components are:
\begin{align*}
\mathcal{V}_{{\bold e}_{s,t}}&=\mathbb{R}\left\{ 
2a\sum_{j=1}^{n}x_{s j}x_{t j}+b\sum_{p\neq q}x_{s p}x_{t q} \right\}=\mathbb{R}\big\{E_{t,1}E_{s,1}(f)\big\}, \ \text{with}\ 1\leq s\leq t \leq \ell. 
\end{align*}
\begin{prop}\label{201601111410}
Let $f\in\mathbb{R}[{\bold x}_1]$ be an homogenenous symmetric polynomial of degree 2 written in the monomial basis as formula \ref{EcuacionPolinomioSimetricoGrado2BaseMonomial201601111157} 
Then the vector space 
\[ \mathbb{R}\boldsymbol{\oplus}\bigoplus_{1\leq i\leq\ell}\mathcal{V}_{{\bf e}_{i}}\boldsymbol{\oplus}\bigoplus_{1\leq s,t\leq\ell}\mathcal{V}_{{\bf e}_{s,t}} \] 
is closed by polarization operators $E_{i,k}^{(p)}$ and partial derivatives. Furthermore, this vector space coincides with the polarization module generated by $f$.
\end{prop}
\begin{proof}
Clearly, if $p>1$, then $E_{i,k}^{(p)}\left(\sum_{j}x_{sj}\right)=0$ and $E_{i,k}\left(\sum_{j}x_{sj}\right)=\sum_{j}x_{ij}$. Since $E_{i,k}$ is a derivation on $\mathcal{R}_{n}^{(\ell)}$, it satisfies the product rule of derivatives and then we have the following identities:
\begin{align}
E_{i,k}\left(\sum_{j=1}^{n}x_{sj}x_{tj}\right)&=\delta_{k,s}\sum_{j=1}^{n}x_{ij}x_{tj}+\delta_{k,t}\sum_{j=1}^{n}x_{ij}x_{sj}.\label{201601111924A} \\
E_{i,k}\left(\sum_{u\neq v}x_{su}x_{tv}\right)&=\delta_{k,s}\sum_{u\neq v}x_{iu}x_{tv}+\delta_{k,t}\sum_{u\neq v}x_{su}x_{iv}.
\label{201601111924B}
\end{align}
More generally, if $s\neq t$ the effect of $E_{i,k}$ on the polynomial $\sum_{\alpha,\beta}x_{s\alpha}x_{t\beta}$ is given by
\begin{equation}\label{201601111453}
 E_{i,k}\left(\sum_{\alpha,\beta}x_{s\alpha}x_{t\beta}\right)=\delta_{k,s}\sum_{\alpha,\beta}x_{i\alpha}x_{t\beta}+\delta_{k,t}\sum_{\alpha,\beta}x_{s\alpha}x_{i\beta}. 
\end{equation}
where the sum is over any set of pairs $(\alpha,\beta)$ with $1\leq\alpha,\beta \leq\ell $. Also, we have the identity
\[ E_{i,k}^{(2)}\left(\sum_{j=1}^{n}x_{sj}^{2}\right)=2\delta_{k,s}\sum_{j=1}^{n}x_{ij}. \]
and
\[ E_{\alpha,\beta}^{(2)}\left(\sum_{j=1}^{n}x_{sj}x_{tj}\right)=0,\ \ \textrm{if} \ s\neq t. \]
This imply that
\[ E_{\alpha,\beta}^{(2)}\left(2a\sum_{j=1}^{n}x_{s j}x_{t j}+b\sum_{p\neq q}x_{s p}x_{t q}\right)
=\begin{cases} 
\displaystyle{4a\delta_{\beta,s}\sum_{j=1}^{n}x_{\alpha j}} & \text{if}\ s=t,\\
0 & \text{otherwise.}
\end{cases}
\]
Then, for a homogeneous component of degree 1 we can easily check that:
\[ E_{\alpha,\beta}\big(\mathcal{V}_{{\bold e}_{i}}\big)=\delta_{\beta,i}\mathcal{V}_{{\bold e}_{\alpha}}. \]
Any homogeneous component of total degree 2 is closed under polarizations $E_{i,k}^{(p)}$, in the following sense:
\[ E_{\alpha,\beta}\big(\mathcal{V}_{{\bold e}_{s,t}}\big)\subseteq \delta_{\beta,t}\mathcal{V}_{{\bold e}_{\alpha,s}}+\delta_{\beta,s}\mathcal{V}_{{\bold e}_{\alpha,t}}.\]
\[ E_{\alpha,\beta}^{(2)}\left(\mathcal{V}_{{\bold e}_{s,t}}\right)\subseteq\delta_{\beta,t}\mathcal{V}_{{\bold e}_{\alpha}}.  \]
this is true since:
\[ E_{\alpha,\beta}\bigg(E_{t,1}E_{s,1}(f)\bigg)=\delta_{\beta,t}E_{\alpha,1}E_{s,1}(f)+\delta_{\beta,s}E_{t,1}E_{\alpha,1}(f). \]
and for $p>1$ we have 
\[ E_{\alpha,\beta}^{(p)}\bigg(E_{t,1}E_{s,1}(f)\bigg)
=
\begin{cases} 
0 & \text{if}\  p>2, \\ 
0 & \text{if}\  p=2 \ \text{and}\ s\neq t, \\
2\delta_{\beta,t}\, E_{\alpha,1}^{(2)}(f) & \text{if}\  p=2 \ \text{and}\ s=t,  \\
\end{cases}\]
Every homogenenous component of total degree 2 is closes under derivatives because: 
\[ \partial_{sj}\bigg(E_{s,1}E_{t,1}(f)\bigg)\subseteq E_{t,1}\partial_{1j}(f)\in\mathcal{V}_{{\bold e}_{t}}. \]
\[  \partial_{sj}\bigg(E_{s,1}^{2}(f)\bigg)=E_{s,1}\partial_{1j}(f)\in\mathcal{V}_{{\bold e}_{s}}. \] 
To show the second assertion we use last considerations to see that  $$\mathcal{M}_f\subseteq\mathbb{R}\boldsymbol{\oplus}\bigoplus_{1\leq i\leq\ell}\mathcal{V}_{{\bf e}_{i}}\boldsymbol{\oplus}\bigoplus_{1\leq s,t\leq\ell}\mathcal{V}_{{\bf e}_{s,t}}.$$ Conversely, we notice that all homogenenous components $\mathcal{V}_{\bold d}$ are spanned by elements in $\mathcal{M}_f$ and so, we have:
\[  \mathcal{M}_f=\mathbb{R}\boldsymbol{\oplus}\bigoplus_{1\leq i\leq\ell}\mathcal{V}_{{\bf e}_{i}}\boldsymbol{\oplus}\bigoplus_{1\leq s,t\leq\ell}\mathcal{V}_{{\bf e}_{s,t}}.\] 
\end{proof}
\begin{lemma}\label{LemaBaseGrado2}
Let $n\geq 2$. Let $f\in\mathbb{R}[{\bold x}_1]$ be a homogeneous symmetric polynomial of degree 2. Suppose that $f$ is given in the monomial basis as formula \ref{EcuacionPolinomioSimetricoGrado2BaseMonomial201601111157}
then we have the following properties:  
\begin{enumerate}
\item If $a\neq 0$ and $b\neq 2a$ then, for each integer $i$ such that $1\leq i\leq \ell$, the set 
\begin{equation*}
\bigg\{E_{i,1}\partial_{1j}(f)\,\bigg\vert\,\ 1\leq j\leq n-1\bigg\}\bigcup\left\{E_{i,1}^{(2)}(f)\right\}
\end{equation*}
is linearly independent.
\item Also, if $b\neq 2a$ then, for any integer $i$ such that $1\leq i\leq \ell$ the set
\begin{equation*}
\bigg\{E_{i,1}\partial_{1j}(f)\,\bigg\vert\,\ 1\leq j\leq n-1\bigg\}
\end{equation*}
is linearly independent. 
\end{enumerate} 
\end{lemma}
\begin{proof}
To prove first property, let $M(a,b,n)$ be a $n\times n$ matrix associate to the following linear system (we denote by $\lambda_{j}$ the unknowns): 
\begin{equation*}
\sum_{j=1}^{n-1}\lambda_{j}\cdot E_{i,1}\partial_{1j}(f)+\lambda_{n}\cdot E_{i,1}^{(2)}(f)
=\sum_{j=1}^{n-1}\lambda_{j}\cdot\left((2a-b)\,x_{ij}+b\,\sum_{r=1}^{n}x_{ir}\right)+\lambda_{n}\cdot 2a\,\sum_{j=1}^{n}x_{ij}.
\end{equation*}
We can easily see that
\begin{equation*}
M(a,b,n):=T_{n}(2a,b,b,2a,2a),
\end{equation*}
so we can show that
\begin{equation*}
\mathrm{det}\left(M(a,b,n)\right)=\mathrm{det}\left(T_{n}(2a,b,b,2a,2a)\right)=2a\,(2a-b)^{n-1}.
\end{equation*}
For the second property, we compute the determinant of the matrix \[ N(a,b,n)^{t}N(a,b,n),\] where $N(a,b,n)$ is the $n\times (n-1)$ matrix below 
\begin{equation*}
N(a,b,n):=H_{n}(2a,b,b).
\end{equation*}
in other words, $N(a,b,n)$ is the matrix obtained by dropping the last column of the matrix $M(a,b,n)$.
Now we prove that this matrix has full rank when $b\neq 2a$. Indeed, if we consider the matrix $N^{t}N$ we find that 
\begin{align*} 
&\mathrm{det}\big(N(a,b,n)^{t}N(a,b,n)\big)=
\mathrm{det}\big(H_{n}(2a,b,b)^{t}H_{n}(2a,b,b)\big)\\
&=(2a-b)^{2(n-2)}\big((2a+(n-2)b)^2+(n-1)b^2\big). 
\end{align*}
Finally, since $n\geq 2$ we can assert that 
\[ (2a-b)^{2(n-2)}\big((2a+(n-2)b)^2+(n-1)b^2\big)=0 \quad \text{\rm if and only if} \quad b=2a.\]
\end{proof}
The Lemma below describes a linear basis of $\mathcal{M}_f$ when $f$ is an homogeneous symmetric polynomial of degree 2.
\begin{lemma}\label{Lema23}
Let $f({\bold x}_1)$ be an homogeneous symmetric polynomial of degree 2, in $n$ variables $x_{11},\ldots,x_{1n}$. Then, a linear basis of $\mathcal{M}_f$ is described by the following:
\begin{itemize}
\item[(i)]\ Suppose that $a\neq 0$,
\begin{enumerate}
\item \textbf{Case 1:} If\ $b=2a$,\ then \ $\displaystyle{\partial_{1i}(f)=E_{i,1}^{(2)}(f)=2a\cdot\sum_{r=1}^{n}x_{ir}}$ \ and thus we have the following basis for each homogeneous component of $\mathcal{M}_f$
\begin{equation*}
\mathcal{B}_{\bold 0}:=\{1\},
\end{equation*}
\begin{equation*}
\mathcal{B}_{{\bold e}_{i}}:=\left\{\sum_{r=1}^{n}x_{ir}\right\},\ \ \forall\,i\ \text{such that}\ 1\leq i\leq \ell,
\end{equation*}
\begin{align*}
\mathcal{B}_{{\bold e}_{s+t}}&:=\big\{E_{s,1}E_{t,1}(f)\big\}=\left\{\sum_{p,q}x_{sp}x_{tq}\right\},\ \ \forall\,(s,t)\ \text{such that}\ 1\leq s\leq\ell,\ 1\leq t\leq\ell.
\end{align*}
\item \textbf{Case 2:} If $a\neq 0$ and $b\neq 2a$, then the following set
\begin{equation*}
\bigg\{E_{i,1}\partial_{1j}(f)\,\bigg\vert\,\ 1\leq j\leq n-1\bigg\}\bigcup\left\{E_{i1}^{(2)}(f)\right\}
\end{equation*} 
is linearly independent, therefore one basis is as follows
\begin{equation*}
\mathcal{B}_{\bold 0}:=\{1\},
\end{equation*}
and for all $i$ such that $1\leq i\leq \ell$ we define
\begin{align*}
\mathcal{B}_{{\bold e}_{i}}&:=\bigg\{E_{i,1}\partial_{1j}(f)\,\bigg\vert\,\ 1\leq j\leq n-1\bigg\}\bigcup\left\{E_{i1}^{(2)}(f)\right\}\\
&=\left\{(2a-b)\cdot x_{ij}+b.\sum_{r=1}^{n}x_{ir}\ \middle\vert\ 1\leq j\leq n-1\right\}\bigcup\left\{2a\cdot\sum_{r=1}^{n}x_{ir}\right\}.
\end{align*}
Also, we can choose another basis for $\mathcal{V}_{{\bold e}_i}$, as follows:
\begin{align*}
\mathcal{B}_{{\bold e}_i}&:=\left\{x_{ij}\,:\,1\leq j\leq n\right\},
\end{align*}
\begin{align*}
\mathcal{B}_{{\bold e}_{s+t}}&:=\big\{E_{s,1}E_{t,1}(f)\big\},\ \ \forall\,(s,t)\ \text{such that} \ 1\leq s\leq\ell,\ 1\leq t\leq\ell,\\
&=\left\{(2a-b)\,\sum_{j=1}^{n}x_{sj}x_{tj}+b\,\sum_{p,q}x_{sp}x_{tq}\right\}.
\end{align*}
\end{enumerate}
\item[(ii)] \ If\ $a=0$\ then\ $b\neq 0$\ and\ $f=b\cdot m_{1,1}({\bold x}_1)=b\cdot e_{2}({\bold x}_1)$. In this case the basis is given by
$\mathcal{B}_{\bold 0}:=\{1\}$,
\begin{align*}
\mathcal{B}_{{\bold e}_{i}}&:=\bigg\{E_{i,1}\partial_{1j}(f)\,\bigg\vert\,\ 1\leq j\leq n\bigg\}
=\left\{\sum_{r\neq j}^{n}x_{ir}\ \middle\vert\ 1\leq j\leq n\right\},\ \ \forall\,i,\ 1\leq i\leq \ell,
\end{align*}
\begin{align*}
\mathcal{B}_{{\bold e}_{s,t}}&:=\big\{E_{s,1}E_{t,1}(f)\big\},\ \ 
=\left\{\sum_{p\neq q}x_{sp}x_{tq}\right\},\forall\,(s,t),\ 1\leq s\leq\ell,\ 1\leq t\leq\ell.\\
\end{align*} 
\end{itemize}
\end{lemma}
\begin{proof}
The proof follows directly from Lemma \ref{LemaBaseGrado2} and third part of Theorem \ref{ProposicionUNO}.
\end{proof}
\begin{lemma}\label{ElLemasVentiSeis}
Let $f({\bold x}_1)$ be an homogeneous symmetric polynomial of degree 2, in $n$ variables ${\bold x}=x_{11},x_{12},\ldots,x_{1n}$. The graded $\mathfrak{S}_n$-character of $\mathcal{M}_f$ is given by
\begin{enumerate}
\item[(1)] ${\chisota}_{\mathcal{V}_{{\bold 0}}}(\sigma)=1$,
\item[(2)] ${\chisota}_{\mathcal{V}_{{\bold e}_i}}(\sigma)
=\begin{cases}
\big\vert\mathrm{Fix}(\sigma)\big\vert & \text{if} \ b\neq 2a,\\
1 & \text{otherwise}. 
\end{cases}$,
\item[(3)] ${\chisota}_{\mathcal{V}_{{\bold e}_{s,t}}}(\sigma)=1$.
\end{enumerate}
\end{lemma}
\begin{proof} The first assertion is evident. Assertion (3) is immediate since the basis $\mathcal{B}_{{\bold e}_{s,t}}$ contains just one diagonally symmetric polynomial, so 
\begin{equation*}
{\chisota}_{\mathcal{V}_{{\bold e}_{s,t}}}(\sigma)=1.
\end{equation*}
For assertion (2) we have two cases:
\begin{enumerate}
\item[(i)] If $b=2a$, the basis $\mathcal{B}_{{\bold e}_{i}}$ has only one diagonally symmetric polynomial and therefore,
\begin{equation*}
{\chisota}_{\mathcal{V}_{{\bold e}_i}}(\sigma)=1,\ \text{if}\ b=2a.
\end{equation*}
\item[(ii)] If $b\neq 2a$ the  the basis $\mathcal{B}_{{\bold e}_{i}}$ is given by
\begin{equation*}
\mathcal{B}_{{\bold e}_{i}}=\bigg\{E_{i,1}\partial_{1j}(f)\,\bigg\vert\,\ 1\leq j\leq n-1\bigg\}\bigcup\left\{E_{i1}^{(2)}(f)\right\}.
\end{equation*}
This basis has $n$ polynomials, in this case we get
\begin{align*}
{\chisota}_{\mathcal{V}_{{\bold e}_{i}}}(\sigma)
&=\sum_{g\in\mathcal{B}_{{\bold e}_{i}}}(\sigma\cdot g)(X)\bigg\vert_{g}=1+\sum_{j=1}^{n-1}\left\{\begin{array}{cc} 1 
& \text{if} \ \sigma(j)=j\ \,\text{and}\ \,\sigma(j)<n,\ \\ 0 
& \text{if} \  \sigma(j)\neq j\ \,\text{and}\,\ \sigma(j)<n,\\ -1 \ &  \text{if} \  \sigma(j)=n.\end{array}\right.\\
&=\big\vert{\rm Fix}(\sigma)\big\vert.
\end{align*}
\end{enumerate}
\end{proof}

\subsection*{Computing $\mathcal{M}_{f}({\bold q},{\bold w})$ when $f$ has degree 2}

Let $f({\bold x}_1)$ be an homogeneous symmetric polynomial of degree 2, in $n$ variables ${\bold x}=x_{11},x_{12},\ldots,x_{1n}$. Suppose that $f({\bold x}_1)$ is written in the monomial basis as follows: 
\[ f({\bold x}_1)=a\cdot m_{2}({\bold x}_1)+b\cdot m_{1,1}({\bold x}_1). \] 
Then the graded Frobenius characteristic of $\mathcal{M}_{f}$ is given by
\begin{equation*}
\mathcal{M}_{f}({\bold{q}},{\bold{w}})=
\left\{
\begin{array}{cc}
\big(1+s_1({\bold q})+s_2({\bold q})\big)\cdot s_{n}({\bold{w}}) &  {\rm if} \ b=2a, \\
\\
\big(1+s_1({\bold q})+s_2({\bold q})\big)\cdot s_{n}({\bold{w}})+s_1({\bold q})\cdot s_{n-1,1}({\bold{w}})
& {\rm otherwise}.
\end{array}
\right.
\end{equation*}

Indeed, the case $b=2a$ is a particular case of Corollary \ref{201504231304pm}. So we can suppose that 
$b\neq 2a$. By Lemma \ref{Lema23} we know that an homogeneous basis of  $\mathcal{M}_{f}$ is
\begin{equation*}
\mathcal{B}_{{\bold e}_{i}}
=\left\{2a\,\sum_{k=1}^{n}x_{ik}\right\}
\bigcup\left\{(2a-b)x_{ij}+b.\sum_{r=1}^{n}x_{ir}\middle\vert\ \ 1\leq j\leq n-1  \right\},
\end{equation*}
so we obtain
\begin{equation*}
{\chisota}_{\mathcal{V}_{{\bold e}_{i}}}(\sigma)
=\big\vert{\rm Fix}(\sigma)\big\vert={\chisota}^{(n)}(\sigma)+{\chisota}^{(n-1,1)}(\sigma).
\end{equation*}
Also, for each pair of integers\ $(r,s)$\ with $1\leq r,s\leq \ell$, a basis for the homogeneous component of degree  
${\bold e}_{r,s}$ is given by
\begin{equation*}
\mathcal{B}_{{\bold e}_{r,s}}
=\left\{2a\,\sum_{j=1}^{n}x_{rj}x_{sj}+b\,\sum_{1\leq p<q\leq n}x_{rp}x_{sq}\right\}.
\end{equation*}
So, for the homogeneous component ${\bold e}_{r,s}$ we have
\begin{equation*}
{\chisotazo}_{\mathcal{V}_{{\bold e}_{r,s}}}(\sigma)
=1={\chisota}^{(n)}(\sigma).
\end{equation*}
Finally, we compute the graded Frobenius series as follows: 
\begin{align*}
&\mathcal{M}_{f}({\bold q},{\bold w})
=\frac{1}{n!}\sum_{\sigma\in\frak{S}_n}\left(\sum_{{\bold d}\in\mathbb{N}^{\ell}}{\chisota}_{{\mathcal{V}_{\bold d}}}(\sigma)\,{\bold q}^{\bold d}\right)\,p_{\lambda(\sigma)}({\bold w}).
\end{align*}
The graded character is
\begin{align*}
&\sum_{{\bold d}\in\mathbb{N}^{\ell}}{\chisota}_{{\mathcal{V}_{\bold d}}}(\sigma)\,{\bold q}^{\bold d}
=1+\sum_{\big\vert{\bold d}\big\vert=1}{\chisota}_{{\mathcal{V}_{\bold d}}}(\sigma)\,{\bold q}^{\bold d}+\sum_{\big\vert{\bold d}\big\vert=2}{\chisota}_{{\mathcal{V}_{\bold d}}}(\sigma)\,{\bold q}^{\bold d}\\
&=1+\sum_{\big\vert{\bold d}\big\vert=1}\left({\chisota}^{(n)}(\sigma)+{\chisota}^{(n-1,1)}(\sigma)\right)\,{\bold q}^{\bold d}+\sum_{\big\vert{\bold d}\big\vert=2}{\chisota}^{(n)}(\sigma)\,{\bold q}^{\bold d},
\end{align*}
and so 
\begin{align*}
\mathcal{M}_{f}({\bold q},{\bold w})&=s_{n}({\bold w})+s_{1}({\bold q})\cdot\big(s_{n}({\bold w})
+s_{n-1,1}({\bold w})\big)+s_{2}({\bold q})\cdot s_{n}({\bold w})\\
&=\big(1+s_1({\bold q})+s_2({\bold q})\big)\cdot s_{n}({\bold w})+s_1({\bold q})\cdot s_{n-1,1}({\bold w}).
\end{align*}

\begin{corollary}
Suppose $f$ is an homogeneous symmetric polynomial of degree 2. Then the Hilbert series of $\mathcal{M}_{f}$ is given by:
\begin{equation*}
\mathcal{M}_{f}({\bold{q}})
=\left\{\begin{array}{cc} 1+s_{1}({\bold q})+s_2({\bold q}), & {\rm if}\ b=2a,\\
\\
1+n\cdot s_1({\bold q})+s_2({\bold q}), & {\rm otherwise.} \end{array}\right.
\end{equation*}
\end{corollary}
Therefore, the dimension of $\mathcal{M}_f$ is given by the formula:
\begin{equation*}
\dim\left(\mathcal{M}_{f}\right)
=
\begin{cases} 
\displaystyle{1+\ell+{\ell+1\choose 2}}={\ell+2\choose 2} & \text{if}\ b=2a, \\
&\\
\displaystyle{1+n\ell+{\ell+1\choose 2}} & \text{if}\ b\neq 2a.
\end{cases}.
\end{equation*}


\section{Proof of Theorem \ref{ProposicionTRES}}

In this case, we compute a linear basis of $\mathcal{M}_f$ when $f$ is any homogeneous symmetric polynomial of degree 3. We start with such a polynomial $f({\bold x}_1)$, and writing $f$ (over $\mathbb{R}$) as a linear combination of the monomial basis:
\begin{equation}\label{EcuacionPolinomioSimetricoGrado3BaseMonomial201601111155}
f({\bold x}_1)=a\, m_{3}({\bold x}_1)+b\,m_{2,1}({\bold x}_1)+c\,m_{1,1,1}({\bold x}_1),  
\end{equation}
\noindent Polynomials in $\mathcal{M}_f$ have the form:
\begin{equation*}
f(x_{11},\dots,x_{1n})=a\,\sum_{j=1}^{n}x_{1j}^3+b\,\sum_{p\neq q}x_{1p}^2x_{1q}
+c\,\sum_{1\leq i<j<k\leq n}x_{1i}x_{1j}x_{1k},
\end{equation*}
\begin{align*}
&\partial_{1j}(f)=3a\, x_{1j}^{2}+b\,\sum_{p\neq j}x_{1p}^2+2b\, x_{1j}\,\sum_{q\neq j}x_{1q}
+c\,\sum_{\alpha<\beta,\ \alpha\neq j,\ \beta\neq j}x_{1\alpha}x_{1\beta}\\
&=(3a-b)\,x_{1j}^{2}+b\,\sum_{s=1}^{n}x_{1s}^2+2b\,x_{1j}\,\sum_{r=1}^{n}x_{1r}
+c\,\sum_{\alpha<\beta,\ \alpha\neq j,\ \beta\neq j}x_{1\alpha}x_{1\beta}\\
&=(3a-b)\,x_{1j}^{2}+b\,\sum_{s=1}^{n}x_{1s}^2+(2b-c)\,x_{1j}\,\sum_{r=1}^{n}x_{1r}
+c\,\sum_{\alpha<\beta}x_{1\alpha}x_{1\beta}.
\end{align*}
\begin{equation*}
\partial_{1j}^{2}(f)=2(3a-b)x_{1j}+2b\sum_{s=1}^{n}x_{1s},
\end{equation*}
\begin{align*}
\partial_{1,r}\partial_{1,s}(f)&=2b\,(x_{1r}+x_{1s})+c\,\sum_{j:\ j\neq r,\ j\neq s}x_{1j}\\
&=(2b-c)\,(x_{1r}+x_{1s})+c\,\sum_{j=1}^{n}x_{1j}, \ \text{with}\ r\neq s,
\end{align*}
\begin{equation*}
E_{1,1}^{(2)}(f)=6a\,\sum_{t=1}^{n}x_{1t}^2+{4}\,b\,\sum_{\alpha<\beta}x_{1\alpha}x_{1\beta},
\end{equation*}
For all $i>1$ we have
\begin{equation*}
E_{i,1}(f)=3a\,\sum_{j=1}^{n}x_{1j}^2x_{ij}+2b\,\sum_{p\neq q}x_{1p}x_{1q}x_{ip}
+b\,\sum_{p\neq q}x_{1p}^2x_{iq}
+c\,\sum_{\underset{\boldsymbol{(\gamma-\alpha)(\gamma-\beta)\neq 0}}{\alpha<\beta}}x_{1\alpha}x_{1\beta}x_{i\gamma}.
\end{equation*}
\begin{equation*}
E_{i,1}^{2}(f)=6a\sum_{j=1}^{n}x_{1j}x_{ij}^{2}+2b\sum_{r\neq s}x_{1r}x_{is}^{2}+4b\sum_{r< s}x_{1r}x_{ir}x_{is}
+2c\sum_{\underset{\boldsymbol{(r-s)(r-t)\neq 0}}{s<t}}x_{1r}x_{is}x_{it}.
\end{equation*}
\begin{equation*}
E_{i,1}^{(2)}(f)=6a\,\sum_{j=1}^{n}x_{1j}x_{ij}+2b\,\sum_{\alpha\neq\beta}x_{1\alpha}x_{i\beta},
\end{equation*}
\begin{equation*}
E_{i,1}^{(3)}(f)=6a\,\sum_{s=1}^{n}x_{is},
\end{equation*}
Let $g(x,y,z)=(x-y)(y-z)(x-z)$, then $$\displaystyle{\sum_{p\neq q,\ p\neq j,\ q\neq j}x_{\alpha p}x_{\beta q}=\sum_{(p,q)\,:\,g(p,q,j)\neq 0}x_{\alpha p}x_{\beta q}}.$$ 
Now we can write, when $\alpha\neq\beta$
\begin{align*}
&E_{\alpha,1}E_{\beta,1}\partial_{1j}(f)\\
&=2(3a-b)\sum_{j=1}^{n}x_{\alpha j}x_{\beta j}
+2b\,\left(x_{\beta j}\sum_{r\neq j}x_{\alpha {r}}+x_{\alpha j}\sum_{t\neq j}x_{\beta t}\right)
+c\,\sum_{(p,q)\,:\,g(p,q,j)\neq 0}x_{\alpha p}x_{\beta q}.
\end{align*}
\[ E_{\alpha,1}^{2}\partial_{1,j}(f)=\partial_{\alpha,j}f({\bold x}_{\alpha}). \]
\begin{equation*}
E_{i,1}\partial_{1,j}^{2}(f)=2(3a-b)\,x_{ij}+2b\,\sum_{t=1}^{n}x_{it},
\end{equation*}
\begin{equation*}
E_{i,1}\partial_{1,r}\partial_{1s}(f)=2b\,(x_{ir}+x_{is})+c\,\sum_{p:\ p\neq r,\ p\neq s}x_{ip},
\end{equation*}
\begin{equation*}
E_{s,1}E_{t,1}E_{1,1}^{(2)}(f)
=12a\,\sum_{j=1}^{n}x_{sj}x_{tj}+4b\,\sum_{\alpha\neq\beta}x_{s\alpha}x_{t\beta}, \ \text{with} \ s\neq t.
\end{equation*}
\[  E_{s,1}^{(2)}E_{1,1}^{(2)}(f)
=12a\,\sum_{j=1}^{n}x_{sj}x_{tj}+8b\,\sum_{\alpha<\beta}x_{r\alpha}x_{t\beta}.  \]
If $p,q,r>1$ and $g(p,q,r)\neq 0$ then we get 
\begin{equation*}
E_{p,1}E_{q,1}E_{r,1}(f)=6a\,\sum_{j=1}^{n}x_{pj}x_{qj}x_{rj}+2b\,\sum_{s\neq t}x_{ps}x_{qs}x_{rt}
+c\,\sum_{g(i,j,k)\neq 0}x_{pi}x_{qj}x_{rk}.
\end{equation*}
If $i\neq k$ and $i,k>1$ then we get
\begin{equation*}
E_{i,1}^{2}E_{k,1}(f)=6a\sum_{j=1}^{n}x_{kj}x_{ij}^{2}+2b\sum_{r\neq s}x_{kr}x_{is}^{2}+4b\sum_{r< s}x_{kr}x_{ir}x_{is}
+2c\sum_{\underset{\boldsymbol{(r-s)(r-t)\neq 0}}{s<t}}x_{kr}x_{is}x_{it}.
\end{equation*}
\begin{equation*}
E_{p,1}^{3}(f)=6a\,\sum_{j=1}^{n}x_{pj}^3+6b\,\sum_{s\neq t}x_{ps}^{2}x_{pt}
+6c\,\sum_{i<j<k}x_{pi}x_{pj}x_{pk}=6\cdot f({\bold x}_{p}).
\end{equation*}
\subsubsection{Construction of degree 3 homogeneous components}
In the following, we explain how to extract a basis of each homogeneous component of $\mathcal{M}_{f}$. 
The homogeneous components of $\mathcal{M}_f$ are $\mathcal{V}_{\bf 0}:=\mathbb{R}$,
for each $i$ such that $1\leq i\leq \ell$ we define $\mathcal{V}_{{\bf e}_{i}}$ to be the real span of the set below
\begin{align*}
\{ E_{i,1}\partial_{1,j}^{2}(f):\ 1\leq j\leq n\}\ \bigcup\ \{ 
E_{i,1}\partial_{1,r}\partial_{1,s}(f):\ 1\leq r<s\leq n\}\bigcup\left\{E_{i,1}^{(3)}f\right\},
\end{align*}
for each pair $(s,t)$ such that $1\leq s\leq t \leq \ell$ we set
\begin{equation*}
\mathcal{V}_{{\bf e}_{s,t}}:=\mathbb{R}\left\{\{E_{s,1}E_{t,1}\partial_{1,j}(f):\ 1\leq j\leq n\}\bigcup
\{E_{s,1}E_{t,1}E_{1,1}^{(2)}(f)\}\right\}, 
\end{equation*}
for each triple $(p,q,r)$ such that $1\leq p\leq q\leq r \leq \ell$ we set
\begin{equation*}
\mathcal{V}_{{\bf e}_{p,q,r}}:=\mathbb{R}\big\{E_{p,1}E_{q,1}E_{r,1}(f)\big\}.
\end{equation*}
\begin{remark}
Recall that polarizations operators $E_{p,1}$,$E_{q,1}$ and $E_{r,1}$ commute (see formula \ref{201601151427}).
\end{remark}
\begin{prop}
Let $f\in\mathcal{R}_{n}$ be an homogenenous symmetric polynomial of degree 3 written in the monomial basis as formula \ref{EcuacionPolinomioSimetricoGrado3BaseMonomial201601111155} 
Then the vector space 
\[ \mathbb{R}\boldsymbol{\oplus}\bigoplus_{1\leq i\leq\ell}\mathcal{V}_{{\bf e}_{i}}\boldsymbol{\oplus}\bigoplus_{1\leq s,t\leq\ell}\mathcal{V}_{{\bf e}_{s,t}}\,\boldsymbol{\oplus}\bigoplus_{1\leq p,q,r\leq \ell}\mathcal{V}_{{\bf e}_{p,q,r}} \]
is closed under derivatives and polarization operators $E_{i,k}^{(p)}$. Furthermore,
\[ \mathcal{M}_f=\mathbb{R}\boldsymbol{\oplus}\bigoplus_{1\leq i\leq\ell}\mathcal{V}_{{\bf e}_{i}}\boldsymbol{\oplus}\bigoplus_{1\leq s,t\leq\ell}\mathcal{V}_{{\bf e}_{s,t}}\,\boldsymbol{\oplus}\bigoplus_{1\leq p,q,r\leq \ell}\mathcal{V}_{{\bf e}_{p,q,r}}. \]
\end{prop}
\begin{proof}
We start with degree 3. First, we see that if $(p-q)(q-r)(p-r)\neq 0$ then:
\begin{align*} 
& E_{\alpha,\beta}\bigg(E_{p,1}E_{q,1}E_{r,1}(f)\bigg) \\
&=\delta_{\beta p}\,E_{\alpha,1}E_{q,1}E_{r,1}(f)+\delta_{\beta q}\,E_{p,1}E_{\alpha,1}E_{r,1}(f)+\delta_{\beta r}\,E_{p,1}E_{q,1}E_{\alpha,1}(f). 
\end{align*}
Hence 
\[  E_{\alpha,\beta}\bigg(E_{p,1}E_{q,1}E_{r,1}(f)\bigg)
\in\,\delta_{\beta p}\,\mathcal{V}_{{\bf e}_{\alpha,q,r}}
+\delta_{\beta q}\,\mathcal{V}_{{\bf e}_{p,\alpha,r}}
+\delta_{\beta r}\,\mathcal{V}_{{\bf e}_{p,q,\alpha}} \]
and also we can see that
\[ E_{\alpha,k}\bigg(E_{i,1}^{2}E_{k,1}(f)\bigg)= E_{i,1}^{2}E_{\alpha,1}(f)\in\mathcal{V}_{{\bf e}_{i,i,\alpha}}. \]
\[ E_{\alpha,i}\bigg(E_{i,1}^{2}E_{k,1}(f)\bigg)=2\,E_{\alpha,1}E_{k,1}E_{i,1}(f)\in\mathcal{V}_{{\bf e}_{\alpha,k,i}}. \]
\[ E_{\alpha,i}\bigg(E_{i,1}^{3}(f)\bigg)=E_{\alpha,i}\big(6\,f({\bold x}_i)\big)=6\,E_{\alpha,i}\big(f({\bold x}_i)\big)
=3\,E_{i,1}^{2}E_{\alpha,1}(f)\in\mathcal{V}_{{\bf e}_{i,i,\alpha}}. \]
\[  E_{\alpha,\beta}^{(p)}\bigg(E_{p,1}E_{q,1}E_{r,1}(f)\bigg)=0,\ \forall\, p>1. \]
\[E_{r,s}^{(2)}\big(E_{s,1}^{2}E_{t,1}(f)\big)=E_{t,1}E_{r,1}E_{1,1}^{(2)}(f)\in\mathcal{V}_{{\bf e}_{t,r}}.\]
For the homogeneous components of degree 2, we proceed as in subsection \ref{seccion201601131318} in the proof of Proposition \ref{201601111410}. For instance, it's not difficult to see that:
\[ E_{\alpha,\beta}\bigg(E_{s,1}E_{t,1}E_{1,1}^{(2)}(f)\bigg)=\delta_{\beta,s}E_{\alpha,1}E_{t,1}E_{1,1}^{(2)}(f)
+\delta_{\beta,t}E_{\alpha,1}E_{s,1}E_{1,1}^{(2)}(f). \]
\[ E_{\alpha,\beta}\bigg(E_{s,1}E_{t,1}\partial_{1,j}(f)\bigg)=\delta_{\beta,s}E_{\beta,1}E_{t,1}\partial_{1,j}(f)+\delta_{\beta,t}E_{s,1}E_{\beta,1}\partial_{1,j}(f). \]
\[ E_{\alpha,\beta}\bigg(E_{s,1}^{2}E_{1,1}^{(2)}(f)\bigg)=2\delta_{\beta,s}\,E_{\alpha,1}E_{s,1}E_{1,1}^{(2)}(f). \]
\[ E_{\alpha,\beta}\bigg(E_{s,1}^{2}\partial_{1,j}(f)\bigg)=\delta_{\beta,s}\,E_{\alpha,1}E_{s,1}\partial_{1,j}(f). \]
for $p>1$ we have:
\[  E_{\alpha,\beta}^{(p)}\bigg(E_{s,1}E_{t,1}E_{1,1}^{(2)}(f)\bigg)=E_{\alpha,\beta}^{(p)}\bigg(E_{s,1}E_{t,1}\partial_{1,j}(f)\bigg)=0. \]
It is not difficult to see that the homogeneous components are closed under partial derivatives. For instance, when $(p-q)(p-r)(r-q)\neq 0$ we have
\[ \partial_{s,j}{E_{p,1}E_{q,1}E_{r,1}(f)}=\delta_{s,p}E_{q,1}E_{r,1}\partial_{1,j}(f)+\delta_{s,q}E_{p,1}E_{r,1}\partial_{1,j}(f)+\delta_{s,r}E_{p,1}E_{q,1}\partial_{1,j}(f). \]
and when $p=q$, $r>1$ and $q\neq r$ then
\[  \partial_{s,j} E_{q,1}^{2}E_{r,1}(f)=2\delta_{s,q}E_{q,1}^{2}E_{r,1}\partial_{1,j}(f)=\delta_{s,q}E_{1,1}E_{r,1}\partial_{q,j}f({\bold x}_{q}). \]
\[  \partial_{r,j} E_{q,1}^{2}E_{r,1}(f)=E_{q,1}^{2}\partial_{1,j}(f).  \]
The case degree 1 follows easily applying polarization operators $E_{\alpha,\beta}$. For the last assertion we use the same arguments as in the proof of Proposition \ref{201601111410}.
\end{proof}
\label{ConstruccionDeUnaBaseHomogeneaGrado3}
\subsection*{Constructing an homogeneous basis of $\mathcal{M}_{f}$ when the degree of $f$ is 3}

In the Lemma below we study linearly independent subsets of $\mathcal{M}_{f}$. This will lead us to a complete rule to find a homogeneous basis of the polarization module generated by a given homogeneous symmetric polynomial of degree 3. As before we suppose that $f$ has been identified with $[f]:=[a:b:c]$ (its corresponding point in $\mathbb{RP}^2$). 

\begin{lemma}\label{LemaDeLI}
Let $n\geq 3$ and $f=a\cdot m_{3}({\bold x}_1)+b\cdot m_{21}({\bold x}_1)+c\cdot m_{111}({\bold x}_1)$. We have the following properties:
\begin{enumerate}
\item If $b\neq 3a$ then, for each $i$ such that $1\leq i\leq \ell$, the set of partial derivatives
\begin{equation*}
\big\{E_{i,1}\partial_{1,j}^{2}f\,:\,1\leq j\leq n-1\big\}
\end{equation*}
is linearly independent.\\
\item If $a\neq 0$ and $b\neq 3a$ then, for each $i$ such that $1\leq i\leq \ell$, the set 
\begin{equation*}
\big\{E_{i,1}\partial_{1,j}^{2}f\,:\, 1\leq j\leq n-1\big\}\bigcup\big\{E_{i,1}^{(3)}(f)\big\}
\end{equation*}
is linearly independent.\\
\item If $a\neq 0$,\ $b=3a$ and $c\neq 6a$ then, for each $i$ such that $1\leq i\leq \ell$, the set
\begin{equation*}
\{E_{i,1}\partial_{1j}\partial_{11}f\,:\,1\leq j\leq n\}
\end{equation*}
is linearly independent. \\
\item If $a=0$ and $b\neq 0$ then, for each $i$ such that $1\leq i\leq \ell$, the set  
\begin{equation*}
\big\{E_{i,1}\partial_{1j}^{2}f\,:\,1\leq j\leq n\big\}
\end{equation*}
is linearly independent. \\
\item If $c\neq 2b$ then, for each $i$ such that $1\leq i\leq \ell$, the set
\begin{equation*}
\big\{E_{i,1}\partial_{1,j}\partial_{1,1}f\,:\,2\leq j\leq n\big\}
\end{equation*}
is linearly independent if $b^2+c^2>0$. \\
\item If $b\neq 0$ and $c\neq 2b$ then, for each $i$ such that $1\leq i\leq \ell$, the set
\begin{equation*}
\big\{E_{i,1}\partial_{1,j}\partial_{1,1}f\,:\,1\leq j\leq n\big\}
\end{equation*}
is linearly independent. \\
\item If $\big[a:b:c\big]\neq[1:3:6]$ \textbf{is not a $n$-exception} then, the following set 
\begin{equation*}
\big\{E_{\alpha,1}E_{\beta,1}\partial_{1,j}f\,:\,1\leq j\leq n\big\}
\cup\{E_{\alpha,1}E_{\beta,1}E_{1,1}^{(2)}(f)\}
\end{equation*}
is linearly independent for every pair $(\alpha,\beta)$ such that $1\leq \alpha,\beta\leq \ell$.\\
\item If $[a:b:c]$ is a \textbf{$n$-exception} then we have two cases: \\
\begin{enumerate}
\item If $R_{n}(a,b,c)\neq 0$ then the set
\[\big\{E_{\alpha,1}E_{\beta,1}\partial_{1,j}f\,:\,1\leq j\leq n\big\}\]
is linearly independent. \\
\item If $R_{n}(a,b,c)=0$ then the set
\[\big\{E_{\alpha,1}E_{\beta,1}\partial_{1,j}f\,:\,1\leq j\leq n-1\big\}
\cup\{E_{\alpha,1}E_{\beta,1}E_{1,1}^{(2)}(f)\}\]
is linearly independent. \\
\end{enumerate}
\end{enumerate}
\end{lemma}
\begin{proof}
We consider only assertions (7) and (8), because the proof of the other results is similar to the proof of Theorem \ref{TeoremaExcepcionGrado3} and Lemma \ref{LemaBaseGrado2}. We consider the associated matrix, let's say $M$, of each system of linear equations. Then we show that the matrix $M^{t}M$ is of the form $T_{n}$, we compute the determinant of  $M^{t}M$, and assertions (1)-(6) follow.  To prove assertion (7), we observe that the partial derivatives of an homogeneous polynomial of degree $d$ are all homogeneous polynomials of degree $d-1$. Also, the image of a homogeneous polynomial of degree $d$ by the operator $E_{1,1}^{(2)}$ is a homogeneous polynomial of degree $d-1$. Then we can use Lemma \ref{PolarizacionesEikInyectivas} to see that the linear combination
\[\sum_{j=1}^{n}\lambda_{j}E_{\alpha,1}E_{\beta,1}\partial_{1,j}(f)
+\mu\,E_{\alpha,1}E_{\beta,1}E_{1,1}^{(2)}(f)={\bold 0},\]
has only the trivial solution. To see this, we apply on both sides the operator $E_{1,\beta}E_{1,\alpha}$, and we get
\[ 4\cdot\sum_{j=1}^{n}\lambda_{j}\partial_{1,j}(f)
+4\cdot\mu\,E_{1,1}^{(2)}(f)={\bold 0}, \]
since Lemma \ref{PolarizacionesEikInyectivas} implies that
\begin{align*}
&E_{1,\beta}E_{1,\alpha}\left(E_{\alpha,1}E_{\beta,1}\partial_{1,j}(f)\right)=4\cdot f,\\
&E_{1,\beta}E_{1,\alpha}\left(E_{\alpha,1}E_{\beta,1}E_{1,1}^{(2)}(f)\right)=4\cdot E_{1,1}^{(2)}(f).
\end{align*}
Since $[a:b:c]$ is not a $n$-exception, Theorem \ref{TeoremaExcepcionGrado3} implies that the set
\[ \big\{\partial_{11}(f),\ldots,\partial_{1n}(f),E_{1,1}^{(2)}(f)\big\} \] 
is linearly independent. Then, $\mu=0$ and $\lambda_{j}=0$ for all $j$. Therefore, the set 
\[ \big\{E_{\alpha,1}E_{\beta,1}\partial_{1,j}f\,:\,1\leq j\leq n\big\} \cup\{E_{\alpha,1}E_{\beta,1}E_{1,1}^{(2)}(f)\} \] 
is also linearly independent. To prove assertion (8) we use again Lemma \ref{PolarizacionesEikInyectivas} and the last part of Theorem \ref{TeoremaExcepcionGrado3}. Finally, observe that the sets described in (1)-(6) are all formed by homogenenous polynomials of total degree 1. So it is easy to check directly that they are fixed by polarization operators of the form $E_{s,i}$, and for any $p>1$ they are killed by the operator $E_{s,i}^{(p)}$. So any homogeneous component of total degree 1 is mapped to another homogeneneous component of total degree 1.
\end{proof}
\begin{lemma}\label{EllemaVentisiete}
Let $f({\bold x}_1)$ be an homogeneous symmetric polynomial of degree 3 in $n$ variables ${\bold x}_1=x_{11},\ldots,x_{1n}$. An homogeneous linear basis of $\mathcal{M}_f$ is described by the following rule:
\begin{equation*}
\mathcal{B}_{\bold 0}:=\{1\},
\end{equation*}
\underline{\textbf{Case 1:}} \ If $\big[a:b:c\big]=[1:3:6]$ we use Lemma \ref{Lema5}. \\
\underline{\textbf{Case 2:}} \ If $\big[a:b:c\big]\neq[1:3:6]$ we have the cases below for a basis $\mathcal{B}_{{\bold e}_i}$ of $\mathcal{V}_{{\bold e}_{i}}$.
\begin{enumerate}
\item If $a=b=0$, then $c\neq 0$ and by Lemma \ref{Lema110035} the set $\{x_{i1},x_{i2},\ldots,x_{in}\}$ is a basis of $\mathcal{V}_{{\bold e}_{i}}$. 
\item If $a=0$ and $b\neq 0$, we use the basis
\begin{equation*}
\{E_{i,1}\partial_{1j}^2 f\,\vert\, 1\leq j\leq  n\}.
\end{equation*}
\item If $a\neq 0 $ and $b\neq 3a$, we use the basis 
\begin{align*}
\left\{E_{i,1}\partial_{1,j}^{2}(f)\,:\,\,1\leq j\leq n-1\middle\}\bigcup\middle\{E_{i,1}^{(3)}(f)\right\}.
\end{align*}
\item If $a\neq 0$, $b=3a$ and $c\neq 6a$, we use the basis  
\begin{equation*}
\{E_{i,1}\partial_{1j}\partial_{11}f\,:\,1\leq j\leq n\}.
\end{equation*}
\end{enumerate}
Notice that the homogeneous component $V_{{\bold e}_i}$ is closed by general polarization operators $E_{s,i}^{(p)}$. In fact, we can check that 
$E_{s,i}(\mathcal{B}_{{\bold e}_i})=\mathcal{B}_{{\bold e}_s}$ and $E_{s,i}^{(p)}(\mathcal{B}_{{\bold e}_i})=\{\bold 0\}$ for any $p>1$.\\

For a basis $\mathcal{B}_{{\bold e}_{s,t}}$ of $\mathcal{V}_{{\bold e}_{s,t}}$ we have two cases:
\begin{enumerate}
\item If $\big[a:b:c\big]$ is a $n$-exception we have two cases:\\
\begin{enumerate}
\item If $R_{n}(a,b,c)\neq 0$ we have the basis:
\begin{align*}
\left\{E_{s,1}E_{t,1}\partial_{1,j}(f)\,:\,\,1\leq j\leq n\right\},
\end{align*}
\item If $R_{n}(a,b,c)=0$ we have the basis:
\begin{align*}
\left\{E_{s,1}E_{t,1}\partial_{1,j}(f)\,:\,\,1\leq j\leq n-1\right\}
\cup\{E_{1,1}^{(2)}(f)\},
\end{align*}
\end{enumerate}
\item If $\big[a:b:c\big]$ is not an $n$-exception we have the basis:
\begin{align*}
\left\{E_{s,1}E_{t,1}\partial_{1,j}(f)\,:\,\,1\leq j\leq n\middle\}\bigcup\middle\{E_{s,1}E_{t,1}E_{1,1}^{(2)}(f)\right\}.
\end{align*}
\end{enumerate}
The two assertions above are justified by parts 7 and 8 of Lemma \ref{LemaDeLI}. Finally, a basis $\mathcal{B}_{{\bold e}_{p,q,r}}$ of 
$\mathcal{V}_{{\bold e}_{p,q,r}}$ is given by 
\begin{equation*}
\mathcal{B}_{{\bold e}_{p,q,r}}:=\left\{E_{p,1}E_{q,1}E_{r,1}(f)\right\}.
\end{equation*}
\end{lemma}

\subsection{Computing $\mathcal{M}_{f}({\bold q},{\bold w})$ when the degree of $f$ is 3}

\begin{lemma}\label{Lema2421082014}
Let $f$ be an homogeneous symmetric polynomial $f$ of degree 3, in $n$ variables ${\bold x}_1=x_{11},\ldots,x_{1n}$. The graded $\mathfrak{S}_n$-character of $\mathcal{M}_f$ is given by
\begin{enumerate}
\item ${\chisota}_{\mathcal{V}_{{\bold 0}}}(\sigma)=1,\ \forall\,\sigma\in\mathfrak{S}_n,$
\item ${\chisota}_{\mathcal{V}_{{\bold e}_i}}(\sigma)
=\begin{cases}
\big\vert\mathrm{Fix}(\sigma)\big\vert & \text{if}\ \big[a:b:c\big]\neq\big[1:3:6\big],\\
1 &  \text{otherwise.}
\end{cases},\ \forall\,\sigma.$
\item ${\chisota}_{\mathcal{V}_{{\bold e}_{s,t}}}(\sigma)
=\begin{cases}
1 & \text{if}\ \big[a:b:c\big]=\big[1:3:6\big],\\
\big\vert\mathrm{Fix}(\sigma)\big\vert & \text{if}\ \big[a:b:c\big]\ \text{is a $n$-exception,}\\
1+\big\vert\mathrm{Fix}(\sigma)\big\vert & \text{if}\ \big[a:b:c\big]\ \text{is not a $n$-exception.}
\end{cases}\ \forall\,\sigma.$
\item ${\chisota}_{\mathcal{V}_{{\bold e}_{p,q,r}}}(\sigma)=1, \ \ \forall\,\sigma.$
\end{enumerate}
\end{lemma}
\begin{proof}
We consider only assertions (2) and (3), because the others follow directly from the definition. To verify assertion (2), we first notice that the case $\big[a:b:c\big]=\big[1:3:6\big]$ is clear by Theorem \ref{ProposicionUNO}. So we can suppose that 
$\big[a:b:c\big]\neq\big[1:3:6\big]$. We have four cases:
\begin{enumerate}
\item[(i)] If $a\neq 0$ and $b\neq 3a$, we use the basis of $\mathcal{V}_{{\bold e}_i}$ from Lemma \ref{EllemaVentisiete}, 
that is,
\begin{equation*}
\mathcal{B}_{{\bold e}_i}=\left\{2(3a-b)\,x_{ij}+2b\,\sum_{t=1}^{n}x_{it}\,:\,1\leq j\leq n-1\middle\}\bigcup\middle\{6a\,\sum_{r=1}^{n}x_{ir}\right\}.
\end{equation*}
Then, as in the proof of Lemma \ref{ElLemasVentiSeis}, we have
\begin{align*}
{\chisota}_{\mathcal{V}_{{\bold e}_{i}}}(\sigma)
&=\sum_{g\in\mathcal{B}_{{\bold e}_{i}}}(\sigma\cdot g)(X)\bigg\vert_{g}\\
&=1+\sum_{j=1}^{n-1}\left\{\begin{array}{cc} 1 & \text{if} \ \sigma(j)=j,\ \,\text{and}\,\ \sigma(j)<n,\\ 0 & \text{if} \ \sigma(j)\neq j,\ \,\text{and}\,\ \sigma(j)<n,\\ -1\ & \sigma(j)=n\end{array}\right.=\big\vert{\rm Fix}(\sigma)\big\vert.
\end{align*} 
\item[(ii)] If $a=b=0$, then $c\neq 0$, and then we use the basis from Lemma \ref{Lema5}, that is, the basis of $\mathcal{V}_{{\bold e}_{i}}$ given by $\{x_1,x_2,\ldots,x_n\}$. So, we obtain
\[{\chisota}_{\mathcal{V}_{{\bold e}_{i}}}(\sigma)
=\sum_{j=1}^{n}x_{\sigma(j)}\bigg\vert_{x_j}=\big\vert{\rm Fix}(\sigma)\big\vert.\]
\item[(iii)] If $a=0$ and $b\neq 0$, then the basis $\mathcal{B}_{{\bold e}_i}$ from Lemma \ref{EllemaVentisiete} is given by
\begin{equation*}
\mathcal{B}_{{\bold e}_i}:=\{E_{i,1}\partial_{1j}^2 f\,\vert\, 1\leq j\leq  n\}=\{\partial_{ij}^{2}(f({\bold x}_i))\,\vert\, 1\leq j\leq n\}.
\end{equation*}
Since $f$ is symmetric and the operator $E_{i,1}$ is symmetric (see formula (\ref{201504201232c1})), polynomials in $\mathcal{B}_{{\bold e}_i}$ satisfy 
\[ \sigma\cdot \partial_{1j}^2 f=\partial_{1,\sigma(j)}^{2}(\sigma\cdot f)=\partial_{1,\sigma(j)}^{2}f. \] 
This implies that
\[ \sigma\cdot E_{i,1}\big(\partial_{1j}^2 f\big)=E_{i,2}\big(\sigma\cdot \partial_{1j}^2 f\big) =E_{i,1}\big(\partial_{1,\sigma(j)}^{2}f\big)\in\mathcal{B}_{{\bold e}_i}, \]
and therefore, the value of the character ${\chisota}_{\mathcal{V}_{{\bold e}_{i}}}$ is  ${\chisota}_{\mathcal{V}_{{\bold e}_{i}}}(\sigma)=\big\vert{\rm Fix}(\sigma)\big\vert$. This also follows from formula \ref{CaracterDeLaRepresentacionPorPermutacion}.
\item[(iv)] If $a\neq 0$ and $c\neq 6a$, then a basis $\mathcal{B}_{{\bold e}_i}$ for $\mathcal{V}_{{\bold e}_i}$ is for instance, the basis from Lemma \ref{EllemaVentisiete} given by
\begin{equation*}
\mathcal{B}_{{\bold e}_i}:=\big\{E_{i,1}\partial_{11}\partial_{1j}f\,:\,1\leq j\leq n\big\}=\{\partial_{i1}\partial_{ij}f({\bold x}_i)\,\vert\, 1\leq j\leq n\}.
\end{equation*}
By formula \ref{CaracterDeLaRepresentacionPorPermutacion}, we know that the value of the character is  ${\chisota}_{\mathcal{V}_{{\bold e}_{i}}}(\sigma)=\big\vert{\rm Fix}(\sigma)\big\vert.$
\end{enumerate}
To prove assertion (3), we use the basis $\mathcal{B}_{{\bold e}_{s,t}}$ of $\mathcal{V}_{{\bold e}_{s,t}}$ from Lemma \ref{EllemaVentisiete}, that is,
\begin{align*}
\mathcal{B}_{{\bold e}_{s,t}}&:=\left\{E_{s,1}E_{t,1}\partial_{1,j}(f)\,:\,\,1\leq j\leq n\middle\}\bigcup\middle\{E_{s,1}E_{t,1}E_{1,1}^{(2)}(f)\right\}.
\end{align*}
We observe two conditions coming from formulas \ref{201504201232c1}, \ref{201504201232c2} and the fact that $f$ is symmetric.
\begin{enumerate}
\item[(A1)] The polynomial $E_{s,1}E_{t,1}E_{1,1}^{(2)}(f)$ is diagonally symmetric,
\item[(A2)] For every permutation $\sigma\in\frak{S}_{n}$ we have
\begin{equation*}
\sigma\cdot E_{s,1}E_{t,1}\partial_{1,j}(f)
=E_{s,1}E_{t,1}\partial_{1,\sigma(j)}(f)\in\mathcal{B}_{{\bold e}_{s,t}},\ \ \text{since}\ \ \sigma\cdot f=f.
\end{equation*}
\end{enumerate}
Then the two assertions above imply that
\begin{equation*}
{\chisota}_{\mathcal{V}_{{\bold e}_{s,t}}}(\sigma)
=\begin{cases}
\big\vert\mathrm{Fix}(\sigma)\big\vert & \text{if}\ \big[a:b:c\big]\ \text{is a $n$-exception,}\\
1+\big\vert\mathrm{Fix}(\sigma)\big\vert & \text{if}\ \big[a:b:c\big]\ \text{is not a $n$-exception.}
\end{cases}\ \forall\,\sigma\in\mathfrak{S}_n.
\end{equation*}
Also, the last identity is a consequence of formula (\ref{CaracterDeLaRepresentacionPorPermutacion}).
\end{proof}
\subsection*{Summary of general results for degree 3}
We have shown that, for a given homogeneous symmetric polynomial $f$ of degree 3, in the $n$ variables ${\bold x}_1$, given in the monomial basis as 
\[ f({\bold x}_1)=a\cdot m_{3}({\bold x}_1)+b\cdot m_{21}({\bold x}_1)+c\cdot m_{111}({\bold x}_1), \  a,b,c\in\mathbb{R},  \]
the graded Frobenius characteristic of $\mathcal{M}_{f}$  is one of the three cases below: 
\begin{enumerate}
\item[(1)] If $[a:b:c]=[1:3:6]$, then
\begin{align*}
\mathcal{M}_{f}({\bold q},{\bold w},n)&=\big(1+s_1({\bold q})+s_2({\bold q})+s_3({\bold q})\,\big)\cdot s_{n}({\bold w})
=\mathcal{M}_{p_1^{3}}({\bold q},{\bold w},n).
\end{align*} 
\item[(2)] If $[a:b:c]$ \textbf{is a $n$-exception}, then
\begin{align*}
\mathcal{M}_{f}({\bold q},{\bold w},n)&=\big(1+s_1({\bold q})+s_2({\bold q})+s_3({\bold q})\,\big)\cdot s_{n}({\bold{w}})+\big(s_1({\bold q})+s_2({\bold q})\,\big)\cdot s_{n-1,1}({\bold{w}})\\
&=\mathcal{M}_{p_{3}}({\bold q},{\bold w},n).
\end{align*}
\item[(3)] If $\big[a:b:c\big]$ \textbf{is not a $n$-exception}, we get 
\begin{align*}\mathcal{M}_{f}({\bold q},{\bold w})&=\big(1+s_1({\bold q})+\boldsymbol{2}\,s_2({\bold q})+s_3({\bold q})\,\big)\cdot s_{n}({\bold{w}})+\big(s_1({\bold q})+s_2({\bold q})\,\big)\cdot s_{n-1,1}({\bold{w}})\\
=&\mathcal{M}_{h_{3}}({\bold q},{\bold w},n).\\
\end{align*}
\end{enumerate}

\begin{corollary}
The Hilbert series of $\mathcal{M}_f$ is given by only three possible cases:
\begin{enumerate}
\item[(1)] If $[a:b:c]=[1:3:6]$, then
\begin{align*}
\mathcal{M}_{f}({\bold q},n)=1+h_1({\bold{q}})+h_2({\bold{q}})+h_3({\bold{q}})=\mathcal{M}_{p_1^{3}}({\bold q},n).
\end{align*}
\item[(2)] If $[a:b:c]$ is an $n$-exception, then 
\begin{align*}
\mathcal{M}_{f}({\bold q},n)&=1+n\cdot h_1({\bold{q}})+n\cdot h_2({\bold{q}})+h_3({\bold{q}})
=\mathcal{M}_{p_{3}}({\bold q},n).
\end{align*}
\item[(3)] If $[a:b:c]$ is not a $n$-exception, then
\begin{align*}
\mathcal{M}_{f}({\bold q},n)=1+n\cdot h_1({\bold{q}})+(n+1)\cdot h_2({\bold{q}})+h_3({\bold{q}})=\mathcal{M}_{h_{3}}({\bold q},n).
\end{align*}
\end{enumerate}
\end{corollary}

So, in degree 3, the possible dimensions of $\mathcal{M}_f$ are (recall that, we have the identification $[f]=[a:b:c]$.)

\begin{equation*}
\dim\left(\mathcal{M}_{f}\right)
=
\begin{cases} 
\displaystyle{1+\ell+{\ell+1\choose 2}+{\ell+2\choose 3}}={\ell+3\choose 3} & \text{if}\ [f]=[1:3:6], \\
&\\
\displaystyle{1+n\,\ell+n{\ell+1\choose 2}+{\ell+2\choose 3}} & \text{if}\ [f]\ \text{is a}\ \text{$n$-exception},\\
&\\
\displaystyle{1+n\,\ell+(n+1){\ell+1\choose 2}+{\ell+2\choose 3}} & \text{if}\ [f]\ \text{is not}\ \text{$n$-exception.}\\
\end{cases}
\end{equation*}

Finally, if we write the last formulas in terms of the functions $h_{\lambda}({\bold w})$, we find that its coefficients are $h({\bold q})$-positive, we have shown then the formulas of Theorem \ref{ProposicionTRES}: 
\begin{enumerate}
\item[(1)] If $[a:b:c]=[1:3:6]$, then  
\[\mathcal{M}_{f}({\bold q},{\bold w},n)=\big(1+h_1({\bold q})+h_2({\bold q})+h_3({\bold q})\big)\cdot h_{n}({\bold{w}}).\]
\item[(2)] If  $[a:b:c]$ is a $n$-exception, then
\[\mathcal{M}_{f}({\bold q},{\bold w},n)=\big(1+h_3({\bold q})\big)\cdot h_{n}({\bold{w}})+\big(h_1({\bold q})+h_2({\bold q})\big)\cdot h_{n-1,1}({\bold{w}}).\]
\item[(3)] If  $[a:b:c]$ is not a $n$-exception, then
\[\mathcal{M}_{f}({\bold q},{\bold w},n)=\big(1+h_2({\bold q})+h_3({\bold q})\big)\cdot h_{n}({\bold{w}})+\big(h_1({\bold q})+h_2({\bold q})\big)\cdot h_{n-1,1}({\bold{w}}).\]
\end{enumerate}

\section{Final remarks and further research directions}

This work may serve as a starting point for a deeper understanding and a complete classification of polarizations modules generated by any symmetric polynomial. For instance, an explicit description of the form of the Frobenius characteristic for polarization modules generated by homogeneous symmetric polynomials of degree 4 and 5 are given in Section 6 of \cite{Blandin}. We suggest as an homogeneous basis for the polarization module $\mathcal{M}_{\mathcal{C}}$ appearing in \cite{Blandin} the following set of polynomials
\[ B_{\mathcal{C}}:=\left\{\sum_{f:A\longrightarrow[r]}\prod_{j\in A}x_{f(j),j}\,:\, 1\leq r\leq\ell, A\subseteq[n],\, \vert A\vert\leq d \right\}.\]
For each vector ${\bold d}\in\mathbb{N}^{\ell}$, a basis for the homogeneous component of multidegree ${\bold d}=(d_1,\ldots,d_{\ell})$ with $d_1+\cdots+d_{\ell}=r$ is the set
\[ B_{\mathcal{C},{\bold d}}:=\left\{\sum_{\{f:A\longrightarrow[r]\,:\,\vert f^{-1}(i)\vert=d_i.\}}\prod_{j\in A}x_{f(j),j}\,:\, 1\leq r\leq\ell,\, A\subseteq[n],\, \vert A\vert\leq d \right\}. \]

\subsection{Polarization modules for complex reflections groups}

\noindent Let $W$ be any rank $n$ complex reflection group. The elements of $W$ can be considered as $n\times n$ matrices. The diagonal action of $W$ on $\mathcal{R}_{n}^{(\ell)}$ is  $(w\cdot f)(X)=f(Xw),\,\forall\ w\in W$. A polynomial $f$ is said to be \textbf{$W$-invariant} if $(w\cdot f)(X)=f(X)$.
Following \cite{Haiman1, Hunziker} we observe that the spaces $\mathcal{M}_{F}$ can be turned into $W\times GL_{\ell}(\mathbb{C})$-modules  in the following manner: Let $f_{1}({\bold x}),\dots,f_{n}({\bold x})$ be a set of basic invariants for $W$. We set  $f_{j}({\bold x}_i):=f_{j}(x_{i1},\dots,x_{in})$, for all $i$ such that $1\leq i\leq \ell$, and all $j$ with $1\leq j\leq n$. With this approach we consider in a more general context the \textbf{$W$-invariant polarization operators $E_{i,k;p}^{W}$} (see \cite{Hunziker,Haiman1}) given by
\begin{equation*}
E_{i,k;p}^{W}
:=\sum_{j=1}^{n}x_{ij}\cdot\frac{\partial f_{p}}{\partial x_{kj}}\left(\frac{\partial\ }{\partial{x_{{k1}}}},\ldots,\frac{\partial\ }{\partial{x_{{kn}}}}\right).
\end{equation*}
For a $W$-stable family of homogeneous polynomials, we can define $\mathcal{M}_{F}$ as the smallest $\mathbb{C}$-vector space closed by derivation and taking $W$-invariant polarization operators and containing the family $F$.

{\small
\begin{table}
\centering
\caption{Hilbert Series for degree 4}
\label{table:HilbertGrado4}
\begin{tabular}{|c|c|}
\hline
\multirow{3}{*}{$1+s_1+s_2+s_3+s_4$} & \\ &$e_{1}^{4}$\\ &  \\
\hline
\multirow{3}{*}{$1+n\cdot s_1+n\cdot s_2+n\cdot s_3+s_4$} 
& \\ & $p_{4}$ \\ &
\\
\hline
\multirow{3}{*}{$1+n\cdot s_1+{n\choose 2}\cdot s_2+n\cdot s_3+s_4$} & 
\\ &$e_4$ \\ &  \\
\hline
\multirow{3}{*}{$1+n\cdot s_1+2n\cdot s_2+(n+1)\cdot s_3+s_4$} 
& \\ & $e_{31}$ \\ & \\
\hline
\multirow{3}{*}{$1+n\cdot s_1+{n+1 \choose 2}\cdot s_2+(n+1)\cdot s_3+s_4$} & 
\\ & {$s_{211},h_{22},m_{211}$}\\ & \\
\hline
\multirow{3}{*}{$1+n\cdot s_1+(n+1)\cdot s_2+(n-1)\cdot s_{11}+(n+1)\cdot s_{3}+s_{21}+s_4$} & \\ & $p_{211},e_{211},h_{211}$ \\ & \\
\hline
\multirow{3}{*}{$1+n\cdot s_1+2n\cdot s_2+(n-1)\cdot s_{11}+(n+1)\cdot s_{3}+s_{21}+s_4$} & \\ & $h_{31},m_{31},p_{31}$\\ & \\
\hline
\multirow{3}{*}{$1+n\cdot s_1+{n+1\choose 2}\cdot s_2+(n-1)\cdot s_{11}+n\cdot s_3+s_{21}+s_4$} & \\ & $m_{22}$\\  & \\
\hline
\multirow{3}{*}{$1+n\cdot s_1+{n+1 \choose 2}\cdot s_2+(n-1)\cdot s_{11}+(n+1)\cdot s_{3}+s_{21}+s_4$} & \\ & $s_4,s_{31},s_{2,2},e_{22},p_{22}$\\ & \\
\hline
\end{tabular}
\end{table}}

{\small
\begin{table}
\centering
\caption{Degree 4 Hilbert series in terms of $h({\bold q})$ functions}
\label{table:HilbertGrado4CompletasHomogeneas}
\begin{tabular}{|c|c|}
\hline
\multirow{3}{*}{$1+h_1+h_2+h_3+h_4$} & \\ & $e_{1}^{4}$ \\ & \\
\hline
\multirow{3}{*}{$1+n\cdot h_1+n\cdot h_2+n\cdot h_3+h_4$} & \\ &
$p_{4}$ \\ & \\
\hline
\multirow{3}{*}{$1+n\cdot h_1+{n\choose 2}\cdot h_2+n\cdot h_3+h_4$}& \\ & $e_4$ \\ & \\
\hline
\multirow{3}{*}{$1+n\cdot h_1+2n\cdot h_2+(n+1)\cdot h_3+h_4$} & \\ & $e_{31}$ \\ & \\
\hline
\multirow{3}{*}{$1+n\cdot h_1+{n+1 \choose 2}\cdot h_2+(n+1)\cdot h_3+h_4$} & \\ & $s_{211},h_{22},m_{211}$ \\ & \\
\hline
\multirow{3}{*}{$1+n\cdot h_1+(n-1)\cdot h_1^2+2\cdot h_2+h_{21}+n\cdot h_3+h_4$} & \\ &
$p_{211},e_{211},h_{211}$ \\ & \\
\hline
\multirow{3}{*}{$1+n\cdot h_1+(n+1)\cdot h_2+(n-1)h_1^{2}+h_{21}+n\cdot h_3+h_4$} & \\ & 
$h_{31},m_{31},p_{31}$ \\ & \\
\hline
\multirow{3}{*}{$1+n\cdot h_1+(n-1)\cdot h_1^{2}+\left({n\choose 2}+1\right)\cdot h_2+(n-1)\cdot h_3+h_4$} & \\ & $m_{22}$ \\ & \\
\hline
\multirow{3}{*}{$1+n\cdot h_1+(n-1)\cdot h_1^2+\left({n\choose2}+1\right)h_2+n\cdot h_3+h_4$} & \\ & 
$s_{_4},s_{_{31}},s_{_{22}},e_{_{22}},p_{_{22}}$ \\ & \\
\hline
\end{tabular}
\end{table}}

\newpage


{\footnotesize
\begin{table}[ht!]
\centering
\caption{Hilbert Series for degree 5}
\label{table:HilbertGrado5}
\begin{tabular}{|c|c|}
\hline
\multirow{3}{*}{$1+s_1+s_2+s_3+s_4+s_5$}
&\\
&$e_{1}^5$\\
&\\
\hline
\multirow{3}{*}{$1+n\cdot s_1+n\cdot s_2+n\cdot s_3+n\cdot s_4+s_5$} 
&\\ 
& $p_{5}$ \\
& \\ 
\hline
\multirow{3}{*}{$1+n\cdot s_1+{n\choose 2}\cdot s_2+{n\choose 2}\cdot s_3+n\cdot s_4+s_5$} 
&  \\ 
& $e_5$\\ 
&  \\ 
\hline
\multirow{3}{*}{$1+n\cdot s_1+{n+1 \choose 2}\cdot s_2+{n+1 \choose 2}\cdot s_3+(n+1)\cdot s_4+s_5$}    
& $m_{_{2111}},$ \\
& $s_{_{2111}},$\\ 
&$e_{41}$\\
\hline
\multirow{3}{*}{$1+n s_1+{n+1 \choose 2} s_2+(n-1)s_{11}+{n+1 \choose 2} s_3+ns_{21}+(n+1) s_4+s_5$} 
&  \\ 
& $s_{221},$ \\ 
& \\
\hline
\multirow{2}{*}{$1+n\cdot s_1+2n\cdot s_2+(n-1)\cdot s_{11}+2n\cdot s_3+n\cdot s_{21}
+(n+1)\cdot s_4+s_{31}+s_5$ } 
& $m_{41},$ \\ 
& $p_{41}$\\
\hline
\multirow{9}{*}{ 
$1+n\, s_1+{n+1\choose 2}\, s_2+(n-1)\, s_{11}+\frac{n(n+3)}{2}\, s_3+n\, s_{21}
+(n+1)\, s_4+s_{31}+s_5$}  
&$h_5,$ \\ & $h_{41},$\\
&$h_{32},$ \\ & $h_{221},$\\ 
&$p_{221},$ \\ & $s_{41},$\\ 
&$s_{32},$ \\ & $s_{311},$\\
&$e_{221},$ \\ & $m_{311}$\\
\hline
\multirow{4}{*}{$1+n\, s_1+{n+1\choose 2}\, s_2+(n-1)\, s_{11}+\left(\frac{n(n+3)}{2}-1\right)s_3+n\, s_{21}+(n+1)\, s_4+s_{31}+s_5$}  
&$p_{32},$ \\ & $e_{32},$ \\ 
&$m_{32},$ \\ & $m_{221}$\\
\hline
\multirow{3}{*}{$1+n\,s_1+(n+1)s_2+(n-1)s_{11}+(n+1)s_3+n\,s_{21}+(n+1)s_4+s_{31}+s_5$}
& $p_{2111},$\\ 
& $h_{2111},$\\ 
& $e_{2111}$\\
\hline
\multirow{3}{*}{$1+n\,s_1+2n\,s_2+(n-1)\,s_{11}+(2n+1)\,s_3+n\,s_{21}+(n+1)\,s_4+s_{31}+s_5$}
&$e_{311},$ \\
&$h_{311},$ \\
&$p_{311}$ \\
\hline
\end{tabular}
\end{table}}

{\footnotesize
\begin{table}
\centering
\caption{Degree 5 Hilbert series in terms of $h({\bold q})$ functions}
\label{table:HilbertGrado5CompletasHomogeneas}
\begin{tabular}{|c|c|}
\hline
\multirow{3}{*}{$1+h_1+h_2+h_3+h_4+h_5$}
& \\
& $e_{1}^5$ \\
&\\
\hline
\multirow{3}{*}{$1+n\cdot h_1+n\cdot h_2+n\cdot h_3+n\cdot h_4+h_5$} 
&\\
& $p_{5}$ \\
& \\ 
\hline
\multirow{3}{*}{
$1+n\cdot h_1+{n\choose 2}\cdot h_2+{n\choose 2}\cdot h_3+n\cdot h_4
+h_5$} 
& \\ 
& $e_5$ \\ 
&  \\ 
\hline
\multirow{3}{*}{$1+n\cdot h_1+{n+1 \choose 2}\cdot h_2+{n+1 \choose 2}\cdot h_3+(n+1)\cdot h_4+h_5$}    
& $m_{_{2111}},$ \\
& $s_{_{2111}},$\\ 
&$e_{41}$\\
\hline
\multirow{3}{*}
{$1+nh_1+\left(1+{n\choose 2}\right)h_2+(n-1)h_{1}^{2}+{n\choose 2}h_3+n\,h_{21}+n\,h_ {4}+h_1h_3+h_{5}$} 
&\\ 
& $s_{221}$  \\ 
&\\
\hline
\multirow{3}{*}
{$1+n\,h_1+(n+1)\,h_2+(n-1)h_1^2+n\,h_2h_1+n\,h_3+h_1h_3+n\,h_4+h_5$} 
&$s_{221},$ \\ 
& $m_{41},$ \\ 
& $p_{41}$\\
\hline
\multirow{9}{*}{$1+n\,h_1+(n-1)\,h_1^2+\left(1+{n\choose 2}\right)\,h_2+n\,h_2h_1+{n+1\choose 2}h_3+n\,h_4+h_1h_3+h_5$}  
&$h_5,$ \\ & $h_{41},$\\
&$h_{32},$ \\ & $h_{221},$\\ 
&$p_{221},$ \\ & $s_{41},$\\ 
&$s_{32},$ \\ & $s_{311},$\\
&$e_{221},$ \\ & $m_{311}$\\
\hline
\multirow{4}{*}{$1+n\,h_1+(n-1)\,h_1^2+\left(1+{n\choose 2}\right)\,h_2+n\,h_2h_1+\left({n+1\choose 2}+1\right)h_3+n\,h_4+h_1h_3+h_5$}  
&$p_{32},$ \\ & $e_{32},$ \\ 
&$m_{32},$ \\ & $m_{221}$\\
\hline
\multirow{3}{*}{$1+n\,h_1+(n-1)\,h_1^2+2\,h_2+n\,h_2h_1+h_3+n\,h_4+h_1h_3+h_5$}
& $p_{2111},$\\ 
& $h_{2111},$\\ 
& $e_{2111}$\\
\hline
\multirow{3}{*}{$1+n\,h_1+(n-1)\,h_1^2+2\,h_2+n\,h_2h_1
+(n+1)h_3+n\,h_4+h_1h_3+h_5$}
&$e_{311},$ \\
&$h_{311},$ \\
&$p_{311}$ \\
\hline
\end{tabular}
\end{table}}


\begin{table}
\centering
\caption{Frobenius characteristic for the space $\mathcal{M}_{h_{m}}$}
\label{table:FrobeniusCompletasHomogeneas}
{\tiny
\begin{tabular}{|c|c|}
\hline
\multirow{3}{*}{$h_1$}
& \\ & $\big(1+s_{1}\big)s_{n}({\bold w})$,\,$\forall n\geq 1$\\
& \\
\hline
\multirow{5}{*}{$h_2$}
& \\
& $\big(1+s_{1}+s_{2}\big)s_{1}({\bold w})$,\ \, {if}\ $n=1$,\\
&\\
& $\big(1+s_{1}+s_{2}\big)s_{n}({\bold w})+s_{1}s_{n-1,1}({\bold w})$,\ $\forall n\geq 2$\\
& \\
\hline
\multirow{5}{*}{$h_3$}
&\\
& $(1+s_1+s_2+s_3)s_{1}({\bold w})$, \, if \ $n=1$,\\
&\\
& $\big(1+s_{1}+\boldsymbol{2}\,s_{2}+s_{3}\big)s_{n}({\bold w})+\big(s_{1}+s_{2}\big)s_{n-1,1}({\bold w})$,\ \, $\forall n\geq 2$\\
&\\
\hline
\multirow{8}{*}{$h_4$}
&\\
& $(1+s_1+s_2+s_3+s_{4})s_{1}({\bold w})$, \, if\ $n=1$,\\
&\\
&$(1+s_1+\boldsymbol{2} s_2+\boldsymbol{2} s_3+s_{21}+s_4)\,s_{2}({\bold w})
+(s_1+\boldsymbol{2} s_2+s_{11}+s_3)s_{1,1}({\bold w})$, \ if $n=2$,\\
&\\
& $(1+s_1+\boldsymbol{2} s_2+\boldsymbol{2} s_3+s_{21}+s_4)\,s_{n}({\bold w})
+(s_1+\boldsymbol{2} s_2+s_{11}+s_3)s_{n-1,1}({\bold w})$\\
&$+s_2 s_{n-2,2}({\bold w})$,\ \, $\forall n\geq 2$\\
&\\
\hline
\multirow{10}{*}{$h_5$}
&\\
& $(1+s_1+s_2+s_3+s_{4}+s_5)s_{1}({\bold w})$, \,  $n=1$,\\
&\\
& $(1+s_1+\boldsymbol{2}s_2+\boldsymbol{2}s_2+s_{21}+\boldsymbol{2}s_4+s_{31}+s_5)s_2({\bold w})
+(s_1+\boldsymbol{2}s_2+s_{11}+\boldsymbol{2}s_3+s_{21}+s_4)s_{1,1}({\bold w})$,\  $n=2$\\
&\\
&$(1+s_1+\boldsymbol{2}s_2+\boldsymbol{3}s_3+s_{21}+\boldsymbol{2}s_4+s_{31}+s_5)s_3({\bold w})
+(s_1+\boldsymbol{2}s_2+s_{11}+\boldsymbol{3}s_3+s_{21}+s_4)s_{2,1}({\bold w})$,\  $n=3$ \\
&\\
& $(1+s_1+\boldsymbol{2}s_2+\boldsymbol{3}s_3+s_{21}+\boldsymbol{2}s_4+s_{31}+s_5)s_n({\bold w})
+(s_1+\boldsymbol{2}s_2+s_{11}+\boldsymbol{3}s_3+s_{21}+s_4)s_{n-1,1}({\bold w})$\\
&$+(s_2+s_3)s_{n-2,2}({\bold w})$,\ \, $\forall n\geq 4$\\
&\\
\hline
\multirow{7}{*}{$h_6$}
&\\
& $(1+s_1+s_2+s_3+s_{4}+s_5+s_6)s_{1}({\bold w})$, \, if\ $n=1$,\\
&\\
& $\begin{array}{c}(1+s_1+\boldsymbol{2}s_2+\boldsymbol{2}s_3+s_{21}+\boldsymbol{3s}_4+s_{31}+s_{22}+\boldsymbol{2}s_{5}+s_{41}+s_{6})s_{2}({\bold w})\\
+(s_1+s_2+s_{11}+\boldsymbol{2}s_3+s_{21}+s_4+s_{31}+s_5)s_{11}({\bold w})\end{array}$, $n=2$\\
&\\
& $\begin{array}{c}(1+s_1+\boldsymbol{2}s_2+\boldsymbol{3}s_3+s_{21}+\boldsymbol{4}s_4+\boldsymbol{2}s_{31}+s_{22}+\boldsymbol{2}s_{5}
+s_{41}+s_{6})s_{3}({\bold w})
\\
+(s_1+\boldsymbol{2}s_2+s_{11}+\boldsymbol{3}s_3+\boldsymbol{3}s_{21}
+\boldsymbol{3}s_4+s_{31}+s_5)s_{2,1}({\bold w})
+(s_{11}+s_{3})s_{1,1,1}({\bold w})
\end{array}$,\ $n=3$\\
&\\
&$\begin{array}{c}
(1+s_1+\boldsymbol{2}s_2+\boldsymbol{3}s_3+s_{21}+\boldsymbol{4}s_4+\boldsymbol{2}s_{31}+s_{22}+\boldsymbol{2}s_5+s_{41}+s_6)s_{n}({\bold w})\\
+(s_1+\boldsymbol{2}s_2+s_{11}+\boldsymbol{4}s_3+\boldsymbol{3}s_{21}+\boldsymbol{3}s_{4}+s_{31}+s_5)s_{n-1,1}({\bold w})\\
+(s_2+s_3+s_{21}+s_4)s_{n-2,2}({\bold w})\\
+(s_{11}+s_3)s_{n-2,1,1}({\bold w}).
\end{array}$\ $\forall\, n\geq 4$.\\
\hline
\end{tabular}}
\end{table}

\clearpage

\subsection*{Acknowledgment}

Thanks to Fran\c cois Bergeron and Franco Saliola for their advice and support during the preparation of my Ph.D. thesis.  Thanks to Franco Saliola and Eduardo Blazek for good suggestions and many independent computations with the computer algebra system SAGE and helpful discussions. Thanks to Adolfo Rodr\'iguez for  help with Maple. Thanks to my colleagues Yannic Vargas, Alejandro Morales, Marco P\'erez, Amy Pang and Luis Serrano for many valuable suggestions to reach the final version of this paper. To my fiancee Oana-Andreea Kosztan for love and support during my Ph.D. thesis and help with \LaTeX. Thanks to Christian Remling for valuable help towards to a generalization of Theorem \ref{TeoremaExcepcionGrado3} to any degree $d\geq 4$.

\bibliographystyle{plain}
\bibliography{RevisedVersionHectorBlandin}


\end{document}